\documentclass{amsart}
\usepackage{graphicx, amssymb, amsfonts, amsmath}
\usepackage{amsmath}
\usepackage{amsfonts}
\usepackage{amssymb}
\usepackage{graphicx}
\providecommand{\U}[1]{\protect\rule{.1in}{.1in}}
\newtheorem{theorem}{Theorem}
\theoremstyle{plain}

\newtheorem{conjecture}{Conjecture}
\newtheorem{corollary}{Corollary}

\newtheorem{definition}{Definition}
\newtheorem{example}{Example}

\newtheorem{proposition}{Proposition}
\newtheorem{remark}{Remark}

\numberwithin{equation}{section}
\begin{document}
\title{Morphisms of $A_{\infty}$-Bialgebras and Applications }
\author{Samson Saneblidze$^{1}$}
\email{sane@rmi.ge}
\address{A. Razmadze Mathematical Institute\\
Department of Geometry and Topology \\
M. Aleksidze St. 1 \\
0193 Tbilisi, Georgia}
\author{Ronald Umble$^{2}$}
\email{ron.umble@millersville.edu}
\address{Department of Mathematics\\
Millersville University of Pennsylvania\\
Millersville, PA. 17551}
\thanks{$^{1} $ The research described in this publication was made possible in part
by Award No. GM1-2083 of the U.S. Civilian Research and Development Foundation
for the Independent States of the Former Soviet Union (CRDF) and by Award No.
99-00817 of INTAS}
\thanks{$^{2}$ This research was funded in part by a Millersville University faculty
research grant.}
%\classification{Primary 55P35, 55P99 ; Secondary 52B05}

\begin{abstract}
We define the notion of a \emph{relative matrad and }realize the free relative
matrad $r\mathcal{H}_{\infty}$ as a free $\mathcal{H}_{\infty}$-bimodule
structure on cellular chains of \emph{bimultiplihedra} $JJ=\{JJ_{n,m}
=JJ_{m,n}\}_{m,n\geq1}.$ We define a \emph{morphism} $G:A\Rightarrow B$ of
$A_{\infty}$-bialgebras as a bimodule over $\mathcal{H}_{\infty}$. \ We  prove that the homology of every
$A_{\infty}$-bialgebra over a commutative ring with unity admits an induced
$A_{\infty}$-bialgebra structure. We extend the Bott-Samelson isomorphism to
an isomorphism of $A_{\infty}$-bialgebras and identify the $A_{\infty}
$-bialgebra structure of $H_{\ast}\left(  \Omega\Sigma X;\mathbb{Q}\right)  $.
For each $n\geq2$, we construct a space $X_{n}$ and
identify an induced nontrivial $A_{\infty}$-bialgebra operation $\omega_{2}^{n}:
H^{\ast}\left(  \Omega X_{n};\mathbb{Z}_{2}\right)  ^{\otimes2}\rightarrow H^{\ast
}\left(  \Omega X_{n};\mathbb{Z}_{2}\right)  ^{\otimes n}$.

\end{abstract}

%\received{June 24, 2011}
%\revised{	September 27, 2012}
\date{September 27, 2012}
%\published{???? ??, 2012}
%\submitted{Johannes Huebschmann}

%\volumeyear{2012} % Volume Year
%\volumenumber{???}  % Volume Number
%\issuenumber{???}   % Issue Number

%\startpage{1}     % PageNumber of first page

\maketitle

\section{Introduction}

This paper assumes familiarity with our prequels \cite{SU3} and \cite{SU4}, in which we defined the objects in the category of
$A_{\infty}$-bialgebras; in this paper we define the morphisms. Let $R$ be
a (graded or ungraded) commutative ring with unity, let $H$ be a graded
$R$-module, let $M=\left\{  M_{n,m}=\right.  $ $\left.  Hom\left(  H^{\otimes
m},H^{\otimes n}\right)  \right\}  $, and consider the (non-bilinear) product
\[
\circledcirc:M\times M\overset{d\times d}{\rightarrow}\mathbf{M\times
M}\overset{\Upsilon}{\mathbf{\rightarrow}}\mathbf{M}\overset{\operatorname*{proj}}{\mathbf{\rightarrow}}M,
\]
where $\mathbf{M}\subset TTM$ (the double tensor module of $M$) is the bisequence submodule, $d:M\rightarrow
\mathbf{M}$ is the biderivative, and $\Upsilon:\mathbf{M\times M\rightarrow
M}$ is the canonical product. A family of maps
\[
\omega=\left\{  \omega_{m}^{n}\in M_{n,m}\mid\left\vert \omega_{m}
^{n}\right\vert =m+n-3\right\}  _{m,n\geq1}\]
defines an $A_{\infty}$-bialgebra structure on $H$ whenever $\omega
\circledcirc\omega=0$; the structure relations are the bihomogeneous
components of this equation.  The structure of an $A_{\infty}$-bialgebra
is controlled by a family of contractible polytopes $KK=\left\{KK_{n,m}\right.$
$\left.=KK_{m,n}\right\} $ called biassociahedra, of which $KK_{n,m}$
is $\left(m+n-3\right)  $-dimensional with a single top dimensional cell; in particular,
$KK_{1,n}$ is the associahedron $K_{n}$. The cellular chains $C_{\ast}\left(
KK\right)$ realize the free matrad $\mathcal{H}_{\infty}$, and there is a
canonical matrad structure on $Hom\left(TH,TH\right)$. Thus $\left(  H,\omega\right)  $
is an $A_{\infty}$-bialgebra whenever there is a map of matrads
$C_{\ast}\left(  KK\right)  \rightarrow Hom\left(  TH,TH\right)  $
whose restriction to top dimensional cells has image $\omega$.
When this occurs we say that $\left(  H,\omega\right)  $ is an
algebra over $\mathcal{H}_{\infty}$.

The structure of an $A_{\infty}$-bialgebra morphism is controlled by a
family of contractible polytopes $JJ=\{JJ_{n,m}=JJ_{m,n}\}_{m,n\geq1}$
called \emph{bimultiplihedra,} of which $JJ_{n,m}$ is a subdivision of
$KK_{n,m}\times I$ with a single top dimensional cell;
in particular, $JJ_{1,n}$ is the multiplihedron $J_{n}.$ The $\mathcal{A}_{\infty}
$-bimodule structure on cellular chains $C_{\ast}\left(  J\right)  $ extends
to an $\mathcal{H}_{\infty}$-bimodule structure on $C_{\ast}\left(  JJ\right)
$ and realizes the free \emph{relative matrad} $r\mathcal{H}_{\infty}$. Given
$A_{\infty}$-bialgebras $A$ and $B,$ there is a canonical $\mathcal{E}nd_{TB}
$-$\mathcal{E}nd_{TA}$-bimodule structure on $Hom(TA,TB),$ which is also a
canonical relative matrad structure. Thus a family of maps \[G=\left\{
g_{m}^{n}\in Hom\left(  A^{\otimes m},B^{\otimes n}\right)
\mid \left\vert g_{m}^{n}\right\vert =m+n-2\right\} \] defines
an $A_{\infty}$-\emph{bialgebra morphism} $A\Rightarrow B$ whenever
there is a map of relative matrads $C_{\ast}\left(  JJ\right)  \rightarrow
Hom(TA,TB)$ whose restriction to top dimensional cells
has image $G$.  When this occurs we say that $\left(  A,B,G\right)$ is a
bimodule over $\mathcal{H}_{\infty}$.

Given an $A_{\infty}$-bialgebra $B$ whose homology $H_*(B)$ is a free $R$-module, and a homology isomorphism $g:H_{\ast}\left(  B\right)  \rightarrow B$, we prove that the $A_{\infty}$-bialgebra structure on $B$ pulls back along $g$ to an $A_{\infty}$-bialgebra structure on $H_{\ast}\left(  B\right)$, and any two such structures so obtained are isomorphic. This is a special case of our main result:\smallskip

\noindent\textbf{Theorem} \ref{AAHopf}. \textit{Let }$B$\textit{\ be an
}$A_{\infty}$\textit{-bialgebra with homology }$H=H_{\ast}\left(  B\right)
$\textit{, let }$\left(  RH,d\right)  $\textit{\ be a free }$R$\textit{-module
resolution of }$H,$ \textit{and let }$h$\textit{\ be} \textit{a perturbation
of }$d$\textit{\ such that }$g:\left(  RH,d+h\right)  \rightarrow\left(
B,d_{B}\right)  $\textit{\ is a homology isomorphism. Then}

\begin{enumerate}
\item[\textit{(i)}] (Existence) $g$\textit{\ induces an }$A_{\infty}
$\textit{-bialgebra structure }$\omega_{RH}$\textit{\ on }$RH$\textit{\ and
extends to a map }$G:RH\Rightarrow B$\textit{\ of }$A_{\infty}$
\textit{-bialgebras.}

\item[\textit{(ii)}] (Uniqueness) $\left(  \omega_{RH},G\right)  $\textit{\ is
unique up to isomorphism.}\smallskip
\end{enumerate}

Whereas an $A_\infty$-coalgebra is an $A_\infty$-bialgebra whose nontrivial operations are exclusively $A_\infty$-coalgebra operations, Theorem \ref{AAHopf} achieves the same result as the Coalgebra Perturbation Lemma (CPL) whenever it applies (see \cite{G-L-S2}, \cite{Hueb-Kade}, \cite{Markl1}).  However, Theorem \ref{AAHopf} applies to a general $A_\infty$-bialgebra $B$, and our proof, which follows a geometrical construction, does not assume that the map $g$ has a right-homotopy inverse (without which the CPL cannot be formulated).  Indeed, if $A$ is a free $R$-module, $B$ is an $A_\infty$-structure, and $g:A\to B$ is a homology isomorphism, Theorem \ref{AAHopf} transfers the $A_\infty$-structure from $B$ to $A$ along $g$ \emph{whenever} (1) the $A_\infty$-structure on $B$ is controlled by a family $\mathcal{P}$ of contractible polytopes, each having a single top dimensional cell, and (2) the morphisms of $A_\infty$-structures are controlled by an appropriate subdivision of $\mathcal{P}\times I.$   For some remarks on the history of perturbation theory see \cite{Huebschmann} and \cite{Hueb-Kade}. We only note that, in the particular case of the de Rham complex, K.T. Chen \cite{Chen} constructed the requisite perturbation by hand (also see \cite{G-L-S}).

Several applications of Theorem \ref{AAHopf} appear in Section 7. Given a topological space $X$ and a field $F,$ the
bialgebra structure of simplicial singular chains $S_{\ast}\left(  \Omega
X;F\right)  $ of Moore base pointed loops induces an $A_{\infty}$-bialgebra
structure on $H_{\ast}\left(  \Omega X;F\right)  $ and the $A_{\infty}
$-(co)algebra substructures are exactly the ones observed earlier by Gugenheim
\cite{Gugenheim} and Kadeishvili \cite{Kadeishvili1}.
 Furthermore, the
$A_{\infty}$-coalgebra structure on ${H}_{\ast}(X;F)$ extends to an
$A_{\infty}$-bialgebra structure on the tensor algebra $T^{a}\tilde{H}_{\ast
}(X;F)$, which is trivial if and only if the $A_{\infty}$-coalgebra structure
on ${H}_{\ast}(X;F)$ is trivial. The Bott-Samelson isomorphism
$t_{\ast}:T^{a}\tilde{H}_{\ast}(X;F)\overset{\approx}{\rightarrow}H_{\ast
}(\Omega\Sigma X;F)$ extends to an isomorphism of $A_{\infty}$-bialgebras
(Theorem \ref{B-S}). Indeed, the $A_{\infty}$-bialgebra structure of $H_{\ast}
(\Omega\Sigma X;\mathbb{Q})$ provides the first
nontrivial rational homology invariant for $\Omega\Sigma X$ (Corollary
\ref{invariant}). Finally, for each $n\geq2,$ we construct a space $X_{n}$ and
identify a nontrivial $A_{\infty}$-bialgebra operation $\omega_{2}^{n}
:H^{\ast}(\Omega X_{n};\mathbb{Z}_{2})^{\otimes2}\rightarrow H^{\ast}(\Omega
X_{n};\mathbb{Z}_{2})^{\otimes n}$ defined in terms of the action of the
Steenrod algebra $\mathcal{A}_{2}$ on $H^{\ast}(X_{n};\mathbb{Z}_{2}).$

The paper is organized as follows: Section 2 reviews the theory of matrads and
the related notation used throughout the paper; see \cite{SU4} for details. We
construct relative matrads in Section 3 and the bimultiplihedra in Section 4.
We define morphisms of $A_{\infty}$-bialgebras in Section 5. We prove our main
result (Theorem \ref{AAHopf}) in Section 6 and conclude with various applications
and examples in Section 7.

\section{Matrads}

Let $M=\left\{  M_{n,m}\right\}  _{m,n\geq1}$ be a bigraded module over a
(graded or ungraded) commutative ring $R$ with unity $1_{R}$. Each pair of
matrices $X=\left[  x_{ij}\right]  ,Y=\left[  y_{ij}\right]  \in
\mathbb{N}^{q\times p},$ $p,q\geq1,$ uniquely determines a submodule
\[
{M}_{Y,X}=\left(  M_{y_{11},x_{11}}\otimes\cdots\otimes M_{y_{1p},x_{1p}
}\right)  \otimes\cdots\otimes\left(  M_{y_{q1},x_{q1}}\otimes\cdots\otimes
M_{y_{qp},x_{qp}}\right)  \subset TTM.
\]
Fix a set of bihomogeneous module generators $G\subset{M}$. A \emph{monomial
in} $TM$ is an element of $G^{\otimes p}$, and a \emph{monomial in} $TTM$ is an
element of $\left(  G^{\otimes p}\right)  ^{\otimes q}.$ Thus $A\in\left(
G^{\otimes p}\right)  ^{\otimes q}$ is naturally represented by the $q\times
p$ matrix
\[
A=\left[
\begin{array}
[c]{ccc}
\alpha_{x_{1,1}}^{y_{1,1}} & \cdots & \alpha_{x_{1,p}}^{y_{1,p}}\\
\vdots &  & \vdots\\
\alpha_{x_{q,1}}^{y_{q,1}} & \cdots & \alpha_{x_{q,p}}^{y_{q,p}}
\end{array}
\right]
\]
with entries in $G$, and rows identified with elements of $G^{\otimes p}.$ We
refer to $A$ as a $q\times p$ \emph{monomial}, and to the submodules
\[
\overline{\mathbf{M}}=\bigoplus_{\substack{X,Y\in\mathbb{N}^{q\times
p}\\p,q\geq1}}M_{Y,X}\text{ \ and \ }\overline{\mathbf{V}}=\bigoplus
_{\substack{X,Y\in\mathbb{N}^{1\times p}\cup\mathbb{N}^{q\times1}\text{
}\\p,q\geq1}}M_{Y,X}
\]
as the \emph{matrix }and\emph{\ vector }submodules of\emph{\ }$TTM,$
respectively. The matrix transpose $A\mapsto A^{T}$ induces the canonical
permutation of tensor factors $\sigma_{p,q}:\left(  M^{\otimes p}\right)
^{\otimes q}\overset{\approx}{\rightarrow}\left(  M^{\otimes q}\right)
^{\otimes p}$ given by
\[
\left(  \alpha_{x_{1,1}}^{y_{1,1}}\otimes\cdots\otimes\alpha_{x_{1,p}
}^{y_{1,p}}\right)  \otimes\cdots\otimes\left(  \alpha_{x_{q,1}}^{y_{q,1}
}\otimes\cdots\otimes\alpha_{x_{q,p}}^{y_{q,p}}\right)  \mapsto
\]
\[
\left(  \alpha_{x_{1,1}}^{y_{1,1}}\otimes\cdots\otimes\alpha_{x_{q,1}
}^{y_{q,1}}\right)  \otimes\cdots\otimes\left(  \alpha_{x_{1,p}}^{y_{1,p}
}\otimes\cdots\otimes\alpha_{x_{q,p}}^{y_{q,p}}\right)  .
\]

Let $\bar{X} \in \mathbb{N}^{q\times p}$ be a matrix with
constant columns $\left(  x_{i}\right)  ^{T}$; let $\bar{Y} \in \mathbb{N}^{q\times p}$ be a matrix with constant rows $\left(
y_{j}\right)  $. Given $q\times p$ monomials
\[
A=\left[
\begin{array}
[c]{ccc}
\alpha_{x_{1}}^{y_{1,1}} & \cdots & \alpha_{x_{p}}^{y_{1,p}}\\
\vdots &  & \vdots\\
\alpha_{x_{1}}^{y_{q,1}} & \cdots & \alpha_{x_{p}}^{y_{q,p}}
\end{array}
\right]  \in M_{Y,\bar{X}}\text{ \ and \ }B=\left[
\begin{array}
[c]{ccc}
\beta_{x_{1,1}}^{y_{1}} & \cdots & \beta_{x_{1,p}}^{y_{1}}\\
\vdots &  & \vdots\\
\beta_{x_{q,1}}^{y_{q}} & \cdots & \beta_{x_{q,p}}^{y_{q}}
\end{array}
\right]  \in M_{\bar{Y},X},
\]
define the \emph{row leaf sequence} \emph{of} $A$ to be the $p$-tuple
${\operatorname*{rls}}\left(  A\right)  =\left(  x_{1},\ldots,x_{p}\right)  $;
define the \emph{column} \emph{leaf sequence} \emph{of} $B$ to be the $q$-tuple
${\operatorname*{cls}}\left(  B\right)  =\left(  y_{1},\ldots,y_{q}\right)
^{T}.$ Let
\[
\overline{\mathbf{M}}_{{\operatorname*{row}}}:=\bigoplus_{\mathbf{\bar{X}}
,Y\in\mathbb{N}^{q\times p};\text{ }p,q\geq1}M_{Y,\bar{X}}\text{ \ \ and
\ \ }\overline{\mathbf{M}}^{{\operatorname*{col}}}:=\bigoplus_{X,\bar{Y}
\in\mathbb{N}^{q\times p};\text{ }p,q\geq1}M_{\bar{Y},X};
\]
then
\[
\mathbf{M}=\overline{\mathbf{M}}_{{\operatorname*{row}}}\cap\overline
{\mathbf{M}}^{{\operatorname*{col}}}
\]
is the \emph{bisequence submodule }of $TTM,$ and a $q\times p$ monomial
$A\in\mathbf{M}$ is represented as a \emph{bisequence} \emph{matrix}
\[
A=\left[
\begin{array}
[c]{lll}
\alpha_{x_{1}}^{y_{1}} & \cdots & \alpha_{x_{p}}^{y_{1}}\\
\multicolumn{1}{c}{\vdots} & \multicolumn{1}{c}{} & \multicolumn{1}{c}{\vdots
}\\
\alpha_{x_{1}}^{y_{q}} & \cdots & \alpha_{x_{p}}^{y_{q}}
\end{array}
\right]  .
\]
Unless otherwise stated, we shall assume that $\mathbf{x\times y\in}\left(
x_{1},\ldots,x_{p}\right)  \times\left(  y_{1},\ldots,y_{q}\right)  ^{T}
\in\mathbb{N}^{1\times p}\times\mathbb{N}^{q\times1},$ $p,q\geq1$; we will
often express $\mathbf{y}$ as a row vector. Let
\[
\mathbf{M}_{\mathbf{x}}^{\mathbf{y}}=\left\langle A\in\mathbf{M}\mid
\mathbf{x}={\operatorname*{rls}}\left(  A\right)  \text{ and }\mathbf{y}
={\operatorname*{cls}}\left(  A\right)  \right\rangle ;
\]
then
\[
\mathbf{M}=\bigoplus_{\mathbf{x}\times\mathbf{y}}\mathbf{M}_{\mathbf{x}
}^{\mathbf{y}}.
\]

For example, given a graded $R$-module $H,$ set $M_{n,m}=Hom\left(  H^{\otimes
m},H^{\otimes n}\right)  $, and think of $\alpha_{m}^{n}\in M_{n,m}$ as a
composition of multilinear indecomposable operations $\theta_{i}
^{j}:H^{\otimes i}\rightarrow H^{\otimes j}$ pictured as an upward-directed
graph with $m$ inputs and $n$ outputs; then a $q\times p$ monomial
$A\in\mathbf{M}_{\mathbf{x}}^{\mathbf{y}}$ is a $q\times p$ matrix of such
graphs. By identifying $\left(  H^{\otimes q}\right)  ^{\otimes p}
\approx\left(  H^{\otimes p}\right)  ^{\otimes q}$ with $\left(  q,p\right)
\in\mathbb{N}^{2}$, we think of a $q\times p$ monomial $A\in\mathbf{M}
_{\mathbf{x}}^{\mathbf{y}}$ as an operator on the positive integer lattice
$\mathbb{N}^{2},$ pictured as an arrow $\left(  \left\vert \mathbf{x}
\right\vert ,q\right)  \mapsto\left(  p,\left\vert \mathbf{y}\right\vert
\right)  ,$ where $\left\vert \mathbf{u}\right\vert =\Sigma u_{i};$ but
unfortunately, this representation is not faithful.

The \emph{bisequence vector submodule} is the intersection\emph{\ }
\[
\mathbf{V}=\overline{\mathbf{V}}\cap\mathbf{M}=\bigoplus\limits_{s,t\in
\mathbb{N}}\mathbf{M}_{s}^{\mathbf{y}}\oplus\mathbf{M}_{\mathbf{x}}^{t}.
\]
A submodule
\[
\mathbf{W}=M\oplus\bigoplus\limits_{\mathbf{x},\mathbf{y}\notin\mathbb{N}
;\text{ }s,t\in\mathbb{N}}\mathbf{W}_{s}^{\mathbf{y}}\oplus\mathbf{W}
_{\mathbf{x}}^{t}\subseteq\mathbf{V}
\]
is \emph{telescoping }if for all $\mathbf{x},\mathbf{y},s,t$

\begin{enumerate}
\item[\textit{(i)}] $\mathbf{W}_{s}^{\mathbf{y}}\subseteq\mathbf{M}
_{s}^{\mathbf{y}}$ and $\mathbf{W}_{\mathbf{x}}^{t}\subseteq\mathbf{M}
_{\mathbf{x}}^{t};$

\item[\textit{(ii)}] $\alpha_{s}^{y_{1}}\otimes\cdots\otimes\alpha_{s}^{y_{q}
}\in\mathbf{W}_{s}^{\mathbf{y}}$ implies $\alpha_{s}^{y_{1}}\otimes
\cdots\otimes\alpha_{s}^{y_{j}}\in\mathbf{W}_{s}^{y_{1}\cdots y_{j}}$ for all
$j<q;$

\item[\textit{(iii)}] $\beta_{x_{1}}^{t}\otimes\cdots\otimes\beta_{x_{p}}
^{t}\in\mathbf{W}_{\mathbf{x}}^{t}$ implies $\beta_{x_{1}}^{t}\otimes
\cdots\otimes\beta_{x_{i}}^{t}\in\mathbf{W}_{x_{1}\cdots x_{i}}^{t}$ for all
$i<p.$
\end{enumerate}

\noindent Thus the truncation\emph{\ }maps $\tau:\mathbf{W}_{s}^{y_{1}\cdots
y_{j}}\rightarrow\mathbf{W}_{s}^{y_{1}\cdots y_{j-1}}$ and $\tau
:\mathbf{W}_{x_{1}\cdots x_{i}}^{t}\rightarrow\mathbf{W}_{x_{1}\cdots x_{i-1}
}^{t}$ determine the following \textquotedblleft telescoping\textquotedblright
\ sequences of submodules:
\begin{align*}
\tau\left(  \mathbf{W}_{s}^{\mathbf{y}}\right)   &  \subseteq\tau^{2}\left(
\mathbf{W}_{s}^{\mathbf{y}}\right)  \subseteq\cdots\subseteq\tau^{q-1}\left(
\mathbf{W}_{s}^{\mathbf{y}}\right)  =\mathbf{W}_{s}^{y_{1}}\\
\tau\left(  \mathbf{W}_{\mathbf{x}}^{t}\right)   &  \subseteq\tau^{2}\left(
\mathbf{W}_{\mathbf{x}}^{t}\right)  \subseteq\cdots\subseteq\tau^{p-1}\left(
\mathbf{W}_{\mathbf{x}}^{t}\right)  =\mathbf{W}_{x_{1}}^{t}.
\end{align*}
In general, $\mathbf{W}_{\mathbf{x}}^{t}$ is an \emph{additive }submodule of
$\mathbf{M}_{x_{1}}^{t}\otimes\cdots\otimes\mathbf{M}_{x_{p}}^{t}$, which does
\emph{not} necessarily decompose as $B_{1}\otimes\cdots\otimes B_{p}$ with
$B_{i}\subseteq\mathbf{M}_{x_{i}}^{t}$. The \emph{telescopic extension}\ of a
telescoping submodule $\mathbf{W\subseteq V}$ is the submodule of matrices
$\mathcal{W\subseteq}\overline{\mathbf{M}}$ with the following properties: If
$A=\left[  \alpha_{x_{i,j}}^{y_{i,j}}\right]  \ $is a $q\times p$ monomial in
$\mathcal{W}$, and

\begin{enumerate}
\item[(\textit{i)}] $\left[  \alpha_{x_{i,j}}^{y_{i,j}}\text{ }\cdots\text{
}\alpha_{x_{i,j+m}}^{y_{i,j+m}}\right]  $ is a string in the $i^{th}$ row of
$A$ such that $y_{i,j}=\cdots=y_{i,j+m}=t,$ then $\alpha_{x_{i,j}}^{t}
\otimes\cdots\otimes\alpha_{x_{i,j+m}}^{t}\in\mathbf{W}_{x_{i,j}
,...,x_{i,j+m}}^{t}.$

\item[\textit{(ii)}] $\left[  \alpha_{x_{i,j}}^{y_{i,j}}\text{ }\cdots\text{
}\alpha_{x_{i+r,j}}^{y_{i+r,j}}\right]  ^{T}$ is a string in the $j^{th}$
column of $A$ such that $x_{i,j}=\cdots=x_{i+r,j}=s,$ then $\alpha
_{s}^{y_{i,j}}\otimes\cdots\otimes\alpha_{s}^{y_{i+r,j}}\in\mathbf{W}
_{s}^{y_{i,j},...,y_{i+r,j}}.$
\end{enumerate}

\noindent Thus if $A\in\mathbf{M}_{\mathbf{x}}^{\mathbf{y}}\cap\mathcal{W}$,
the $i^{th}$ row of $A$ lies in $\mathbf{W}_{\mathbf{x}}^{y_{i}}$, and the
$j^{th}$ column of $A$ lies in $\mathbf{W}_{x_{j}}^{\mathbf{y}}$.

\begin{definition}
A pair of $\left(  q\times s,t\times p\right)  $ monomials $A\otimes B=\left[
\alpha_{v_{kl}}^{y_{kl}}\right]  \otimes\left[  \beta_{x_{ij}}^{u_{ij}
}\right]  \in\overline{\mathbf{M}}\mathbf{\otimes}\overline{\mathbf{M}}$ is a

\begin{enumerate}
\item[\textit{(i)}] \textbf{Transverse Pair} (TP) if $s=t=1,$ $u_{1,j}=q,$ and
$v_{k,1}=p$ for all $j$ and $k,$ i.e., setting $x_{j}=x_{1,j}$ and
$y_{k}=y_{k,1}$ gives
\[
A\otimes B=\left[
\begin{array}
[c]{c}
\alpha_{p}^{y_{1}}\\
\vdots\\
\alpha_{p}^{y_{q}}
\end{array}
\right]  \otimes\left[
\begin{array}
[c]{lll}
\beta_{x_{1}}^{q} & \cdots & \beta_{x_{p}}^{q}
\end{array}
\right]  \in\mathbf{M}_{p}^{\mathbf{y}}\otimes\mathbf{M}_{\mathbf{x}}^{q}.
\]

\item[\textit{(ii)}] \textbf{Block Transverse Pair} (BTP) if there exist
$t\times s$ block decompositions $A=\left[  A_{k^{\prime},l}^{\prime}\right]
$ and $B=\left[  B_{i,j^{\prime}}^{\prime}\right]  $ such that $A_{il}
^{\prime}\otimes B_{il}^{\prime}$ is a TP for\ all $i$ and $l$.
\end{enumerate}
\end{definition}

\noindent A pair $A\otimes B\in\mathbf{M}_{\mathbf{v}}^{\mathbf{y}}
\otimes\mathbf{M}_{\mathbf{x}}^{\mathbf{u}}$ is a BTP if and only if
$\mathbf{x\times y}\in\mathbb{N}^{1\times|\mathbf{v}|}\times\mathbb{N}
^{|\mathbf{u}|\times1}$ if and only if the initial point of arrow $A$
coincides with the terminal point of arrow $B$ as operators on $\mathbb{N}
^{2}$.

Given a family of maps $\bar{\gamma}=\{M^{\otimes q}\otimes M^{\otimes
p}\rightarrow M\}_{p,q\geq1},$ let $\gamma=\{\gamma_{\mathbf{x}}^{\mathbf{y}
}:\mathbf{M}_{p}^{\mathbf{y}}\otimes\mathbf{M}_{\mathbf{x}}^{q}\rightarrow
\mathbf{M}_{\left\vert \mathbf{x}\right\vert }^{\left\vert \mathbf{y}
\right\vert }\}.\ $Then $\gamma$ induces a global product $\Upsilon
:\overline{\mathbf{M}}\otimes\overline{\mathbf{M}}\rightarrow\overline
{\mathbf{M}}$ defined by
\begin{equation}
\Upsilon\left(  A\otimes B\right)  =\left\{
\begin{array}
[c]{ll}
\left[  \gamma\left(  A_{ij}^{\prime}\otimes B_{ij}^{\prime}\right)  \right]
, & A\otimes B\ \text{is a\ BTP}\\
0, & \text{otherwise,}
\end{array}
\right.  \label{upsilon}
\end{equation}
where $A_{ij}^{\prime}\otimes B_{ij}^{\prime}$ is the $\left(  i,j\right)
^{th}$ TP in the BTP decomposition of $A\otimes B.$ Obviously $\Upsilon$ is
closed in both $\overline{\mathbf{M}}_{{\operatorname*{row}}}$ and
$\overline{\mathbf{M}}^{{\operatorname*{col}}},$ and consequently in
$\mathbf{M.}$ We denote the $\Upsilon$-product by \textquotedblleft$\cdot
$\textquotedblright\ or juxtaposition; when $A\otimes B=\left[  \alpha
_{p}^{y_{j}}\right]  ^{T}\otimes\left[  \beta_{x_{i}}^{q}\right]  $ is a TP,
we write
\[
AB=\gamma(\alpha_{p}^{y_{1}},\ldots,\alpha_{p}^{y_{q}};\beta_{x_{1}}
^{q},\ldots,\beta_{x_{p}}^{q}).
\]
When pictured as an arrow in $\mathbb{N}^{2},$ $AB$ \textquotedblleft
transgresses\textquotedblright\ from the horizontal axis $y=1$ to the vertical
axis $x=1$ in $\mathbb{N}^{2}$.

Recall that the $\Upsilon$-product acts associatively on $\mathbf{M}$ (see \cite{SU4}). However, if  $A_1\!\cdots A_n\!$ is a matrix string in $\overline{\mathbf{M}}$, associativity fails whenever one association of $A_1\cdots A_n$ produces a sequence of BTPs while some other does not. For example, in
\[
ABC=\left[
\begin{array}{c}
\theta _{2}^{1} \\
\theta _{2}^{1}
\end{array}
\right] \left[
\begin{array}{cc}
\theta _{2}^{1} & \mathbf{1} \\
\mathbf{1} & \theta _{2}^{1}
\end{array}
\right] \left[
\begin{array}{ccc}
\theta _{1}^{2} & \theta _{1}^{2} & \theta _{1}^{2}
\end{array}
\right],
\]
$\{A\otimes B, AB\otimes C\}$ is a sequence of BTPs but $\{B\otimes C, A\otimes BC\}$ is not; thus $(AB)C\neq0$ while $A(BC)=0$.

Let $\mathbf{1}^{1\times p}=\left(  1,\ldots,1\right)  \in\mathbb{N}^{1\times
p}$, and $\mathbf{1}^{q\times1}=\left(  1,\ldots,1\right)  ^{T}\in
\mathbb{N}^{q\times1}$; we often suppress the exponents when the context is clear.

\begin{definition}
\label{prematrad}A \textbf{prematrad} $(M,\gamma,\eta)$ is a bigraded
$R$-module $M=\left\{  M_{n,m}\right\}  _{m,n\geq1}$ together with a family of
structure maps $\gamma=\{\gamma_{\mathbf{x}}^{\mathbf{y}}:\mathbf{M}
_{p}^{\mathbf{y}}\otimes\mathbf{M}_{\mathbf{x}}^{q}\rightarrow\mathbf{M}
_{\left\vert \mathbf{x}\right\vert }^{\left\vert \mathbf{y}\right\vert }\}$
and a unit $\eta:R\rightarrow\mathbf{M}_{1}^{1}$ such that

\begin{enumerate}
\item[\textit{(i)}] $\Upsilon\left(  \Upsilon\left(  A;B\right)  ;C\right)
=\Upsilon\left(  A;\Upsilon\left(  B;C\right)  \right)  $ whenever $A\otimes
B$ and $B\otimes C$ are BTPs in $\overline{\mathbf{M}}\otimes\overline
{\mathbf{M}};$

\item[\textit{(ii)}] the following compositions are the canonical
isomorphisms:
\begin{align*}
&  R^{\otimes b}\otimes\mathbf{M}_{a}^{b}\overset{\eta^{\otimes b}
\otimes\operatorname*{Id}}{\longrightarrow}\mathbf{M}_{1}^{\mathbf{1}
^{b\times1}}\otimes\mathbf{M}_{a}^{b}\overset{\gamma_{a}^{\mathbf{1}
^{b\times1}}}{\longrightarrow}\mathbf{M}_{a}^{b};\\
&  \mathbf{M}_{a}^{b}\otimes R^{\otimes a}\overset{\operatorname*{Id}
\otimes\eta^{\otimes a}}{\longrightarrow}\mathbf{M}_{a}^{b}\otimes
\mathbf{M}_{\mathbf{1}^{1\times a}}^{1}\overset{\gamma_{\mathbf{1}^{1\times
a}}^{b}}{\longrightarrow}\mathbf{M}_{a}^{b}.
\end{align*}
We denote the element $\eta(1_{R})$ by $\mathbf{1}_{\mathbf{M}}.$
\end{enumerate}

\noindent A \textbf{morphism of prematrads} $(M,\gamma)$ and $(M^{\prime
},\gamma^{\prime})$ is a map $f:{M}\rightarrow{M}^{\prime}$ such that
$f\gamma_{\mathbf{x}}^{\mathbf{y}}=\gamma{^{\prime}}_{\mathbf{x}}^{\mathbf{y}
}(f^{\otimes q}\otimes f^{\otimes p})$ for all $\mathbf{x}\times\mathbf{y} $.
\end{definition}

Note that if  $A_1\cdots A_n$ is a matrix string in a prematrad, Axiom (i) implies that  associativity holds whenever \emph{every} association of $A_1\cdots A_n$ produces a sequence of BTPs.

\begin{definition}
\label{basic}Let $(M,\gamma,\eta)$ be a prematrad. A string\textbf{\ }of
matrices $A_{s}\cdots A_{1}$ is \textbf{basic} \textbf{in} $\mathbf{M}_{m}
^{n}$ if

\begin{enumerate}
\item[\textit{(i)}] $A_{1}\in\mathbf{M}_{\mathbf{x}}^{b},\ \left\vert
\mathbf{x}\right\vert =m,$

\item[\textit{(ii)}] $A_{i}\in\overline{\mathbf{M}}\smallsetminus\left\{
\mathbf{1}^{q\times p}\mid p,q\in\mathbb{N}\right\}  $ for all $i,$

\item[\textit{(iii)}] $A_{s}\in\mathbf{M}_{a}^{\mathbf{y}},$ $\left\vert
\mathbf{y}\right\vert =n,$ and

\item[\textit{(iv)}] some association of $A_{s}\cdots A_{1}$ defines a
sequence of BTPs and non-zero $\Upsilon$-products.
\end{enumerate}
\end{definition}

Definition \ref{freeprematrad} constructs a bigraded set
$G^{\operatorname{pre}}=G_{\ast,\ast}^{\operatorname*{pre}},$ where
$G_{n,m}^{\operatorname*{pre}}$ is defined in terms of
\[
G_{\left[  n,m\right]  }=\bigcup\limits_{\substack{i\leq m,\text{ }j\leq
n,\\i+j<m+n}}G_{j,i}^{\operatorname*{pre}}.
\]
We denote the sets of matrices over $G_{\left[  n,m\right]  }$ and
$G^{\operatorname{pre}}$ by $\overline{\mathbf{G}}_{\left[  n,m\right]  }
\ $and $\overline{\mathbf{G}},$ respectively; $\mathbf{G}$ denotes the subset
of bisequence matrices in $\overline{\mathbf{G}}\mathbf{.}$

\begin{definition}
\label{freeprematrad}Let $\Theta=\left\langle \theta_{m}^{n}\mid\theta_{1}
^{1}=\mathbf{1}\right\rangle _{m,n\geq1}$ be a free bigraded $R$-module
generated by singletons $\theta_{m}^{n}$, and set $G_{1,1}^{\operatorname{pre}
}=\mathbf{1}.$ Inductively, if $m+n\geq3$ and $G_{j,i}^{\operatorname*{pre}}$
has been constructed for $i\leq m,$ $j\leq n,$ and $i+j<m+n,$ define
\[
G_{n,m}^{\operatorname{pre}}=\theta_{m}^{n}\cup\left\{  \text{basic strings
}A_{s}\cdots A_{1}\text{ in }\mathbf{G}_{m}^{n}\text{ with }s\geq2\right\}  .
\]
Let $\sim$ be the equivalence relation on $G^{\operatorname{pre}}=G_{\ast
,\ast}^{\operatorname*{pre}}$ generated by $\left[  A_{ij}B_{ij}\right]
\sim\left[  A_{ij}\right]  \left[  B_{ij}\right]  $ if and only if $\left[
A_{ij}\right]  \times\left[  B_{ij}\right]  \in\overline{\mathbf{G}}
\times\overline{\mathbf{G}}$ is a BTP, and let $F^{{\operatorname{pre}}
}\left(  \Theta\right)  =\langle G^{\operatorname{pre}}\diagup\sim\rangle.$
The \textbf{free prematrad generated by} $\Theta$ is the prematrad
$(F^{{\operatorname{pre}}}(\Theta),\gamma,\eta),$ where $\gamma$ is
juxtaposition and $\eta:R\rightarrow F_{1,1}^{\operatorname{pre}}\left(
\Theta\right)  $ is given by $\eta\left(  1_{R}\right)  =\mathbf{1}$.
\end{definition}

Let $\mathcal{W}$ be the telescopic extension of a telescoping submodule
$\mathbf{W.}$ If $A\otimes B$ is a BTP in $\mathcal{W}\otimes\mathcal{W}$,
each TP $A^{\prime}\otimes B^{\prime}$ in $A\otimes B$ lies in $\mathbf{W}
_{p}^{\mathbf{y}}\otimes\mathbf{W}_{\mathbf{x}}^{q}$ for some $\mathbf{x}
,\mathbf{y},p,q.$ Consequently, $\gamma_{_{\mathbf{W}}}$ extends to a global
product $\Upsilon:\mathcal{W}\otimes\mathcal{W}\rightarrow\mathcal{W}$ as in
(\ref{upsilon}). In fact, $\mathcal{W}$ is the smallest matrix submodule
containing $\mathbf{W}$ on which $\Upsilon$ is well-defined.

\begin{definition}
Let $\mathbf{W}$ be a telescoping submodule of $TTM$, and let $\mathcal{W}$ be
its telescopic extension. Let $\gamma_{_{\mathbf{W}}}=\left\{  \gamma
_{\mathbf{x}}^{\mathbf{y}}:\mathbf{W}_{p}^{\mathbf{y}}\otimes\mathbf{W}
_{\mathbf{x}}^{q}\rightarrow\mathbf{W}_{\left\vert \mathbf{x}\right\vert
}^{\left\vert \mathbf{y}\right\vert }\right\}  $ be a structure map, and let
$\eta:R\rightarrow\mathbf{M}$. The triple $\left(  M,\gamma_{_{\mathbf{W}}
},\eta\right)  $ is a \textbf{local prematrad (with domain} $\mathbf{W}$) if
the following axioms are satisfied:

\begin{enumerate}
\item[\textit{(i)}] $\mathbf{W}_{\mathbf{x}}^{1}=\mathbf{M}_{\mathbf{x}}^{1}$
and $\mathbf{W}_{1}^{\mathbf{y}}=\mathbf{M}_{1}^{\mathbf{y}}$ for all
$\mathbf{x},\mathbf{y};$

\item[\textit{(ii)}] $\Upsilon\left(  \Upsilon\left(  A;B\right)  ;C\right)
=\Upsilon\left(  A;\Upsilon\left(  B;C\right)  \right)  $ whenever $A\otimes
B$ and $B\otimes C$ are BTPs in $\mathcal{W\otimes W}$;

\item[\textit{(iii)}] the prematrad unit axiom holds for $\gamma
_{_{\mathbf{W}}}.$
\end{enumerate}
\end{definition}

The ``configuration module'' of a local prematrad is defined in terms of
the combinatorics of permutahedra.
Recall that each codimension $k$ face of the permutahedron $P_{m}$ is
identified with a pair of planar rooted trees with levels (PLTs)-- one up-rooted
with $m+1$ leaves and $k+1$ levels, and the other its down-rooted mirror image
(see \cite{Loday}, \cite{SU2}). Define the $\left(  m,1\right)  $\emph{-row
descent sequence} of the $m$-leaf up-rooted corolla $\curlywedge_{m}$ to be
$\mathbf{m}=\left(  m\right)  .$ Given an up-rooted PLT $T$ with $k\geq2$
levels, successively remove the levels of $T$ and obtain a sequence of
subtrees $T=T^{k},T^{k-1},...,T^{1},$ in which $T^{i}-T^{i-1}$ is the sequence
of corollas $\curlywedge_{m_{i,1}}\cdots\curlywedge_{m_{i,r_{i}}}.$ Recover
$T$ inductively by attaching $\curlywedge_{m_{i,j}}$ to the $j^{th}$ leaf of
$T^{i-1}.$ Define the $i^{th}$ \emph{leaf sequence} of $T$ to be the row
matrix $\mathbf{m}_{i}=(m_{i,1},\ldots,m_{i,r_{i}})$, and the $\left(
m,k\right)  $\emph{-row descent sequence of} $T$ to be the $k$-tuple of row
matrices $\left(  \mathbf{m}_{1},\ldots,\mathbf{m}_{k}\right)  .$ Dually,
define the $\left(  n,1\right)  $\emph{-column descent sequence} of the
$n$-leaf down-rooted corolla $\curlyvee^{n}$ to be $\mathbf{n}=\left(
n\right)  .$ Define the $\left(  n,1\right)  $\emph{-column descent sequence}
of the $n$-leaf down-rooted corolla $\curlyvee_{n}$ to be $\mathbf{n}=\left(
n\right)  .$ Given a down-rooted PLT $T$ with $l\geq2$ levels, successively
remove the levels of $T$, and obtain a sequence of subtrees $T=T^{l}
,T^{l-1},...,T^{1},$ in which $T^{i}-T^{i-1}$ is the sequence of corollas
$\curlyvee^{n_{i,1}}\cdots\curlyvee^{n_{i,s_{i}}}.$ Recover $T$ inductively by
attaching $\curlyvee^{n_{i,j}}$ to the $j^{th}$ leaf of $T^{i-1}$ for each
$i.$ Define the $i^{th}$\emph{\ leaf sequence} of $T$ to be the column matrix
$\mathbf{n}_{i}=(n_{i,1},\ldots,n_{i,s_{i}})^{T}$, and the $\left(  n,l\right)
$-\emph{column} \emph{descent sequence of }$T$ to be the $l$-tuple of column
matrices $\left(  \mathbf{n}_{l},\ldots,\mathbf{n}_{1}\right)  $.

Let $\mathcal{W}_{{\operatorname{row}}}=\mathcal{W}\cap\overline{\mathbf{M}
}_{{\operatorname{row}}}$, and $\mathcal{W}^{\operatorname{col}}=\mathcal{W}
\cap\overline{\mathbf{M}}^{\operatorname{col}}.$

\begin{definition}
Given a local prematrad $(M,\gamma_{\mathbf{W}})$ with domain $\mathbf{W},$
let $\zeta\in{M}_{\ast,m}$ and $\xi\in M_{n,\ast}$ be elements with
$m,n\geq2.$

\begin{enumerate}
\item[\textit{(i)}] A \textbf{row factorization of} $\zeta$ \textbf{with
respect to} $\!\mathbf{W}$ \!is a $\Upsilon\!$-factorization
$A_{1}\!\cdots\! A_{k}\\
 = \zeta$ such that $A_{j}\in\mathcal{W}_{\operatorname{row}}$
and ${\operatorname{rls}}(A_{j})\neq\mathbf{1}$ for all $j$. The sequence
$\alpha=\left(  {\operatorname{rls}}(A_{1}), ...,{\operatorname{rls}}
(A_{k})\right)  $ is the related $\left(  m,k\right)  $-row \textbf{descent
sequence of} $\zeta$.

\item[(\textit{ii)}] A \textbf{column factorization of }$\xi$ \textbf{with
respect to} $\mathbf{W}$ is a $\Upsilon$-factorization $B_{l}\cdots B_{1}=\xi$
such that $B_{i}\in\mathcal{W}^{{\operatorname{col}}}$\thinspace and
${\operatorname{cls}}(B_{i})\neq\mathbf{1}$ for all $i.$ The sequence
$\beta=\left(  {\operatorname{cls}}(B_{l}),...,{\operatorname{cls}}
(B_{1})\right)  $ is the related $\left(  n,l\right)  $-\textbf{column}
\textbf{descent sequence of} $\xi$.
\end{enumerate}
\end{definition}

Given a local prematrad $(M,\gamma_{\mathbf{W}})$, and elements $A\in{M}
_{\ast,s}$ and $B\in{M}_{t,\ast}$ with $s,t\geq2$, choose a row factorization
$A_{1}\cdots A_{k}$ of $A$ with respect to $\mathbf{W}$ having $\left(
s,k\right)  $-row descent sequence $\alpha,$ and a column factorization
$B_{l}\cdots B_{1}$ of $B\ $with respect to $\mathbf{W}$ having $\left(
t,l\right)  $-column descent sequence $\beta.$ Then $\alpha$ identifies $A$
with an up-rooted $s$-leaf, $k$-level PLT and a codimension $k-1$ face
$\overset{_{\wedge}}{e}_{A}$ of $P_{s-1}$, and $\beta$ identifies $B$ with a
down-rooted $t$-leaf, $l$-level PLT and a codimension $l-1$ face
$\overset{_{\vee}}{e}_{B}$ of $P_{t-1}$. Extending to Cartesian products,
identify the monomials $A=A_{1}\otimes\cdots\otimes A_{q}\in\left(  {M}
_{\ast,s}\right)  ^{\otimes q}$ and $B=B_{1}\otimes\cdots\otimes B_{p}
\in\left(  {M}_{t,\ast}\right)  ^{\otimes p}$ with the product cells
\[
\overset{\wedge}{e}_{A}=\overset{_{\wedge}}{e}_{A_{1}}\times\cdots
\times\overset{_{\wedge}}{e}_{A_{q}}\subset P_{s-1}^{\times q}\ \ \text{and}
\ \ \overset{_{\vee}}{e}_{B}=\overset{_{\vee}}{e}_{B_{1}}\times\cdots
\times\overset{_{\vee}}{e}_{B_{p}}\subset P_{t-1}^{\times p}.
\]

Now consider the S-U diagonal $\Delta_{P}$ (see \cite{SU2}), and recall that there is a
$k$-subdivision $P_{r}^{\left(  k\right)  }$ of $P_{r}$ and a cellular
inclusion $P_{r}^{\left(  k\right)  }\hookrightarrow\Delta^{\left(  k\right)
}\left(  P_{r}\right)  \subset P_{r}^{\times k+1}$ for each $k$ and $r$ (see
\cite{SU4}). Thus for each $q\geq2$, the product cell $\overset{\wedge}{e}
_{A}$ either\emph{\ is} or \emph{is not} a subcomplex of ${\Delta}
^{(q-1)}(P_{s-1})\subset P_{s-1}^{\times q}$, and dually for $\overset{\vee
}{e}_{B}.$ Let $\mathbf{x}_{m,i}^{p}=\left(  1,\ldots,m,\ldots,1\right)
\in\mathbb{N}^{1\times p}$ with $m$ in the $i^{th}$ position, and let
$\mathbf{y}_{q}^{n,j}=\left(  1,\ldots,n,\ldots,1\right)  ^{T}\in
\mathbb{N}^{q\times1}$ with $n$ in the $j^{th}$ position.

\begin{definition}
The \textbf{(left) configuration module} of a local prematrad $\left(
M,\gamma_{_{\mathbf{W}}}\right)  $ is the $R$-module

\[
\Gamma(M,\gamma_{_{\mathbf{W}}})=M\oplus\bigoplus_{\mathbf{x,y}\notin
\mathbb{N};\text{ }s,t\geq1}\Gamma_{s}^{\mathbf{y}}(M)\oplus\Gamma
_{\mathbf{x}}^{t}(M),
\]
where
\begin{align*}
\Gamma_{s}^{\mathbf{y}}(M)  &  =\left\{
\begin{array}
[c]{cl}
\mathbf{M}_{1}^{\mathbf{y}}, & s=1;\text{ }\mathbf{y}=\mathbf{y}_{q}
^{n,j}\text{ for some }n,j,q\\
\left\langle A\in\mathbf{M}_{s}^{\mathbf{y}}\mid\overset{_{\wedge}}{e}
_{A}\subset{\Delta}^{(q-1)}(P_{s-1})\right\rangle , & s\geq2\\
0, & \text{otherwise,}
\end{array}
\right. \\
& \\
\Gamma_{\mathbf{x}}^{t}(M)  &  =\left\{
\begin{array}
[c]{cl}
\mathbf{M}_{\mathbf{x}}^{1}, & t=1;\text{ }\mathbf{x}=\mathbf{x}_{m,i}
^{p}\text{ for some }m,i,p\\
\left\langle B\in\mathbf{M}_{\mathbf{x}}^{t}\mid\overset{_{\vee}}{e}
_{B}\subset{\Delta}^{(p-1)}(P_{t-1})\right\rangle , & t\geq2\\
0, & \text{otherwise.}
\end{array}
\right.
\end{align*}

\end{definition}

\noindent Thus $\Gamma_{\mathbf{x}}^{t}(M)$ is generated by those tensor
monomials $B=\beta_{x_{1}}^{t}\otimes\cdots\otimes\beta_{x_{p}}^{t}\in
{M}_{t,x_{1}}\otimes\cdots\otimes{M}_{t,x_{p}}$ whose tensor factor
$\beta_{x_{i}}^{t}$ is identified with some factor of a product cell in
${\Delta}^{\left(  p-1\right)  }(P_{t-1})$ corresponding to some column
factorization $\beta_{x_{i}}^{t}=B_{i,l}\cdots B_{i,1}$ with respect to
$\mathbf{W}$, and dually for $\Gamma_{s}^{\mathbf{y}}(M).$

\begin{definition}
A local prematrad $(M,\gamma_{_{\mathbf{W}}},\eta)$ is a
\textbf{(left)\thinspace matrad} if
\[
\Gamma_{p}^{\mathbf{y}}(M,\gamma_{_{\mathbf{W}}})\otimes\Gamma_{\mathbf{x}
}^{q}(M,\gamma_{_{\mathbf{W}}})=\mathbf{W}_{p}^{\mathbf{y}}\otimes
\mathbf{W}_{\mathbf{x}}^{q}
\]
for all $p,q\geq2.$ A \textbf{morphism of matrads} is a map of underlying
local prematrads.
\end{definition}

The domain of the free prematrad $(M=F^{^{{\operatorname{pre}}}}
(\Theta),\gamma,$ $\eta)$ generated by $\Theta=\left\langle \theta_{m}
^{n}\right\rangle _{m,n\geq1}$ is $\mathbf{V}=M\oplus\bigoplus
\limits_{\mathbf{x},\mathbf{y}\notin\mathbb{N};\text{ }s,t\in\mathbb{N}
}\mathbf{M}_{s}^{\mathbf{y}}\oplus\mathbf{M}_{\mathbf{x}}^{t},\ $whose
submodules $M,$ $\mathbf{M}_{\mathbf{x}}^{1},$ and $\mathbf{M}_{1}
^{\mathbf{y}}$ are contained in the configuration module $\Gamma\left(
M\right)  $. As above, the symbol \textquotedblleft$\cdot$\textquotedblright
\ denotes the $\gamma$ product.

\begin{definition}
\label{defnfreematrad}Let $(M=F^{^{{\operatorname{pre}}}}(\Theta),\gamma
,\eta)$ be the free prematrad generated by $\Theta=\left\langle \theta_{m}
^{n}\right\rangle _{m,n\geq1},$ let $F(\Theta)=\Gamma(M)\cdot\Gamma(M),$ and
let $\gamma_{_{F\left(  \Theta\right)  }}=\gamma|_{_{\Gamma(M)\otimes
\Gamma(M)}}.$ The \textbf{free matrad generated by} $\Theta$ is the triple
$\left(  F(\Theta),\gamma_{_{F\left(  \Theta\right)  }},\eta\right)  .$
\end{definition}

There is a chain map that identifies module generators of $F_{n,m}(\Theta)$
with cells of the biassociahedron $KK_{n,m}$. We denote the differential
object $\left(  F(\Theta),\partial\right)  $ by $\mathcal{H}_{\infty}$, and
define an $A_{\infty}$-infinity bialgebra as an algebra over $\mathcal{H}
_{\infty}$.

\section{Relative Matrads}

Let $\left(  M,\gamma_{_{M}},\eta_{_{M}}\right)  $ and $\left(  N,\gamma
_{_{N}},\eta_{_{N}}\right)  $ be $R$-prematrads, and let $E=\{E_{n,m}
\}_{m,n\geq1}$ be a bigraded $R$-module. Left and right actions $\lambda
=\{\lambda_{\mathbf{x}}^{\mathbf{y}}:\mathbf{M}_{p}^{\mathbf{y}}
\otimes\mathbf{E}_{\mathbf{x}}^{q}\rightarrow\mathbf{E}_{\left\vert
\mathbf{x}\right\vert }^{\left\vert \mathbf{y}\right\vert }\}$ and
$\rho=\{\rho_{\mathbf{x}}^{\mathbf{y}}:\mathbf{E}_{p}^{\mathbf{y}}
\otimes\mathbf{N}_{\mathbf{x}}^{q}\rightarrow\mathbf{E}_{\left\vert
\mathbf{x}\right\vert }^{\left\vert \mathbf{y}\right\vert }\}$ induce left and
right global products $\Upsilon_{\lambda}:\overline{\mathbf{M}}\otimes
\overline{\mathbf{E}}\rightarrow\overline{\mathbf{E}}$ and $\Upsilon_{\rho
}:\overline{\mathbf{E}}\otimes\overline{\mathbf{N}}\rightarrow\overline
{\mathbf{E}}$ in the same way that $\gamma_{M}=\{\gamma_{\mathbf{x}
}^{\mathbf{y}}:\mathbf{M}_{p}^{\mathbf{y}}\otimes\mathbf{M}_{\mathbf{x}}
^{q}\rightarrow\mathbf{M}_{\left\vert \mathbf{x}\right\vert }^{\left\vert
\mathbf{y}\right\vert }\}$ induces $\Upsilon_{M}:\overline{\mathbf{M}}
\otimes\overline{\mathbf{M}}\rightarrow\overline{\mathbf{M}}$ (see \cite{SU4}
and formula (\ref{upsilon})).

\begin{definition}
\label{reldefn}A tuple $\left(  M,E,N,\lambda,\rho\right)  $ is a
\textbf{relative prematrad} if

\begin{enumerate}
\item[\textit{(i)}] Associativity holds:

\begin{enumerate}
\item[\textit{(a)}] $\Upsilon_{\rho}(\Upsilon_{\lambda}\otimes${$\mathbf{1}$
}$) = \Upsilon_{\lambda}(${$\mathbf{1}$}$\otimes\Upsilon_{\rho});$

\item[\textit{(b)}] $\Upsilon_{\lambda}(\Upsilon_{M}\otimes${$\mathbf{1}$ }$)
= \Upsilon_{\lambda}(${$\mathbf{1}$}$\otimes\Upsilon_{\lambda});$

\item[\textit{(c)}] $\Upsilon_{\rho}(\Upsilon_{\rho}\otimes${$\mathbf{1}$}$) =
\Upsilon_{\rho}(${$\mathbf{1}$}$\otimes\Upsilon_{N}).$
\end{enumerate}

\item[\textit{(ii)}] The units $\eta_{M}$ and $\eta_{N}$ induce the following
canonical isomorphisms for all $a,b\in\mathbb{N}$:
\begin{align*}
&  R^{\otimes b}\otimes\mathbf{E}_{a}^{b}\overset{\eta_{_{M}}^{\otimes
b}\otimes{\mathbf{1}}}{\longrightarrow}\mathbf{M}_{1}^{\mathbf{1}^{b\times1}
}\otimes\mathbf{E}_{a}^{b}\overset{\lambda_{a}^{\mathbf{1}^{b\times1}
}}{\longrightarrow}\mathbf{E}_{a}^{b};\\
&  \mathbf{E}_{a}^{b}\otimes R^{\otimes a}\overset{{\mathbf{1}}\otimes\eta
_{N}^{\otimes a}}{\longrightarrow}\mathbf{E}_{a}^{b}\otimes\mathbf{N}
_{\mathbf{1}^{1\times a}}^{1}\overset{\rho_{\mathbf{1}^{1\times a}}
^{b}}{\longrightarrow}\mathbf{E}_{a}^{b}.
\end{align*}

\end{enumerate}

\noindent We refer to $E$ as an $M$-$N$-\textbf{bimodule; } when $M=N$ we refer
to $E$ as an $M$\textbf{-bimodule}.
\end{definition}

\begin{definition}
A \textbf{morphism} $f:(M,E,N,\lambda,\rho)\rightarrow(M^{\prime},E^{\prime
},N^{\prime},\lambda^{\prime},\rho^{\prime})$ of relative prematrads is a
triple $\left(  f_{M}:{M}\rightarrow{M}^{\prime},\text{ }f_{E}:E\rightarrow
E^{\prime},\text{ }f_{N}:N\rightarrow N^{\prime}\right)  $ such that

\begin{enumerate}
\item[\textit{(i)}] $f_{M}$ and $f_{N}$ are maps of prematrads;

\item[\textit{(ii)}] $f_{E}$ commutes with left and right actions, i.e.,
$f_{E}\circ\lambda_{\mathbf{x}}^{\mathbf{y}}=\lambda{^{\prime}}_{\mathbf{x}
}^{\mathbf{y}}\circ(f_{M}^{\otimes q}\otimes f_{E}^{\otimes p})$ and
$f_{E}\circ\rho_{\mathbf{x}}^{\mathbf{y}}=\rho{^{\prime}}_{\mathbf{x}
}^{\mathbf{y}}\circ(f_{E}^{\otimes q}\otimes f_{N}^{\otimes p})$ for all
$\mathbf{x\times y\in}\mathbb{N}^{p\times1}\times\mathbb{N}^{1\times q}.$
\end{enumerate}
\end{definition}

Tree representations of $\lambda_{1}^{\mathbf{x}}$ and $\rho_{1}^{\mathbf{x}}$
are related to those of $\lambda_{\mathbf{x}}^{1}$ and $\rho_{\mathbf{x}}^{1}$
by a reflection in some horizontal axis. Although $\rho_{\mathbf{x}}^{1}$
agrees with Markl, Shnider, and Stasheff's right module action over an operad
\cite{MSS}, $\lambda_{\mathbf{x}}^{1}$ differs fundamentally from their left
module action, and our definition of an \textquotedblleft
operadic\ bimodule\textquotedblright\ is consistent with their definition of
an operadic ideal.

Given graded $R$-modules $A$ and $B,$ let
\[
\begin{array}
[c]{lllll}
U_{A} & = & \mathcal{E}nd_{TA} & = & \left\{  N_{s,p}\,=Hom\left(  A^{\otimes
p},A^{\otimes s}\right)  \right\}  _{p,s\geq1}\\
U_{A,B} & = & Hom\left(  TA,TB\right)  & = & \left\{  E_{t,q}\,\,=Hom\left(
A^{\otimes q},B^{\otimes t}\right)  \right\}  _{q,t\geq1}\\
U_{B} & = & \mathcal{E}nd_{TB} & = & \left\{  M_{u,r}=Hom\left(  B^{\otimes
r},B^{\otimes u}\right)  \right\}  _{r,u\geq1},
\end{array}
\]
and define left and right actions $\lambda$ and $\rho$ in terms of the
horizontal and vertical operations $\times$ and $\circ$ analogous to those in
the prematrad structures on $U_{A}$ and $U_{B}$ (see Section 2 above and
\cite{SU4}). Then the\emph{\ }relative $\operatorname*{PROP}\left(
U_{B},U_{A,B},U_{A},\lambda,\rho\right)  $ is the universal example of a
relative prematrad.

Definition \ref{freerelpre} constructs a bigraded set $\mathcal{G}
^{\operatorname{pre}}=\mathcal{G}_{\ast,\ast}^{\operatorname*{pre}},$ where
$\mathcal{G}_{n,m}^{\operatorname*{pre}}$ is defined in terms of
\[
\mathcal{G}_{\left[  n,m\right]  }=\bigcup\limits_{\substack{i\leq m,\text{
}j\leq n,\\i+j<m+n}}\mathcal{G}_{j,i}^{\operatorname*{pre}}.
\]
We denote the sets of matrices over $\mathcal{G}_{\left[  n,m\right]  }$ and
$\mathcal{G}^{{\operatorname*{pre}}}$ by $\overline{\mathbf{C}}_{\left[
n,m\right]  }$ and $\overline{\mathbf{C}},$ respectively; $\mathbf{C}$ denotes
the subset of bisequence matrices in $\overline{\mathbf{C}}.$ As before,
$\overline{\mathbf{G}}$ denotes the sets of matrices over
$G^{\operatorname{pre}}$, and $\mathbf{G}$ denotes the subset of bisequence
matrices in $\overline{\mathbf{G}}$.

\begin{definition}
\label{freerelpre}Given a free bigraded $R$-module $\Theta=\left\langle
\theta_{m}^{n}\mid\theta_{1}^{1}=\mathbf{1}\right\rangle _{m,n\geq1}$
generated by singletons in each bidegree, let $\left(  F^{{\operatorname*{pre}
}}\left(  \Theta\right)  ,\gamma,\eta\right)  $ be the free prematrad
generated by $\Theta.$ Let $\mathfrak{F}=\left\langle \mathfrak{f}_{m}
^{n}\right\rangle _{m,n\geq1}$ be a free bigraded $R$-module generated by
singletons in each bidegree, and set $\mathcal{G}_{1,1}^{\operatorname*{pre}
}=\mathfrak{f}_{1}^{1}\mathbf{.}$ Inductively, if $m+n\geq3$ and
$\overline{\mathbf{C}}_{\left[  n,m\right]  }$ has been constructed, define
\[
\mathcal{G}_{n,m}^{{\operatorname*{pre}}}=\mathfrak{f}_{m}^{n}\cup\left\{
B_{1}\cdots B_{l}\cdot C\cdot A_{k}\cdots A_{1}\mid k+l\geq1;\text{ }
k,l\geq0\right\}  ,
\]
where

\begin{enumerate}
\item[\textit{(i)}] $A_{1}\in\mathbf{G}_{\mathbf{x}}^{q}$ and $B_{1}
\in\mathbf{G}_{p}^{\mathbf{y}}$ for $\left(  \left\vert \mathbf{x}\right\vert
,\left\vert \mathbf{y}\right\vert \right)  =\left(  m,n\right)  ;$
\vspace*{0.05in}

\item[\textit{(ii)}] $A_{i},B_{j}\in\overline{\mathbf{G}}_{\left[  n,m\right]
}$ for all $i,j;$\vspace*{0.05in}

\item[\textit{(iii)}] $C\in\overline{\mathbf{C}}_{\left[  n,m\right]  };$
and\vspace*{0.05in}

\item[\textit{(vii)}] some association of $B_{1}\cdots B_{l}\cdot C\cdot
A_{k}\cdots A_{1}$ defines a sequence of BTPs.
\end{enumerate}

\noindent Let $\sim$ be the equivalence relation on $\mathcal{G}
^{{\operatorname*{pre}}}=\mathcal{G}_{\ast,\ast}^{{\operatorname*{pre}}}$
generated by
\[
\left[  X_{ij}Y_{ij}\right]  \sim\left[  X_{ij}\right]  \left[  Y_{ij}\right]
\text{ iff }\left[  X_{ij}\right]  \times\left[  Y_{ij}\right]  \in
\overline{\mathbf{C}}\times\overline{\mathbf{G}}\cup\overline{\mathbf{G}
}\times\overline{\mathbf{G}}\cup\overline{\mathbf{G}}\times\overline
{\mathbf{C}}\,\text{ is a BTP},
\]
and let $F^{{\operatorname*{pre}}}\left(  \Theta,\mathfrak{F},\Theta\right)
=\left\langle \mathcal{G}^{{\operatorname*{pre}}}\diagup\sim\right\rangle .$
The \textbf{free relative prematrad generated by }$\Theta$ \textbf{and}
$\mathfrak{F}$ is the relative prematrad
\[
(F^{{\operatorname*{pre}}}(\Theta),F^{{\operatorname*{pre}}}\left(
\Theta,\mathfrak{F},\Theta\right)  ,F^{{\operatorname*{pre}}}(\Theta
),\lambda^{{\operatorname*{pre}}},\rho^{{\operatorname*{pre}}}),
\]
where $\lambda^{{\operatorname*{pre}}}$ and $\rho^{{\operatorname*{pre}}}$ are juxtaposition.
\end{definition}

\begin{example}
The module $F_{2,2}^{{\operatorname*{pre}}}(\Theta,\mathfrak{F},\Theta)$
contains 25 module generators, namely, the indecomposable $\mathfrak{f}
_{2}^{2}$ and the following $\left(  \lambda^{{\operatorname*{pre}}}
,\rho^{{\operatorname*{pre}}}\right)  $-decomposables:\newline Two of the form
$BCA$:
\begin{equation}
\left[  \theta_{1}^{2}\right]  \left[  \mathfrak{f}_{1}^{1}\right]  \left[
\theta_{2}^{1}\right]  \text{ and }\left[
\begin{array}
[c]{r}
\theta_{2}^{1}\medskip\\
\theta_{2}^{1}
\end{array}
\right]  \left[
\begin{array}
[c]{cc}
\mathfrak{f}_{1}^{1}\medskip & \mathfrak{f}_{1}^{1}\\
\mathfrak{f}_{1}^{1} & \mathfrak{f}_{1}^{1}
\end{array}
\right]  \left[
\begin{array}
[c]{ll}
\theta_{1}^{2} & \theta_{1}^{2}
\end{array}
\right]  . \label{relinner}
\end{equation}
Eleven of the form $CA_{k}\cdots A_{1}$:
\[
\left[  \mathfrak{f}_{1}^{2}\right]  \left[  \theta_{2}^{1}\right]  ,\text{
}\left[
\begin{array}
[c]{c}
\mathfrak{f}_{1}^{1}\medskip\\
\mathfrak{f}_{1}^{1}
\end{array}
\right]  \left[  \theta_{2}^{2}\right]  ,\text{ }\left[
\begin{array}
[c]{c}
\mathfrak{f}_{1}^{1}\medskip\\
\mathfrak{f}_{1}^{1}
\end{array}
\right]  \left[  \theta_{1}^{2}\right]  \left[  \theta_{2}^{1}\right]  ,\text{
}\left[
\begin{array}
[c]{c}
\mathfrak{f}_{1}^{1}\medskip\\
\mathfrak{f}_{1}^{1}
\end{array}
\right]  \left[
\begin{array}
[c]{r}
\theta_{2}^{1}\medskip\\
\theta_{2}^{1}
\end{array}
\right]  \left[
\begin{array}
[c]{ll}
\theta_{1}^{2} & \theta_{1}^{2}
\end{array}
\right]  ,\bigskip
\]
\[
\left[
\begin{array}
[c]{c}
\mathfrak{f}_{2}^{1}\medskip\\
\mathfrak{f}_{2}^{1}
\end{array}
\right]  \left[
\begin{array}
[c]{ll}
\theta_{1}^{2} & \theta_{1}^{2}
\end{array}
\right]  ,\text{ }\left[
\begin{array}
[c]{c}
\mathfrak{f}_{2}^{1}\medskip\\
\left[  \mathfrak{f}_{1}^{1}\right]  \left[  \theta_{2}^{1}\right]
\end{array}
\right]  \left[
\begin{array}
[c]{ll}
\theta_{1}^{2} & \theta_{1}^{2}
\end{array}
\right]  ,\text{ }\left[
\begin{array}
[c]{c}
\left[  \mathfrak{f}_{1}^{1}\right]  \left[  \theta_{2}^{1}\right]  \medskip\\
\mathfrak{f}_{2}^{1}
\end{array}
\right]  \left[
\begin{array}
[c]{ll}
\theta_{1}^{2} & \theta_{1}^{2}
\end{array}
\right]  ,\bigskip\text{ }
\]
\[
\left[
\begin{array}
[c]{c}
\left[  \theta_{2}^{1}\right]  \left[
\begin{array}
[c]{cc}
\mathfrak{f}_{1}^{1} & \mathfrak{f}_{1}^{1}
\end{array}
\right]  \medskip\\
\left[  \mathfrak{f}_{1}^{1}\right]  \left[  \theta_{2}^{1}\right]
\end{array}
\right]  \left[
\begin{array}
[c]{ll}
\theta_{1}^{2} & \theta_{1}^{2}
\end{array}
\right]  ,\text{ }\left[
\begin{array}
[c]{c}
\left[  \mathfrak{f}_{1}^{1}\right]  \left[  \theta_{2}^{1}\right]  \medskip\\
\left[  \theta_{2}^{1}\right]  \left[
\begin{array}
[c]{cc}
\mathfrak{f}_{1}^{1} & \mathfrak{f}_{1}^{1}
\end{array}
\right]
\end{array}
\right]  \left[
\begin{array}
[c]{ll}
\theta_{1}^{2} & \theta_{1}^{2}
\end{array}
\right]  ,\text{ }\bigskip
\]
\[
\left[
\begin{array}
[c]{c}
\left[  \theta_{2}^{1}\right]  \left[
\begin{array}
[c]{cc}
\mathfrak{f}_{1}^{1} & \mathfrak{f}_{1}^{1}
\end{array}
\right]  \medskip\\
\mathfrak{f}_{2}^{1}
\end{array}
\right]  \left[
\begin{array}
[c]{ll}
\theta_{1}^{2} & \theta_{1}^{2}
\end{array}
\right]  ,\text{ }\left[
\begin{array}
[c]{c}
\mathfrak{f}_{2}^{1}\medskip\\
\left[  \theta_{2}^{1}\right]  \left[
\begin{array}
[c]{cc}
\mathfrak{f}_{1}^{1} & \mathfrak{f}_{1}^{1}
\end{array}
\right]
\end{array}
\right]  \left[
\begin{array}
[c]{ll}
\theta_{1}^{2} & \theta_{1}^{2}
\end{array}
\right]  .\bigskip
\]
And eleven respective duals of the form $B_{1}\cdots B_{k}C.$

\begin{example}
\label{JJprime}Recall that the bialgebra prematrad $\mathcal{H}
^{{\operatorname*{pre}}}$ has two prematrad generators $c_{1,2}$ and
$c_{2,1},$ and a single module generator $c_{n,m}$ in bidegree $\left(
m,n\right)  $ (see \cite{SU4}). Consequently, the $\mathcal{H}
^{{\operatorname*{pre}}}$-bimodule $\mathcal{JJ}^{{\operatorname*{pre}}}$ has
a single bimodule generator $\mathfrak{f}$ of bidegree $(1,1)$, and a single
module generator in each bidegree satisfying the structure relations
\[
\lambda(c_{n,m};\underset{m}{\underbrace{\mathfrak{f},\ldots,\mathfrak{f}}
})=\rho(\underset{n}{\underbrace{\mathfrak{f},\ldots,\mathfrak{f}}};c_{n,m}).
\]
More precisely, if $\Theta=\left\langle \theta_{1}^{1}=\mathbf{1,}\,\theta
_{2}^{1},\,\theta_{1}^{2}\right\rangle $ and $\mathfrak{F}=\left\langle
\mathfrak{f=f}_{1}^{1}\right\rangle ,$ then
\[
\mathcal{JJ}^{{\operatorname*{pre}}}=F^{{\operatorname*{pre}}}\left(
\Theta,\mathfrak{F},\Theta\right)  /\sim,
\]
where $u\sim u^{\prime}$ if and only if $\operatorname*{bideg}
(u)=\operatorname*{bideg}(u^{\prime}).$ A bialgebra morphism $f:A\rightarrow
B$ is the image of $\mathfrak{f}$ under a map $\mathcal{JJ}
^{{\operatorname*{pre}}}\rightarrow U_{A,B}$ of relative prematrads.
\end{example}
\end{example}

\begin{example}\!\!
Whereas $F_{1,\ast}^{\operatorname*{pre}}(\Theta)$ and $F_{\ast,1}
^{\operatorname*{pre}}(\Theta)$ can be identified with the $A_{\infty}$-operad
${\mathcal{A}}_{\infty}$\! $($see \cite{SU4}$)$, $F_{1,\ast}(\Theta,{\mathfrak{F}
},\Theta)$ and $F_{\ast,1}\!(\Theta,{\mathfrak{F}},\Theta)$ can be identified
with the $\mathcal{A}_{\infty}$-bimodule ${\mathcal{J}}_{\infty}$ whose
bimodule generators are in 1-1 correspondence with $\{{\mathfrak{f}}_{m}
^{1}\}_{m\geq1}$ \!and $\{{\mathfrak{f}}_{1}^{n}\}_{n\geq1}\!,$ respectively. Thus
an $A_{\infty}$-(co)algebra morphism $f:A\rightarrow B$ is the image of the
$\mathcal{A}_{\infty}$-bimodule generators under a map $\mathcal{J}_{\infty
}\rightarrow Hom\left(  TA,B\right)  $ $($or $\mathcal{J}_{\infty}\rightarrow
Hom\left(  A,TB\right)  $\!\! $)$ of relative prematrads.
\end{example}

When $\Theta=\left\langle \theta_{m}^{n}\neq0\mid\theta_{1}^{1}=\mathbf{1}
\right\rangle _{m,n\geq1}$, and $\mathfrak{F}=\left\langle \mathfrak{f}_{m}
^{n}\neq0\right\rangle _{m,n\geq1},$ the canonical projections $\varrho
_{\Theta}^{{\operatorname*{pre}}}:F^{{\operatorname*{pre}}}(\Theta
)\rightarrow\mathcal{H}^{{\operatorname*{pre}}}$ and $\varrho
^{{\operatorname*{pre}}}:F^{{\operatorname*{pre}}}(\Theta,\mathfrak{F}
,\Theta)\rightarrow\mathcal{JJ}^{{\operatorname*{pre}}}$ give a map $\left(
\varrho_{\Theta}^{{\operatorname*{pre}}},\varrho^{{\operatorname*{pre}}
},\varrho_{\Theta}^{{\operatorname*{pre}}}\right)  $ of relative prematrads.
If $\partial^{{\operatorname*{pre}}}$ is a differential on
$F^{{\operatorname*{pre}}}(\Theta,\mathfrak{F},\Theta)$ such that
$\varrho^{{\operatorname*{pre}}}$ is a free resolution in the category of
relative prematrads, the induced isomorphism $\varrho^{{\operatorname*{pre}}
}:H_{\ast}\left(  F^{{\operatorname*{pre}}}(\Theta,\mathfrak{F},\Theta
),\partial^{{\operatorname*{pre}}}\right)  \approx\mathcal{JJ}
^{{\operatorname*{pre}}}$ implies
\[
\begin{array}
[c]{c}
\partial^{{\operatorname*{pre}}}(\mathfrak{f}_{1}^{1})=0\\
\\
\partial^{{\operatorname*{pre}}}(\mathfrak{f}_{2}^{1})=\rho(\mathfrak{f}
_{1}^{1};\theta_{2}^{1})-\lambda(\theta_{2}^{1};\mathfrak{f}_{1}
^{1},\mathfrak{f}_{1}^{1})\\
\\
\partial^{{\operatorname*{pre}}}(\mathfrak{f}_{1}^{2})=\rho(\mathfrak{f}
_{1}^{1},\mathfrak{f}_{1}^{1};\theta_{1}^{2})-\lambda(\theta_{1}
^{2};\mathfrak{f}_{1}^{1}).
\end{array}
\]
This gives rise to the standard isomorphisms
\[
\begin{array}
[c]{ccccccc}
F_{1,2}^{{\operatorname*{pre}}}(\Theta,\mathfrak{F},\Theta)\smallskip & = &
\left\langle \left[  \theta_{2}^{1}\right]  \right.  \left[
\begin{array}
[c]{cc}
\mathfrak{f}_{1}^{1} & \mathfrak{f}_{1}^{1}
\end{array}
\right]  & , & \left[  \mathfrak{f}_{2}^{1}\right]  & , & \left.  \left[
\mathfrak{f}_{1}^{1}\right]  \left[  \theta_{2}^{1}\right]  \right\rangle \\
\approx\text{ }\updownarrow\text{ \ }\smallskip &  & \updownarrow &  &
\updownarrow &  & \updownarrow\\
C_{\ast}(J_{2})\smallskip & = & \left\langle 1|2\right.  & , & 12 & , &
\left.  2|1\right\rangle \\
\approx\text{ }\updownarrow\text{ \ }\medskip &  & \updownarrow &  &
\updownarrow &  & \updownarrow\\
F_{2,1}^{{\operatorname*{pre}}}(\Theta,\mathfrak{F},\Theta) & = & \left\langle
\left[  \theta_{1}^{2}\right]  \left[  \mathfrak{f}_{1}^{1}\right]  \right.  &
, & \left[  \mathfrak{f}_{1}^{2}\right]  & , & \left[
\begin{array}
[c]{c}
\mathfrak{f}_{1}^{1}\medskip\\
\mathfrak{f}_{1}^{1}
\end{array}
\right]  \left.  \left[  \theta_{1}^{2}\right]  \right\rangle
\end{array}
\]
(see Figure 3). A similar application of $\partial^{{\operatorname*{pre}}}$ to
$\mathfrak{f}_{1}^{n}$ and $\mathfrak{f}_{m}^{1}$ gives the isomorphisms
\begin{equation}
F_{n,1}^{{\operatorname*{pre}}}(\Theta,\mathfrak{F},\Theta)\overset{\approx
}{\longrightarrow}\mathcal{J}_{\infty}(n)=C_{\ast}(J_{n}) \label{iso1}
\end{equation}
and
\begin{equation}
F_{1,m}^{{\operatorname*{pre}}}(\Theta,\mathfrak{F},\Theta)\overset{\approx
}{\longrightarrow}\mathcal{J}_{\infty}(m)=C_{\ast}(J_{m}) \label{iso2}
\end{equation}
(see \cite{Stasheff}, \cite{Stasheff2}, \cite{SU2}). As in the absolute case,
there is a differential $\partial$ on a canonical proper submodule
$\mathcal{J}\mathcal{J}_{\infty}\subset{F^{{\operatorname*{pre}}}
(\Theta,\mathfrak{F},\Theta)}$ such that the canonical projection
$\varrho:\mathcal{J}\mathcal{J}_{\infty}\rightarrow\mathcal{JJ}
^{{\operatorname*{pre}}}$ is a free resolution in the category of
\textquotedblleft relative matrads.\textquotedblright

\begin{definition}
Given local prematrads $\left(  {M},\gamma_{\mathbf{W}_{M}},\eta_{M}\right)  $
and $\left(  {N},\gamma_{\mathbf{W}_{N}},\eta_{N}\right)  ,$ a bigraded
$R$-module $E=\left\{  E_{n,m}\right\}  _{m,n\geq1},$ a telescoping submodule
$\mathbf{W}_{E}\subseteq\mathbf{V}_{E}$, and actions $\lambda=\{\lambda
_{\mathbf{x}}^{\mathbf{y}}:\left(  \mathbf{W}_{M}\right)  _{p}^{\mathbf{y}
}\otimes\left(  \mathbf{W}_{E}\right)  _{\mathbf{x}}^{q}\rightarrow\left(
\mathbf{W}_{E}\right)  _{\left\vert \mathbf{x}\right\vert }^{\left\vert
\mathbf{y}\right\vert }\}$ and $\rho=\{\rho_{\mathbf{x}}^{\mathbf{y}}:\left(
\mathbf{W}_{E}\right)  _{p}^{\mathbf{y}}\otimes\left(  \mathbf{W}_{N}\right)
_{\mathbf{x}}^{q}\rightarrow\left(  \mathbf{W}_{E}\right)  _{\left\vert
\mathbf{x}\right\vert }^{\left\vert \mathbf{y}\right\vert }\},$ the tuple
$\left(  {M},{E},{N},\lambda,\rho\right)  $ is a \textbf{relative local
prematrad} \textbf{with domain} $\left(  \mathbf{W}_{M},\mathbf{W}
_{E},\mathbf{W}_{N}\right)  $ if

\begin{enumerate}
\item[\textit{(i)}] $\left(  \mathbf{W}_{E}\right)  _{\mathbf{x}}
^{1}=\mathbf{E}_{\mathbf{x}}^{1}$ and $\left(  \mathbf{W}_{E}\right)
_{1}^{\mathbf{y}}=\mathbf{E}_{1}^{\mathbf{y}}$ for all $\mathbf{x,y};$

\item[\textit{(ii)}] $\Upsilon_{M},\Upsilon_{N},\Upsilon_{\lambda}
,\Upsilon_{\rho}$ interact associatively on $\mathcal{W}_{E}\cap\mathbf{E};$

\item[\textit{(iii)}] the relative prematrad unit axiom holds for $\lambda$
and $\rho$.
\end{enumerate}
\end{definition}

Let
\[
\overline{\mathbf{E}}_{{\operatorname*{row}}}=\bigoplus_{\bar{X}
,Y\in\mathbb{N}^{q\times p};\text{ }p,q\geq1}E_{Y,\bar{X}}\text{
\ and\ \ }\overline{\mathbf{E}}^{{\operatorname*{col}}}=\bigoplus_{X,\bar
{Y}\in\mathbb{N}^{q\times p};\text{ }p,q\geq1}E_{\bar{Y},X},
\]
where $\bar{X}$ has constant columns and $\bar{Y}$ has constant rows. If $A$
is a $q\times p$ monomial in $E_{Y,\bar{X}}$, we refer to a row $\mathbf{x}$
of $\bar{X}$ as the \emph{row leaf sequence} of $A$, and write
${{\operatorname*{rls}}}(A)=\mathbf{x};$ if $B$ is a $q\times p$ monomial in
$E_{\bar{Y},X}$, we refer to a column $\mathbf{y}$ of $\bar{Y}$ as the
\emph{column leaf sequence} of $B$, and write ${{\operatorname*{cls}}
}(B)=\mathbf{y}.$ Given a telescoping submodule $\mathbf{W\subseteq V}_{E}$
and its telescopic extension $\mathcal{W}_{E},$ let $\left(  \mathcal{W}
_{E}\right)  _{{\operatorname*{row}}}=\mathcal{W}_{E}\cap\overline{\mathbf{E}
}_{{\operatorname*{row}}}$ and $\left(  \mathcal{W}_{E}\right)
^{{\operatorname*{col}}}=\mathcal{W}_{E}\cap\overline{\mathbf{E}
}^{{\operatorname*{col}}}.$

\begin{definition}
Let $(M,E,N,\lambda,\rho)$ be a relative local prematrad with domain $\left(
\mathbf{W}_{M},\right.  $ $\left.  \mathbf{W}_{E},\mathbf{W}_{N}\right)  ,$
let $m,n\geq1,$ let $\mathfrak{f}\in E_{\ast,m}$, and let $\mathfrak{g}\in
E_{n,\ast}$.

\begin{enumerate}
\item[\textit{(i)}] A \textbf{row factorization of} $\mathfrak{f}$
\textbf{with respect to} $\left(  \mathbf{W}_{M},\!\mathbf{W}_{E},\!\mathbf{W}
_{N}\right)  $ is an $(\Upsilon_{\!\lambda},\!\!\Upsilon_{\rho})$-factorization
\[
A_{1}\cdots A_{s}\cdot C\cdot A_{s+1}^{\prime}\cdots A_{k}^{\prime
}=\mathfrak{f}
\]
such that $A_{i}\!\in\!(\mathcal{W}_{M})_{{\operatorname*{row}}},$ $C\!\in\!
(\mathcal{W}_{E})_{{\operatorname*{row}}},$ $A_{j}^{\prime}\!\in\!(\mathcal{W}
_{N})_{{\operatorname*{row}}}$ and $\operatorname*{rls}({A}_{i}
),\operatorname*{rls}({A}_{j}^{\prime})\neq\mathbf{1}$ for all $i,j.$ The
$(m+1,k+1)$-\textbf{row descent sequence} of $\mathfrak{f}$ is the $\left(
k+1\right)  $-tuple of row vectors
\[
\mathbf{m=}({\operatorname*{rls} }(A_{1}
),...,{\operatorname*{rls}}(A_{s}),\left(  1,0,...,0\right)
+{\operatorname*{rls}}(C),\left(  1,{\operatorname*{rls}}(A_{s+1}^{\prime
})\right)  ,...,\left(  1,{\operatorname*{rls}}(A_{k}^{\prime})\right)  ).
\]

\item[\textit{(ii)}] \textit{A }\textbf{column factorization of }
$\mathfrak{g}$\textbf{\ with respect to }$\left(  \mathbf{W}_{M}
,\mathbf{W}_{E},\mathbf{W}_{N}\right)  $\textit{\ is an }$(\Upsilon_{\lambda
},\Upsilon_{\rho})$\textit{-factorization}
\[
B_{l}^{\prime}\cdots B_{t+1}^{\prime}\cdot D\cdot B_{t}\cdots B_{1}
=\mathfrak{g}
\]
\textit{such that} $B_{i}^{\prime}\in(\mathcal{W}_{M})^{{\operatorname*{col}}
},$ $D\in(\mathcal{W}_{E})^{{\operatorname*{col}}},$ $B_{j}\in(\mathcal{W}
_{N})^{{\operatorname*{col}}}$ and $\operatorname*{cls}(B_{i}^{\prime
}),\operatorname*{cls}(B_{j})\neq\mathbf{1}$ for all $i,j.$ The $\left(
n+1,l+1\right)  $-\textbf{column descent sequence of} $\mathfrak{g}$ is the
$(l+1)$-tuple of column vectors
\[
\mathbf{n}=((\operatorname*{cls} (B_{l}^{\prime}),1)^{T}
\!\!,\!...,(\operatorname*{cls}(B_{t+1}^{\prime} ),1)^{T}
\!\!,\operatorname*{cls}(D)+(0,\!...,0,1)^{T}\!\!,\operatorname*{cls}
(B_{t}),\!...,\operatorname*{cls}(B_{1})).
\]

\end{enumerate}
\end{definition}

\noindent Column and row factorizations are not unique. Note that
$\mathfrak{f}\in E_{n,m}$ always has trivial row and column factorizations as
a $1\times1$ matrix $C=\left[  \mathfrak{f}\right]  $.

Given elements $\mathfrak{f}\in E_{\ast,m}\,$and $\mathfrak{g}\in E_{n,\ast}$
with $m,n\geq1,$ choose a row factorization $A_{1}\cdots A_{s}\cdot C\cdot
A_{s+1}^{\prime}\cdots A_{k}^{\prime}$ of $\mathfrak{f}$, and a column
factorization $B_{l}^{\prime}\cdots B_{t+1}^{\prime}\cdot D\cdot B_{t}\cdots
B_{1}$ of $\mathfrak{g}.$ The related row descent sequence $\mathbf{m}$
identifies $\mathfrak{f}$ with an up-rooted $\left(  m+1\right)  $-leaf,
$\left(  k+1\right)  $-level PLT, and hence with a codimension $k$ face
$\overset{_{\wedge}}{e}_{\mathfrak{f}}$ of $P_{m}\mathbf{.}$ Dually, the
related column descent sequence $\mathbf{n}$ identifies $\mathfrak{g}$ with a
down-rooted $\left(  n+1\right)  $-leaf, $\left(  l+1\right)  $-level PLT, and
hence with a codimension $l$ face $\overset{_{\vee}}{e}_{\mathfrak{g}}$ of
$P_{n}.$ Extending to Cartesian products, identify the monomials
$F=\mathfrak{f}_{1}\otimes\cdots\otimes\mathfrak{f}_{q}\in\left(  E_{\ast
,m}\right)  ^{\otimes q}$ and $G=\mathfrak{g}_{1}\otimes\cdots\otimes
\mathfrak{g}_{p}\in\left(  E_{n,\ast}\right)  ^{\otimes p}$ with the product
cells
\[
\overset{_{\wedge}}{e}_{F}=\overset{_{\wedge}}{e}_{\mathfrak{f}_{1}}
\times\cdots\times\overset{_{\wedge}}{e}_{\mathfrak{f}_{q}}\subset
P_{m}^{\times q}\ \ \text{and}\ \ \overset{_{\vee}}{e}_{G}=\overset{_{\vee
}}{e}_{\mathfrak{g}_{1}}\times\cdots\times\overset{_{\vee}}{e}_{\mathfrak{g}
_{p}}\subset P_{n}^{\times p}.
\]

As in the absolute case reviewed in Section 2, the product cell
$\overset{\wedge}{e}_{F}$ either\emph{\ is} or \emph{is not} a subcomplex of
${\Delta}^{(q-1)}(P_{m})\subset P_{m}^{\times q}$, and dually for
$\overset{\vee}{e}_{G}.$ Recall that $\mathbf{x\times y}\in\mathbb{N}^{1\times
p}\times\mathbb{N}^{q\times1}$, and let
\[
r\Gamma({E})=E\oplus\bigoplus_{\mathbf{x,y\notin}\mathbb{N};\text{ }s,t\geq
1}r\Gamma_{s}^{\mathbf{y}}({E})\oplus r\Gamma_{\mathbf{x}}^{t}({E}),
\]
where
\begin{align*}
r\Gamma_{s}^{\mathbf{y}}({E})  &  =\left\langle F\in\mathbf{E}_{s}
^{\mathbf{y}}\mid\overset{_{\wedge}}{e}_{F}\subset{\Delta}^{\left(
q-1\right)  }\left(  P_{s}\right)  \right\rangle ,\\
r\Gamma_{\mathbf{x}}^{t}({E})  &  =\left\langle G\in\mathbf{E}_{\mathbf{x}
}^{t}\mid\overset{_{\vee}}{e}_{G}\subset{\Delta}^{\left(  p-1\right)  }\left(
P_{t}\right)  \right\rangle .
\end{align*}

\begin{definition}
The \textbf{(left) configuration module} of a relative local prematrad
$({M},{E},{N},\lambda,\rho)$ with domain $(\mathbf{W}_{M},\mathbf{W}
_{E},\mathbf{W}_{N})$ is the triple
\[
\left(  \Gamma({M,\gamma}_{\mathbf{w}_{M}}),\text{ }r\Gamma({E}),\text{
}\Gamma({N,\gamma}_{\mathbf{W}_{N}})\right)  .
\]

\end{definition}

When $s=t=1$, the arguments in the absolute case carry over verbatim and give
\[
\bigoplus_{\mathbf{x}}r{\Gamma}_{\mathbf{x}}^{1}(E)=T^{+}({E}_{1,\ast})\text{
\ and \ }\bigoplus_{\mathbf{y}}r{\Gamma}_{1}^{\mathbf{y}}({E})=T^{+}({E}
_{\ast,1}).
\]

\begin{definition}
Let $\left(  M,\gamma_{\mathbf{w}_{M}}\right)  $ and $\left(  N,\gamma
_{\mathbf{w}_{N}}\right)  $ be (left) matrads. A relative local prematrad
$({M},{E},{N},\lambda,\rho)$ with domain $\left(  \mathbf{W}_{M}
,\mathbf{W}_{E},\mathbf{W}_{N}\right)  $ is a \textbf{relative
(left)\thinspace matrad} if $r\Gamma({E})=\mathbf{W}_{E}.$ When $M=N$ we refer
to $r\Gamma({E})$ as a $\Gamma(M)$\textbf{-bimodule}. A \textbf{morphism of
relative matrads} is a map of underlying relative local prematrads.
\end{definition}

A relative prematrad $(M,E,N,\rho,\lambda)$ with domain $\left(
\mathbf{W}_{M},\mathbf{W}_{E},\mathbf{W}_{N}\right)  $ restricts to a relative
matrad structure with domain $(\Gamma(M),r\Gamma(E),\Gamma(N)).$

\begin{example}
\textbf{The Bialgebra Morphism Matrad} $\mathcal{J}{\mathcal{J}}.$ The
$\mathcal{H}^{{\operatorname*{pre}}}$-bimodule $\mathcal{J}{\mathcal{J}
}^{{\operatorname*{pre}}}$ discussed in Example \ref{JJprime} satisfies
\[
r\Gamma_{p}^{\mathbf{y}}(E)\otimes\Gamma_{\mathbf{x}}^{q}(M)=\mathbf{E}
_{p}^{\mathbf{y}}\otimes\mathbf{M}_{\mathbf{x}}^{q}\ \ \text{and}
\ \ \Gamma_{s}^{\mathbf{v}}(M)\otimes r\Gamma_{\mathbf{u}}^{t}(E)=\mathbf{M}
_{s}^{\mathbf{v}}\otimes\mathbf{E}_{\mathbf{u}}^{t}.
\]
Consequently, $\mathcal{J}{\mathcal{J}}^{{\operatorname*{pre}}}$ is also a
relative matrad, called \textbf{the bialgebra morphism matrad}, and is
henceforth denoted by $\mathcal{J}{\mathcal{J}}.$
\end{example}

\begin{definition}
Let $\Theta=\left\langle \theta_{m}^{n}\mid\theta_{1}^{1}=\mathbf{1}
\right\rangle _{m,n\geq1}$ and $\mathfrak{F}=\left\langle \mathfrak{f}_{m}
^{n}\right\rangle _{m,n\geq1}.$ Let $F^{{\operatorname*{pre}}}
=F^{{\operatorname*{pre}}}(\Theta)$, and consider the free
$F^{{\operatorname*{pre}}}$-bimodule $E^{{\operatorname*{pre}}}
=F^{{\operatorname*{pre}}}\left(  \Theta,\mathfrak{F},\Theta\right)  .$ The
\textbf{free relative matrad generated by }$(\Theta,\mathfrak{F})$
\textbf{with domain} $(\mathbf{W}_{F(\Theta)},\mathbf{W}_{E},\mathbf{W}
_{F(\Theta)})$ is the tuple $\left(  F\left(  \Theta\right)  ,F\left(
\Theta,\mathfrak{F},\Theta\right)  ,F\left(  \Theta\right)  ,\lambda
,\rho\right)  $, where
\[
F\left(  \Theta,\mathfrak{F},\Theta\right)  =\lambda^{{\operatorname*{pre}}
}\left[  \Gamma_{p}^{\mathbf{y}}(F^{{\operatorname*{pre}}});r\Gamma
_{\mathbf{x}}^{q}(E^{{\operatorname*{pre}}})\right]  \oplus\rho
^{{\operatorname*{pre}}}\left[  r\Gamma_{p}^{\mathbf{y}}
(E^{{\operatorname*{pre}}});\Gamma_{\mathbf{x}}^{q}(F^{{\operatorname*{pre}}
})\right]  ,
\]
\[
\lambda=\lambda^{{\operatorname*{pre}}}|_{_{\mathbf{W}_{F(\Theta)}
\otimes\mathbf{W}_{E}}},\text{ \ }\rho=\rho^{{\operatorname*{pre}}
}|_{_{\mathbf{W}_{E}\otimes\mathbf{W}_{F(\Theta)}}},\text{ \ and}
\]
\[
\mathbf{W}_{E}=\bigoplus_{p,q\in\mathbb{N};\text{ }\mathbf{x,y\notin
}\mathbb{N}}F_{q,p}(\Theta,\mathfrak{F},\Theta)\oplus\Gamma_{p}^{\mathbf{y}
}(E^{{\operatorname*{pre}}})\oplus\Gamma_{\mathbf{x}}^{q}
(E^{{\operatorname*{pre}}}).
\]

\end{definition}

For example, the monomials in (\ref{relinner}) are two of the 17 module
generators in $F_{2,2}\left(  \Theta,\mathfrak{F},\Theta\right)  $ identified
with the faces of the octagon $JJ_{2,2}$ (see Figure 4). In general,
$F_{1,m}\left(  \Theta,\mathfrak{F},\Theta\right)  =F_{1,m}
^{{\operatorname*{pre}}}\left(  \Theta,\mathfrak{F},\Theta\right)  $, and
$F_{n,1}\left(  \Theta,\mathfrak{F},\Theta\right)  =F_{n,1}
^{{\operatorname*{pre}}}\left(  \Theta,\mathfrak{F},\Theta\right)  $ for all
$m,n\geq1.$

\begin{example}
\textbf{The }$A_{\infty}$\textbf{-bialgebra Morphism Matrad.} Let
$\Theta=\! \left\langle \theta_{m}^{n}|
\theta_{1}^{1}\!=\!\mathbf{1}\right\rangle _{\!m,n\geq1}\!,$ and $\mathfrak{F}
=\left\langle \mathfrak{f}_{m}^{n}\right\rangle _{m,n\geq1},$ where
$\left\vert \mathfrak{f}_{m}^{n}\right\vert =m+n-2$. Following the
construction in the absolute case, let $r\mathcal{C}$ be the set indexing the
module generators $F(\Theta,\mathfrak{F},\Theta)\mathfrak{,}$ and let
$\mathcal{A}\mathcal{R}_{m,n}\subset\mathcal{C}_{m,n}\times r\mathcal{C}
_{m,n}$ and $\mathcal{Q}\mathcal{B}_{m,n}\subset r\mathcal{C}_{m,n}
\times\mathcal{C}_{m,n}$ be the subsets of $r\mathcal{C}$ that index the
codimension $1$ elements of $F_{n,m}(\Theta,\mathfrak{F},\Theta)$ of the form
$\lambda(-;-)$ and $\rho(-;-),$ respectively. Let $\left\{  ({A}
_{p}^{\mathbf{y}})_{\alpha}\right\}  $ and $\left\{  ({B}_{\mathbf{x}}
^{q})_{\beta}\right\}  $ be the bases defined in Example 19 of \cite{SU4}, and
let $\left\{  ({Q}_{s}^{\mathbf{\ v}})_{\nu}\right\}  _{\nu\in\mathcal{Q}
_{s}^{\mathbf{v}}}$ and $\left\{  ({R}_{\mathbf{u}}^{t})_{\mu}\right\}
_{\mu\in\mathcal{R}_{\mathbf{u}}^{t}}$ be the analogous bases for $r{\Gamma
}_{s}^{\mathbf{v}}(E)$ and $r{\Gamma}_{\mathbf{u}}^{t}(E)$ in dimensions
$|\mathbf{v}|+s-t$ and $|\mathbf{u}|+t-s;$ then each $\overset{\wedge
}{e}_{Q_{\nu}}$ is a subcomplex of $\Delta^{\left(  t-1\right)  }\left(
P_{s}\right)  $ with associated sign $\left(  -1\right)  ^{\epsilon_{\nu}}$
and each $\overset{\vee}{e}_{R_{\mu}}$ is a subcomplex of $\Delta^{\left(
s-1\right)  }\left(  P_{t}\right)  $ with associated sign $\left(  -1\right)
^{\epsilon_{\mu}}.$ Now consider codimension $1$ face $e_{\left(
\mathbf{y,x}\right)  }=C|D\subset P_{m+n-2}$ defined as follows: If
$\left\vert \mathbf{x}\right\vert =m>p\geq2,$ let $A_{\mathbf{x}
}|B_{\mathbf{x}}$ be the codimension $1$ face of $P_{m-1}$ with leaf sequence
$\mathbf{x};$ dually, if $\left\vert \mathbf{y}\right\vert =n>q\geq2,$ let
$A_{\mathbf{y}}|B_{\mathbf{y}}$ be the codimension $1$ face of $P_{n-1}$ with
leaf sequence $\mathbf{y}$. If $A=\left\{  a_{1},\ldots,a_{r}\right\}
\subset\mathbb{Z}$ and $z\in\mathbb{Z}$, define $-A=\left\{  -a_{1}
,\ldots,-a_{n}\right\}  $ and $A+z=\left\{  a_{1}+z,\ldots,a_{r}+z\right\}  ;$
then set
\[
\begin{array}
[c]{ll}
A_{1}=\left\{
\begin{array}
[c]{cl}
\underline{m-1}, & \text{if }\mathbf{x}=\mathbf{1}^{m},\,m\geq1\medskip\\
\varnothing, & \text{if }\mathbf{x}=m\geq2\medskip\\
-B_{\mathbf{x}}+m, & \text{otherwise,}
\end{array}
\right.  & A_{2}=\left\{
\begin{array}
[c]{cl}
\varnothing, & \text{if }\mathbf{y}=\mathbf{1}^{n},\,n\geq1\medskip\\
\underline{n-1}, & \text{if }\mathbf{y}=n\geq2\medskip\\
A_{\mathbf{y}}, & \text{otherwise,}
\end{array}
\right. \\
& \\
B_{1}=\left\{
\begin{array}
[c]{cl}
\varnothing, & \text{if }\mathbf{x}=\mathbf{1}^{m},\,m\geq1\medskip\\
\underline{m-1}, & \text{if }\mathbf{x}=m\geq2\medskip\\
-A_{\mathbf{x}}+m, & \text{otherwise,}
\end{array}
\right.  & B_{2}=\left\{
\begin{array}
[c]{cl}
\underline{n-1}, & \text{if }\mathbf{y}=\mathbf{1}^{n},\,n\geq1\medskip\\
\varnothing, & \text{if }\mathbf{y}=n\geq2\medskip\\
B_{\mathbf{y}}, & \text{otherwise,}
\end{array}
\right.
\end{array}
\]
and define
\begin{equation}
\label{cell}e_{(\mathbf{y},\mathbf{x})}=A_{1}\cup(A_{2}+m-1)\mid B_{1}
\cup(B_{2}+m-1)
\end{equation}
(this corrects the analogous formula in line (13) of \cite{SU4} in which the
symbols $A_{\mathbf{x}}$ and $B_{\mathbf{x}}$ are reversed). Define a
differential $\partial:F(\Theta,\mathfrak{F},\Theta)\rightarrow F(\Theta
,\mathfrak{F} ,\Theta)$ of degree $-1$ on generators by
\begin{equation}
\begin{array}
[c]{r}
\partial(\mathfrak{f}_{m}^{n})=\hspace*{-0.1in}\sum\limits_{\substack{(\alpha
,\mu)\in\mathcal{A}\mathcal{R}_{m,n}\\\text{ }}}\hspace*{-0.2in}
(-1)^{\epsilon_{1}+\epsilon_{\alpha}+\epsilon_{\mu}}\lambda\left[  ({A}
_{p}^{\mathbf{y}})_{\alpha};({R}_{\mathbf{x}}^{q})_{\mu}\right]  \hspace
*{1in}\\
\hspace*{0.8in}+\sum\limits_{\substack{(\nu,\beta)\in\mathcal{Q}
\mathcal{B}_{m,n}\\\text{ }}}\hspace*{-0.2in}(-1)^{\epsilon_{2}+\epsilon_{\nu
}+\epsilon_{\beta}}\rho\left[  ({Q}_{p}^{\mathbf{y}})_{\nu};({B}_{\mathbf{x}
}^{q})_{\beta}\right]  ,
\end{array}
\label{reldiff}
\end{equation}
where $(-1)^{\epsilon_{1}}$ is the sign associated with $e_{(\mathbf{y}
,\mathbf{x})}$, and $\epsilon_{2}=\#D+\epsilon_{1}+1$. Extend $\partial$ as a
derivation of $\rho$ and $\lambda$; then $\partial^{2}=0$ follows from the
associativity of $\rho$ and $\lambda$. The $\mathcal{A}_{\infty}
$-\textbf{bialgebra morphism matrad is the }DG $\mathcal{H}_{\infty}$-bimodule
$\mathcal{J}\mathcal{J}_{\infty}=\left(  F\left(  \Theta,\mathfrak{F}
,\Theta\right)  ,\partial\right)  $\textbf{.}
\end{example}

We realize the $\mathcal{A}_{\infty}$-bialgebra morphism matrad by the
cellular chains of a new family of polytopes $JJ=\sqcup_{m,n\geq1}JJ_{n,m}$,
called \emph{bimultiplihedra}, to be constructed in the next section. The
standard isomorphisms (\ref{iso1}) and (\ref{iso2}) extend to isomorphisms
\begin{equation}
(\mathcal{J}\mathcal{J}_{\infty})_{n,m}\overset{\approx}{\longrightarrow
}C_{\ast}(JJ_{n,m}), \label{JJ}
\end{equation}
and one recovers $\mathcal{J}_{\infty}$ by restricting the differential
$\partial$ to $(\mathcal{J}\mathcal{J}_{\infty})_{1,\ast}$ or $(\mathcal{J}
\mathcal{J}_{\infty})_{\ast,1}.$

\section{Bimultiplihedra}

In this section we construct the bimultiplihedron $JJ_{n,m}$ as a subdivision
of the cylinder $KK_{n,m}\times I.$ Whereas our construction of $KK_{n+1,m+1}$
uses the combinatorics of $P_{m+n},$ our construction of $JJ_{n+1,m+1}$ uses
the combinatorics of $P_{m+n+1}$ thought of as a subdivision of $P_{m+n}\times
I$ (see \cite{SU2}), or equivalently, as the combinatorial join $P_{m+n+1}
=P_{m+n}\ast_{c}P_{1}$, which corresponds algebraically to adjoining the
parameter $\mathfrak{f}_{1}^{1}$ to the 0-dimensional elements of the free
matrad $\mathcal{H}_{\infty}$ (recall that $\dim P_{m+n+1}=\dim(P_{m+n}
\ast_{c}P_{1})=\dim P_{m+n}+1;$ see \cite{SU4}). Indeed, our construction of
$JJ_{n+1,m+1}$ as the geometric realization of a poset $r{\mathcal{PP}}
_{n,m}\diagup\sim$ resembles our construction of $KK_{n+1,m+1}=\left\vert
{\mathcal{PP}}_{n,m}\diagup\sim\right\vert $ in \cite{SU4}, but with some
technical differences.

Recall that faces of the permutahedron $P_{n}$ are indexed by up-rooted (or
down-rooted) PLTs. In particular, the vertices of $P_{n}$ are indexed by the
set $\wedge_{n}$ ($\vee_{n}$) of all up-rooted (down-rooted) planar binary
trees with $n+1$ leaves and $n$ levels. Since vertices of $P_{n}$ are
identified with permutations of $\underline{n}=\left\{  1,2,\ldots,n\right\}
,$ the Bruhat partial ordering, generated by $a_{1}|\cdots|a_{n}<a_{1}
|\cdots|a_{i+1}|a_{i}|\cdots|a_{n}$ if and only if $a_{i}<a_{i+1},$ induces
natural poset structures on $\wedge_{n}$ and $\vee_{n}.$

In \cite{SU4} we introduced the subposet $X_{m}^{n}\subseteq\wedge_{m}^{\times
n},$ which indexes the vertices of the subcomplex $\Delta^{\left(  n-1\right)
}\left(  P_{m}\right)  \subseteq P_{m}^{\times n};$ an element $x\in X_{m}
^{n}$ is expressed as a column matrix $x=[\hat{T}_{1}\cdots\hat{T}_{n}]^{T}$
of $n$ up-rooted binary trees with $m+1$ leaves and $m$ levels. Let
$\curlywedge$ ($\curlyvee$) denote the up-rooted (down-rooted) $2$-leaf
corolla. Now if $a\in\wedge_{n},$ there exists a unique BTP\ $\Upsilon
$-factorization $a=a_{1}\cdots a_{n}$ such that\ $a_{j}$ is a $1\times j$ row
matrix over $\left\{  \mathbf{1,}\curlywedge\right\}  $ containing the entry
$\curlywedge$ exactly once. Thus by factoring each $\hat{T}_{i},$ we obtain
the BTP\ $\Upsilon$-factorization $x=x_{1}\cdots x_{m}$ in which $x_{j}$ is an
$n\times j$ matrix. Below we distinguish two matrices $x_{\alpha}$ and
$x_{\beta}$ that correspond to the first and last matrices of a $0$
-dimensional element of $\mathcal{JJ}_{\infty}$ containing the entry
$\mathfrak{f}_{1}^{1}$ as follows.

Let $\varrho\left(  x,i\right)  $ be the \emph{lowest level} of $\hat{T}_{i}$
in which a branch is attached on the extreme left. Given positive integers
$\alpha\leq\beta,$ consider the set of \emph{bitruncated elements}
\begin{align*}
X_{m}^{n}(\alpha,\beta)  &  =\left\{  _{\alpha}x_{\beta}=x_{\alpha}\cdots
x_{\beta}\mid x=x_{1}\cdots x_{\alpha}\cdots x_{\beta}\cdots x_{m}\in{X}
_{m}^{n},\right. \\
&  \alpha=\min_{1\leq i\leq n}\varrho(x,i)\text{ and }\beta=\max_{1\leq i\leq
n}\varrho(x,i)\}.
\end{align*}
Define $_{\alpha}x_{\beta}<$ $_{\alpha^{\prime}}x_{\beta^{\prime}}^{\prime}$
if and only if $x<x^{\prime}$ for some $x=x_{1}\cdots x_{\alpha}\cdots
x_{\beta}\cdots x_{m}$ and $x^{\prime}=x_{1}^{\prime}\cdots x_{\alpha^{\prime
}}^{\prime}\cdots x_{\beta^{\prime}}^{\prime}\cdots x_{m}^{\prime}$ in
$X_{m}^{n};$ then $X_{m}^{n}(\alpha,\beta)$ is the \emph{poset of}
\emph{bitruncated elements}$.$

Now replace the entries $\mathbf{1}$ and $\curlywedge$ in each $x_{j}$ with
the integers $1$ and $2,$ respectively. Then the $\left(  m+1,m\right)  $-row
descent sequence $C_{i}=\left(  \mathbf{x}_{i,1}=\left(  2\right)
,\mathbf{x}_{i,2},\ldots,\mathbf{x}_{i,m}\right)  $ of $\hat{T}_{i}$ appears
as the $i^{th}$ rows of $x_{1},\ldots,x_{m}$, and $\varrho=\varrho(x,i)$ is
the largest integer such that $\mathbf{x}_{i,\varrho}=(2,1,...,1).$ Introduce
the decorations $\mathbf{\mathring{x}}_{i,\varrho}=({\mathring{2}}
,{\mathring{1}},...,{\mathring{1}})$ and $\mathbf{\dot{x}}_{i,j}=\left(
{\dot{1}},1,\ldots,2,1,\ldots\right)  $ for $j>\varrho,$ and obtain the
corresponding $\left(  m+1,m\right)  $-\emph{marked row descent sequence }
\[
\hat{C}_{i}=\left(  \mathbf{x}_{i,1},\ldots,\mathbf{x}_{i,\varrho
-1},\mathbf{\mathring{x}}_{i,\varrho},\mathbf{\dot{x}}_{i,\varrho+1}
,\ldots,\mathbf{\dot{x}}_{i,m}\right)  .
\]
The term $\mathbf{\mathring{x}}_{i,\varrho}$ represents the constant
$1\times\varrho$ matrix $\left[  \mathfrak{f}_{1}^{1}\cdots\mathfrak{f}
_{1}^{1}\right]  .$ Let $\mathring{x}$ denote the matrix string $x=x_{1}\cdots
x_{m}$ with decorated integer entries; then the poset of\emph{\ marked
bitruncated} \emph{elements }is
\[
\mathring{X}_{m}^{n}(\alpha,\beta)=\left\{  _{\alpha}{\mathring{x}}_{\beta
}\mid_{\alpha}x_{\beta}\in X_{m}^{n}(\alpha,\beta)\right\}  .
\]
Note that\emph{\ }$\varrho(x,i)$ is constant for all $i\ $if and only if
$\alpha=\beta,$ in which case $\mathring{X}_{m}^{n}(\alpha,\alpha)$ is a
singleton set containing the $n\times\alpha$ matrix
\[
_{\alpha}\mathring{x}_{\alpha}=\left[
\begin{array}
[c]{c}
\mathbf{\mathring{x}}_{1,\alpha}\\
\vdots\\
\mathbf{\mathring{x}}_{n,\alpha}
\end{array}
\right]  =\left[
\begin{array}
[c]{cccc}
\mathring{2} & \mathring{1} & \cdots & \mathring{1}\\
\vdots\medskip & \vdots &  & \vdots\\
\mathring{2} & \mathring{1} & \cdots & \mathring{1}
\end{array}
\right]  ,
\]
which represents the constant matrix $\left[  \mathfrak{f}_{1}^{1}\right]
^{n\times\alpha}$.

Given $_{\alpha}{\mathring{x}}_{\beta}\in\mathring{X}_{m}^{n}(\alpha,\beta),$
consider the marked sequence of $i^{th}$ rows
\[
\left(  \mathbf{x}_{i,\alpha},\ldots,\mathbf{x}_{i,\varrho-1}
,\mathbf{\mathring{x}}_{i,\varrho},\mathbf{\dot{x}}_{i,\varrho+1}
,\ldots,\mathbf{\dot{x}}_{i,\beta}\right)  .
\]
Let $\mathbf{x}_{i,k}^{\prime}$ denote the vector obtained from $\mathbf{\dot
{x}}_{i,k}$ by deleting the marked entry $\dot{1},$ and form the (unmarked)
row-descent sequence
\[
(\mathbf{x}_{i,\alpha}^{\prime\prime},...,\mathbf{x}_{i,\beta-1}^{\prime
\prime})=(\mathbf{x}_{i,\alpha},...,\mathbf{x}_{i,\varrho-1},\mathbf{x}
_{i,\varrho+1}^{\prime},...,\mathbf{x}_{i,\beta}^{\prime}).
\]
Then $_{\alpha}x_{\beta-1}^{\prime\prime}$ denotes the bitruncated element
whose $i^{th}$ rows are $\mathbf{x}_{i,\alpha}^{\prime\prime},...,\mathbf{x}
_{i,\beta-1}^{\prime\prime}$.

\begin{example}
Consider the following element $x\in X_{6}^{2}$ and its BTP factorization:
\[
x=
\raisebox{-0.3892in}{\includegraphics[
height=0.8605in,
width=0.691in
]
{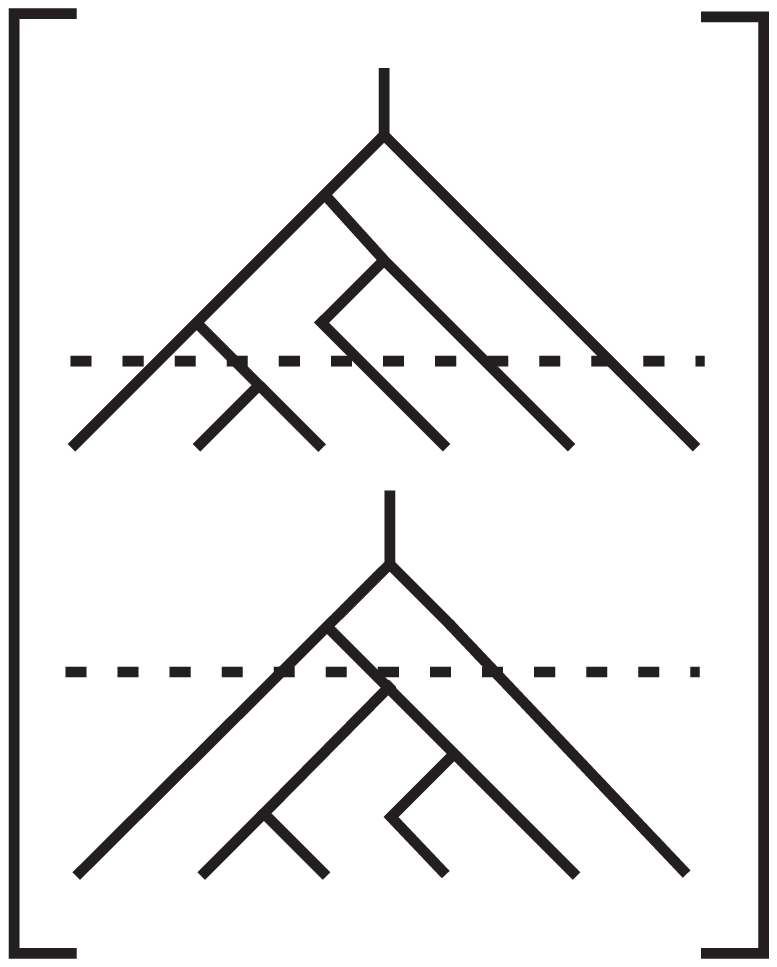}
}
=\left[
\begin{array}
[c]{c}
\curlywedge\\
\curlywedge
\end{array}
\right]  \left[
\begin{array}
[c]{c}
\curlywedge\text{\thinspace}\mathbf{1}\\
\curlywedge\text{\thinspace}\mathbf{1}
\end{array}
\right]  \left[
\begin{array}
[c]{c}
\mathbf{1}\text{\thinspace}\curlywedge\text{\thinspace}\mathbf{1}\\
\mathbf{1}\text{\thinspace}\curlywedge\text{\thinspace}\mathbf{1}
\end{array}
\right]  \left[
\begin{array}
[c]{c}
\curlywedge\text{ }\mathbf{1}\text{ }\mathbf{1}\text{ }\mathbf{1}\\
\ \mathbf{1}\text{ }\mathbf{1}\text{\thinspace}\curlywedge\text{{}}\mathbf{1}
\end{array}
\right]  \left[
\begin{array}
[c]{c}
\mathbf{1}\text{\thinspace}\curlywedge\text{\thinspace}\mathbf{1}\text{
}\mathbf{1}\text{ }\mathbf{1}\\
\mathbf{1}\text{\thinspace}\curlywedge\text{\thinspace}\mathbf{1}\text{
}\mathbf{1}\text{ }\mathbf{1}
\end{array}
\right]  ,
\]
where the dotted lines indicate the lowest levels in which branches are
attached on the extreme left ($\alpha=\varrho(x,2)=2$, and $\beta
=\varrho(x,1)=4$). Then
\[
\mathring{x}=\left[
\begin{array}
[c]{c}
2\\
2
\end{array}
\right]  \left[
\begin{array}
[c]{cc}
2 & 1\\
\mathring{2} & \mathring{1}
\end{array}
\right]  \left[
\begin{array}
[c]{ccc}
1 & 2 & 1\\
\dot{1} & 2 & 1
\end{array}
\right]  \left[
\begin{array}
[c]{cccc}
\mathring{2} & \mathring{1} & \mathring{1} & \mathring{1}\\
\dot{1} & 1 & 2 & 1
\end{array}
\right]  \left[
\begin{array}
[c]{ccccc}
\dot{1} & 2 & 1 & 1 & 1\\
\dot{1} & 2 & 1 & 1 & 1
\end{array}
\right]  \in\mathring{X}_{6}^{2}
\]
and
\[
_{2}\mathring{x}_{4}=\left[
\begin{array}
[c]{cc}
2 & 1\\
\mathring{2} & \mathring{1}
\end{array}
\right]  \left[
\begin{array}
[c]{ccc}
1 & 2 & 1\\
\dot{1} & 2 & 1
\end{array}
\right]  \left[
\begin{array}
[c]{cccc}
\mathring{2} & \mathring{1} & \mathring{1} & \mathring{1}\\
\dot{1} & 1 & 2 & 1
\end{array}
\right]  \in\mathring{X}_{6}^{2}\left(  2,4\right)  .
\]
The projections
\[
\left(  \left(  2,1\right)  ,\left(  1,2,1\right)  ,\left(  \mathring
{2},\mathring{1},\mathring{1},\mathring{1}\right)  \right)  \mapsto\left(
\left(  2,1\right)  ,\left(  1,2,1\right)  \right)
\]
and
\[
\left(  \left(  \mathring{2},\mathring{1}\right)  ,\left(  \dot{1},2,1\right)
,\left(  \dot{1},1,2,1\right)  \right)  \mapsto\left(  \left(  2,1\right)
,\left(  1,2,1\right)  \right)
\]
send $_{2}\mathring{x}_{4}$ to
\[
_{2}x_{3}^{\prime\prime}=\left[
\begin{array}
[c]{cc}
2 & 1\\
2 & 1
\end{array}
\right]  \left[
\begin{array}
[c]{ccc}
1 & 2 & 1\\
1 & 2 & 1
\end{array}
\right]  .
\]

\end{example}

Dually, there is the subposet $Y_{n}^{m}\subseteq\vee_{n}^{\times m},$ which
indexes the vertices of the subcomplex $\Delta^{\left(  m-1\right)  }\left(
P_{n}\right)  \subseteq P_{n}^{\times m};$ an element $y\in Y_{n}^{m}$ is
expressed as a row matrix $y=[\check{T}_{1}\cdots\check{T}_{m}]$ of $m$
down-rooted binary trees with $n+1$ leaves and $n$ levels. Now if $b\in
\vee_{n},$ there exists a unique BTP\ $\Upsilon$-factorization $b=b_{n}\cdots
b_{1}$ such that\ $b_{i}$ is a $i\times1$ column matrix over $\left\{
\mathbf{1,}\curlyvee\right\}  $ containing the entry $\curlyvee$ exactly once.
Thus by factoring each $\check{T}_{j},$ we obtain the (unique) BTP\ $\Upsilon
$-factorization $y=y_{n}\cdots y_{1}$ in which $y_{i}$ is an $i\times m$ matrix.

Let $\varkappa(y,j)$ be the \emph{highest level} of $\check{T}_{j}$ in which a
branch is attached on the extreme right. Given positive integers $\epsilon
\geq\delta,$ consider the set of \emph{bitruncated elements}
\begin{align*}
Y_{n}^{m}(\epsilon,\delta)  &  =\left\{  _{\epsilon}y_{\delta}=y_{\epsilon
}\cdots y_{\delta}\mid y=y_{n}\cdots y_{\epsilon}\cdots y_{\delta}\cdots
y_{1}\in{Y}_{n}^{m},\right. \\
&  \epsilon=\max_{1\leq j\leq m}\varkappa(y,j)\text{ and }\delta=\min_{1\leq
j\leq m}\varkappa(y,j)\}.
\end{align*}
Define $_{\epsilon}y_{\delta}<$ $_{\epsilon^{\prime}}y_{\delta^{\prime}
}^{\prime}$ if and only if $y<y^{\prime}$ for some $y=y_{n}\cdots y_{\epsilon
}\cdots y_{\delta}\cdots y_{1}$ and $y^{\prime}=y_{n}^{\prime}\cdots
y_{\epsilon^{\prime}}^{\prime}\cdots y_{\delta^{\prime}}^{\prime}\cdots
y_{1}^{\prime}$ in $Y_{n}^{m};$ then $Y_{n}^{m}(\epsilon,\delta)$ is the
\emph{poset of bitruncated elements.}

Now replace the symbols $\mathbf{1}$ and $\curlyvee$ in each $y_{i}$ with the
integers $1$ and $2,$ respectively. Then the $\left(  n+1,n\right)  $-column
descent sequence $D_{j}\!=\!\left(  \mathbf{y}_{n,j},\ldots,\mathbf{y}
_{2,j},\mathbf{y}_{1,j}=\left(  2\right)  \right)  $ of $\check{T}_{j}$
appears as the $j^{th}$ columns of $y_{n},\ldots,y_{1},$ and $\varkappa
=\varkappa(y,j)$ is the largest integer such that $\mathbf{y}_{\varkappa
,j}=(1,...,1,2).$ Introduce the decorations $\mathbf{\mathring{y}}
_{\varkappa,j}=({\mathring{1}},...,{\mathring{1}},{\mathring{2}})^{T}\ $and
$\mathbf{\dot{y}}_{i,j}=\left(  \ldots,1,2,\ldots,1,{\dot{1}}\right)  ^{T}$
for $i>\varkappa,$ and obtain the corresponding $\left(  n+1,n\right)
$-\emph{marked column descent sequence }
\[
\check{D}_{j}=\left(  \mathbf{\dot{y}}_{n,j},\ldots,\mathbf{\dot{y}
}_{\varkappa+1,j},\mathbf{\mathring{y}}_{\varkappa,j},\mathbf{y}
_{\varkappa-1,j},\ldots,\mathbf{y}_{1,j}\right)  .
\]
The term $\mathbf{\mathring{y}}_{\varkappa,j}$ represents the constant
$\varkappa\times1$ matrix $\left[  \mathfrak{f}_{1}^{1}\cdots\mathfrak{f}
_{1}^{1}\right]  ^{T}.$ Let $\mathring{y}$ denote the matrix string
$y=y_{n}\cdots y_{1}$ with decorated integer entries; then the poset
of\emph{\ marked bitruncated} \emph{elements }is
\[
\mathring{Y}_{n}^{m}(\epsilon,\delta)=\left\{  _{\epsilon}\mathring{y}
_{\delta}\mid\text{ }_{\epsilon}y_{\delta}\in Y_{n}^{m}(\epsilon
,\delta)\right\}  .
\]
Note that $\varkappa(y,j)$ is constant for all $j$ if and only if
$\delta=\epsilon,$ in which case $Y_{n}^{m}(\delta,\delta)$ is a singleton set
containing the $\delta\times m$ matrix
\[
_{\delta}\mathring{y}_{\delta}=\left[  \mathbf{\mathring{y}}_{\delta,1}
\cdots\mathbf{\mathring{y}}_{\delta,m}\right]  =\left[
\begin{array}
[c]{ccc}
\mathring{1} & \cdots & \mathring{1}\\
\vdots\medskip &  & \vdots\\
\mathring{1} & \cdots & \mathring{1}\\
\mathring{2} & \cdots & \mathring{2}
\end{array}
\right]  ,
\]
which represents the constant matrix $\left[  \mathfrak{f}_{1}^{1}\right]
^{\delta\times m}$.

Given $_{\epsilon}{\mathring{y}}_{\delta}\in\mathring{Y}_{n}^{m}
(\epsilon,\delta),$ consider the marked sequence of $j^{th}$ columns
\[
\left(  \mathbf{\dot{y}}_{\epsilon,j},\ldots,\mathbf{\dot{y}}_{\varkappa
+1,j},\mathbf{\mathring{y}}_{\varkappa,j},\mathbf{y}_{\varkappa-1,j}
,\ldots,\mathbf{y}_{\varkappa,j}\right)  .
\]
Let $\mathbf{y}_{k,j}^{\prime}$ denote the vector obtained from $\mathbf{\dot
{y}}_{k,j}$ by deleting the marked entry $\dot{1},$ and form the (unmarked)
column-descent sequence
\[
\left(  \mathbf{y}_{\epsilon-1,j}^{\prime\prime},...,\mathbf{y}_{\delta
,j}^{\prime\prime}\right)  =\left(  \mathbf{y}_{\epsilon-1,j}^{\prime}
,\ldots,\mathbf{y}_{\varkappa+1,j}^{\prime},\mathbf{y}_{\varkappa-1,j}
,\ldots,\mathbf{y}_{\varkappa,j}\right)  .
\]
Then $_{\epsilon-1}y_{\delta}^{\prime\prime}$ denotes the bitruncated element
whose $j^{th}$ columns are $\mathbf{y}_{\epsilon-1,j}^{\prime\prime
},...,\mathbf{y}_{\delta,j}^{\prime\prime}.$

\begin{example}
Consider the following element $y\in Y_{6}^{2}$ and its BTP factorization:
\[
y=
\raisebox{-0.2153in}{\includegraphics[
height=0.4774in,
width=1.3456in
]
{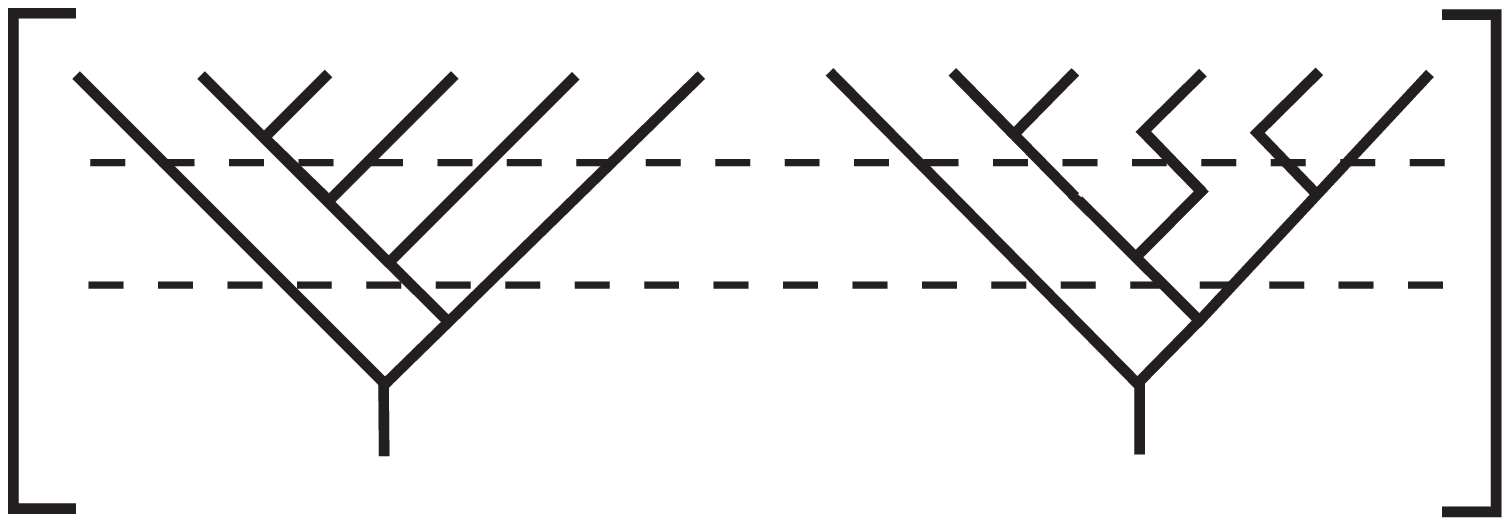}
}
=\left[
\begin{array}
[c]{c}
\mathbf{1}\text{ }\mathbf{1}\\
\curlyvee\curlyvee\\
\mathbf{1}\text{ }\mathbf{1}\\
\mathbf{1}\text{ }\mathbf{1}\\
\mathbf{1}\text{ }\mathbf{1}
\end{array}
\right]  \left[
\begin{array}
[c]{c}
\mathbf{1}\text{ }\mathbf{1}\\
\curlyvee\text{\thinspace}\mathbf{1}\\
\mathbf{1}\text{ }\mathbf{1}\\
\mathbf{1}\text{\thinspace}\curlyvee
\end{array}
\right]  \left[
\begin{array}
[c]{c}
\mathbf{1}\text{ }\mathbf{1}\\
\curlyvee\curlyvee\\
\mathbf{1}\text{ }\mathbf{1}
\end{array}
\right]  \left[
\begin{array}
[c]{c}
\mathbf{1}\text{ }\mathbf{1}\\
\curlyvee{\curlyvee}
\end{array}
\right]  \left[  \curlyvee\curlyvee\right]  ,
\]
where the dotted lines indicate the highest levels in which branches are
attached on the extreme right ($\epsilon=\varkappa(y,1)=4$, and $\delta
=\varkappa(y,2)=2$). Then
\[
\mathring{y}=\left[
\begin{array}
[c]{cc}
1 & 1\\
2 & 2\\
1 & 1\\
1 & 1\\
\dot{1} & \dot{1}
\end{array}
\right]  \left[
\begin{array}
[c]{cc}
1 & \mathring{1}\\
2 & \mathring{1}\\
1 & \mathring{1}\\
\dot{1} & \mathring{2}
\end{array}
\right]  \left[
\begin{array}
[c]{cc}
1 & 1\\
2 & 2\\
\dot{1} & 1
\end{array}
\right]  \left[
\begin{array}
[c]{cc}
\mathring{1} & 1\\
\mathring{2} & 2
\end{array}
\right]  \left[
\begin{array}
[c]{cc}
2 & 2
\end{array}
\right]  \in\mathring{Y}_{6}^{2}
\]
and
\[
_{4}\mathring{y}_{2}=\left[
\begin{array}
[c]{cc}
1 & \mathring{1}\\
2 & \mathring{1}\\
1 & \mathring{1}\\
\dot{1} & \mathring{2}
\end{array}
\right]  \left[
\begin{array}
[c]{cc}
1 & 1\\
2 & 2\\
\dot{1} & 1
\end{array}
\right]  \left[
\begin{array}
[c]{cc}
\mathring{1} & 1\\
\mathring{2} & 2
\end{array}
\right]  \in\mathring{Y}_{6}^{2}\left(  4,2\right)  .
\]
The projections
\[
\left(  \left(
\begin{array}
[c]{c}
1\\
2\\
1\\
\dot{1}
\end{array}
\right)  ,\left(
\begin{array}
[c]{c}
1\\
2\\
\dot{1}
\end{array}
\right)  ,\left(
\begin{array}
[c]{c}
\mathring{1}\\
\mathring{2}
\end{array}
\right)  \right)  \mapsto\left(  \left(
\begin{array}
[c]{c}
1\\
2\\
1
\end{array}
\right)  ,\left(
\begin{array}
[c]{c}
1\\
2
\end{array}
\right)  \right)
\]
and
\[
\left(  \left(
\begin{array}
[c]{c}
\mathring{1}\\
\mathring{1}\\
\mathring{1}\\
\mathring{2}
\end{array}
\right)  ,\left(
\begin{array}
[c]{c}
1\\
2\\
1
\end{array}
\right)  ,\left(
\begin{array}
[c]{c}
1\\
2
\end{array}
\right)  \right)  \mapsto\left(  \left(
\begin{array}
[c]{c}
1\\
2\\
1
\end{array}
\right)  ,\left(
\begin{array}
[c]{c}
1\\
2
\end{array}
\right)  \right)
\]
send $_{4}\mathring{y}_{2}$ to
\[
_{3}y_{2}^{\prime\prime}=\left[
\begin{array}
[c]{cc}
1 & 1\\
2 & 2\\
1 & 1
\end{array}
\right]  \left[
\begin{array}
[c]{cc}
1 & 1\\
2 & 2
\end{array}
\right]  .
\]

\end{example}

Now extend the poset structures on these sets of bitruncated elements to
\[
\mathring{X}_{m}^{j}(\alpha,\beta)\cup\mathring{Y}_{m}^{i}(\epsilon,\delta)
\]
in the following way: Given $_{\alpha}{\mathring{x}}_{\beta}\in\mathring
{X}_{m}^{j}(\alpha,\beta)$ and $_{\epsilon}{\mathring{y}}_{\delta}\in
\mathring{Y}_{m}^{i}(\epsilon,\delta),$ define $_{\alpha}\mathring{x}_{\beta
}\leq\,_{\epsilon}\mathring{y}_{\delta}$ if $\alpha\geq i$ and $j\leq\delta,$
and $_{\epsilon}\mathring{y}_{\delta}\leq\,_{\alpha}\mathring{x}_{\beta}$ if
$\alpha\leq i$ and $j\geq\delta.$ Then $\alpha=i$ and $j=\beta$ implies
$\alpha=\beta$ and $\epsilon=\delta,$ in which case $_{\alpha}\mathring
{x}_{\alpha}=\,_{\delta}\mathring{y}_{\delta}.$ This equality reflects the
correspondence
\[
\left[
\begin{array}
[c]{cccc}
\mathring{2} & \mathring{1} & \cdots & \mathring{1}\\
\vdots\medskip & \vdots &  & \vdots\\
\mathring{2} & \mathring{1} & \cdots & \mathring{1}
\end{array}
\right]  ^{\delta\times\alpha}\leftrightarrow\left[
\begin{array}
[c]{ccc}
\mathfrak{f}_{1}^{1} & \cdots & \mathfrak{f}_{1}^{1}\\
\vdots\medskip &  & \vdots\\
\mathfrak{f}_{1}^{1} & \cdots & \mathfrak{f}_{1}^{1}
\end{array}
\right]  ^{\delta\times\alpha}\leftrightarrow\left[
\begin{array}
[c]{ccc}
\mathring{1} & \cdots & \mathring{1}\\
\vdots\medskip &  & \vdots\\
\mathring{1} & \cdots & \mathring{1}\\
\mathring{2} & \cdots & \mathring{2}
\end{array}
\right]  ^{\delta\times\alpha}.
\]

Let $A=\left[  a_{ij}\right]  $ be an $(n+1)\times m$ matrix over
$\{\mathbf{1},\curlywedge\},$ each row of which contains the entry
$\curlywedge$ exactly once, and let $B=\left[  b_{ij}\right]  $ be an
$n\times(m+1)$ matrix over $\{\mathbf{1},\curlyvee\},$ each column of which
contains the entry $\curlyvee$ exactly once. Recall that a BTP $(A,B)$ is an
$\left(  i,j\right)  $\emph{-edge pair} if $a_{ij}=a_{i+1,j}=\curlywedge$ and
$b_{ij}=b_{i,j+1}=\curlyvee$. Note that $\left(  x_{m},y_{n}\right)  $ is the
only potential edge pair in a matrix string $x_{1}\cdots x_{m}y_{n}\cdots
y_{1}\in X_{m}^{n+1}\times Y_{n}^{m+1}.$

Given an $\left(  i,j\right)  $-edge pair $\left(  A,B\right)  ,$ let
$A^{i\ast}$ and $B^{\ast j}$ denote the matrices obtained by deleting the
$i^{th}$ row of $A$ and the $j^{th}$ column of $B$. If $c=C_{1}\cdots
C_{k}C_{k+1}$ $\cdots C_{r}$ is a matrix string in which $\left(
C_{k},C_{k+1}\right)  $ is an $\left(  i,j\right)  $-edge pair, the
$(i,j)$\emph{-transposition of }$c$\emph{\ in position} $k$ is the matrix
string
\[
\mathcal{T}_{ij}^{k}(c)=C_{1}\cdots C_{k+1}^{\ast j}C_{k}^{i\ast}\cdots
C_{r}.
\]
If $\left(  x_{m},y_{n}\right)  $ is an $\left(  i,j\right)  $-edge pair in
$u=x_{1}\cdots x_{m}y_{n}\cdots y_{1}\in X_{m}^{n+1}\times Y_{n}^{m+1}$, then
$(x_{m-1},y_{n}^{\ast j})$ and $(x_{m}^{i\ast},y_{n-1})$ are the potential
edge pairs in $\mathcal{T}_{ij}^{m}\left(  u\right)  .$ If $(x_{m-1}
,y_{n}^{\ast j})$ is a $\left(  k,l\right)  $-edge pair, then $(x_{m-2}
,y_{n}^{\ast j\ast l})$ is a potential edge pairs in
\[
\mathcal{T}_{kl}^{m-1}\mathcal{T}_{ij}^{m}\left(  u\right)  =x_{1}\cdots
x_{m-2}y_{n}^{\ast j\ast l}x_{m-1}^{k\ast}x_{m}^{i\ast}y_{n-1}\cdots y_{1},
\]
and so on. Clearly, $\mathcal{T}_{i_{t}j_{t}}^{k_{t}}\cdots\mathcal{T}
_{i_{1}j_{1}}^{k_{1}}\left(  u\right)  $ uniquely determines a shuffle
permutation $\sigma\in\Sigma_{m,n}.$ On the other hand, $\left(  A,B\right)  $
can be an $\left(  i,j\right)  $-edge pair for multiple values of $i$ and $j,$
in which case distinctly different compositions $\mathcal{T}_{i_{t}j_{t}
}^{k_{t}}\cdots\mathcal{T}_{i_{1}j_{1}}^{k_{1}}$ act on $u$ and determine the
same $\sigma.$ Thus we define $\mathcal{T}_{{\operatorname*{Id}}
}:={\operatorname*{Id}}$ and
\[
\mathcal{T}_{\sigma}\left(  u\right)  :=\left\{  \mathcal{T}_{i_{t}j_{t}
}^{k_{t}}\cdots\mathcal{T}_{i_{1}j_{1}}^{k_{1}}\left(  u\right)
\mid\mathcal{T}_{i_{t}j_{t}}^{k_{t}}\cdots\mathcal{T}_{i_{1}j_{1}}^{k_{1}
}\ \text{determines}\ \sigma\in\Sigma_{m,n}\right\}  .
\]

Recall that $\mathcal{PP}_{n,m}=X_{m}^{n+1}\times Y_{n}^{m+1}\cup Z_{n,m},$
where
\[
Z_{n,m}=\left\{  \mathcal{T}_{\sigma}\left(  u\right)  \mid u\in X_{m}
^{n+1}\times Y_{n}^{m+1}\text{ and }\sigma\in\Sigma_{m,n}\smallsetminus
\left\{  {\operatorname*{Id}}\right\}  \right\}  .
\]
The poset $r{\mathcal{PP}}_{n,m}$ is built upon $\mathcal{PP}_{n,m}$. Given
$u=x_{1}\cdots x_{m}y_{n}\cdots y_{1}$ and $a={\mathcal{T}}_{\sigma}(u),\ $let
$x_{i}^{\#}$ and $y_{j}^{\#}$ denote either $x_{i}$, $y_{j}$ or their
respective transpositions in $a.$ Let
\[
\mathcal{X}_{n,m}=\mathcal{X}_{n,m}^{\prime}\cup\mathcal{X}_{n,m}
^{\prime\prime},\text{ }m+n>0,m,n\geq0,
\]
where
\[
\mathcal{X}_{n,m}^{\prime}=\left\{  \left.  \left(  a,\text{\thinspace
}_{\alpha}\mathring{x}_{\alpha}\right)  \in\bigcup_{\substack{1\leq\alpha\leq
m+1\\1\leq i\leq n+1}}\mathcal{PP}_{n,m}\times\mathring{X}_{m+1}^{i}
(\alpha,\alpha)\text{ }\right\vert \text{ }x_{\alpha}^{\#}\ \text{has}
\ i\ \text{rows}\right\}  ,
\]
and
\[
\mathcal{X}_{n,m}^{\prime\prime}=\left\{  \left.  \left(  a,\,_{\alpha
}\mathring{x}_{\beta}\right)  \in\hspace*{-0.1in}\bigcup_{\substack{1\leq
\alpha<\beta\leq m+1\\1\leq i\leq n+1}}\hspace*{-0.1in}\mathcal{PP}
_{n,m}\times\mathring{X}_{m+1}^{i}(\alpha,\beta)\text{ }\right\vert \text{
}_{\alpha}x_{\beta-1}^{\prime\prime}\text{ is a substring of }a\right\}  .
\]
Dually, let
\[
\mathcal{Y}_{n,m}=\mathcal{Y}_{n,m}^{\prime}\cup\mathcal{Y}_{n,m}
^{\prime\prime},\text{ }m+n>0,m,n\geq0,
\]
where
\[
\mathcal{Y}_{n,m}^{\prime}=\left\{  \left.  (a,\,_{\delta}\mathring{y}
_{\delta})\in\bigcup_{\substack{1\leq\delta\leq n+1\\1\leq j\leq
m+1}}\mathcal{PP}_{n,m}\times\mathring{Y}_{n+1}^{j}(\delta,\delta)\text{
}\right\vert \text{\ }y_{\delta}^{\#}\ \text{has }j\ \text{columns}\right\}
,
\]
and
\[
\mathcal{Y}_{n,m}^{\prime\prime}=\left\{  \left.  (a,\,_{\epsilon}\mathring
{y}_{\delta})\in\hspace*{-0.1in}\bigcup_{\substack{1\leq\delta<\epsilon\leq
n+1\\1\leq j\leq m+1}}\hspace*{-0.1in}\mathcal{PP}_{n,m}\times\mathring
{Y}_{n+1}^{j}(\epsilon,\delta)\text{ }\right\vert \text{\ }_{\epsilon
-1}y_{\delta}^{\prime\prime}\text{ is a substring of }a\right\}  .
\]
Define
\[
r{\mathcal{PP}}_{n,m}=\mathcal{X}_{n,m}\cup\mathcal{Y}_{n,m}
\]
with the poset structure generated by $(a,b)\leq(a^{\prime},b^{\prime})$ for

\begin{itemize}
\item $a=a^{\prime}$ and $b\leq b^{\prime};$

\item $a\leq a^{\prime}$ and $b=b^{\prime}$ such that
$u^{\prime}=(\nu_{x}\times\nu_{y})u$
for $\nu_{x}\in S_{\alpha-1}^{\times n+1}\times S_{\beta-\alpha}^{\times
n+1}\times S_{m-\beta+1}^{\times n+1}\subset S_{m}^{\times n+1},\,\nu_{y}\in
S_{i-1}^{\times m+1}\times S_{n-i+1}^{\times m+1}\subset S_{n}^{\times m+1},$
$b\in\mathring{X}_{m+1}^{i}(\alpha,\beta),1\leq\alpha\leq\beta\leq m+1,$ or
$\nu_{x}\in S_{j-1}^{\times n+1}\times S_{m-j+1}^{\times n+1}\subset
S_{m}^{\times n+1},\,\nu_{y}\in S_{\delta-1}^{\times m+1}\times S_{\epsilon
-\delta}^{\times m+1}\times S_{n-\epsilon+1}^{\times m+1}\subset S_{n}^{\times
m+1},$ $b\in\mathring{Y}_{n+1}^{j}(\epsilon,\delta),1\leq\delta\leq
\epsilon\leq n+1$ with convention that $S_{0}\times S_{k}=S_{k}\times
S_{0}=S_{k},\,k\geq1.$
\end{itemize}

Note that for $m,n\geq1,$ $a\in X_{m}^{n+1}\times Y_{n}^{m+1}\subset
{\mathcal{P}P}_{n,m}$ whenever $\alpha=m+1$ in $\mathcal{X}_{n,m}^{\prime}$ or
$\delta=n+1$ in $\mathcal{Y}_{n,m}^{\prime}.$ Hence, the poset structure
identifies the subset
\[
\mathcal{X}_{n,m}^{0}:=\left(  X_{m}^{n+1}\times Y_{n}^{m+1}\right)
\times\mathring{X}_{m+1}^{n+1}(m+1,m+1)\subset\mathcal{X}_{n,m}
\]
with the subset
\[
\mathcal{Y}_{n,m}^{0}:=\left(  X_{m}^{n+1}\times Y_{n}^{m+1}\right)
\times\mathring{Y}_{n+1}^{m+1}(n+1,n+1)\subset\mathcal{Y}_{n,m}.
\]
Note also that $\mathcal{X}_{m,0}^{\prime}=\mathcal{X}_{m,0}^{\prime\prime
}=\mathcal{X}_{0,m}^{\prime\prime}=\mathcal{Y}_{0,n}^{\prime}=\mathcal{X}
_{0,n}^{\prime\prime}=\mathcal{Y}_{n,0}^{\prime\prime}=\varnothing;$ thus
$r{\mathcal{PP}}_{0,m}=\mathcal{X}_{0,m}^{\prime}$ and $r{\mathcal{PP}}
_{n,0}=\mathcal{Y}_{n,0}^{\prime}.$

The poset $\mathcal{JJ}_{n+1,m+1}$ is the image of the quotient map
$r\mathcal{PP}_{n,m}\rightarrow r\mathcal{PP}_{n,m} \diagup\!\sim$ given by
restricting the map $\mathcal{PP}_{\!\ast,\ast}\! \rightarrow \!  \mathcal{KK}
_{\ast+1,\ast+1}$ to left-hand factors, and $JJ_{n+1,m+1}$ is the geometric
realization $|\mathcal{JJ}_{n+1,m+1}|.$

The poset structure of $r\mathcal{PP}$ corresponds to the \textquotedblleft
cylindrical poset\textquotedblright\ $\mathcal{KK}\times I$ in the following
way. Given $a\in\mathcal{KK}_{n,m}$ and $t\geq m+n+1,$ consider a partition
\[
(a,b_{1})<\cdots<(a,b_{t})
\]
of the interval $a\times I,$ where $(a,b_{1})=a\times0$ and $(a,b_{t}
)=a\times1.$ If $a$ indexes a vertex $a_{1}|\cdots|a_{n+m}\in P_{n+m},$ and
$b_{t_{\alpha}}$ denotes the single element of $\mathring{X}_{m}^{j}
(\alpha,\alpha),$ then $(a,b_{t_{\alpha}})$ indexes the vertex
\[
a_{1}|\cdots|a_{n+\alpha-j}|m+n+1|a_{n+1+\alpha-j}|\cdots|a_{m+n}\in
P_{m+n+1};
\]
and in particular,
\[
(a,b_{1})\leftrightarrow a_{1}|\cdots|a_{m+n}|m+n+1\text{ and }(a,b_{t}
)\leftrightarrow m+n+1|a_{1}|\cdots|a_{m+n}.
\]
Note that this correspondence agrees with the combinatorial representation of
$P_{m+n+1}$ as a subdivision\ of the cylinder $P_{m+n}\times P_{2}
=P_{m+n}\times I$ (see \cite{SU2}), or equivalently, with the combinatorial
join $P_{m+n+1}=P_{m+n}\ast_{c}P_{1}$ (see \cite{SU4}). If in addition,
$b_{s}\in\mathring{X}_{m+1}^{j}(\alpha,\beta)$ with $\beta=\alpha+1,$ then
$(a,b_{s})$ subdivides the interval $[(a,b_{t_{\alpha+1}}),(a,b_{t_{\alpha}
})]$ (recall that $b_{t_{\alpha+1}}<b_{t_{\alpha}}$). In the octagon
$JJ_{2,2}$ in Figure 4 and Example \ref{octagon} below, we have
\[
(1|2,b_{1})<(1|2,b_{2})<(1|2,b_{3})<(1|2,b_{4})<(1|2,b_{5})
\]
\[
=1|2|3<v_{1}<1|3|2<v_{2}<3|1|2,
\]
where $v_{1}\leftrightarrow b_{2}\in\mathring{Y}_{2}^{2}(1,2)$ and
$v_{2}\leftrightarrow b_{4}\in\mathring{X}_{2}^{2}(1,2);$ on the other hand,
\[
(2|1,b_{1})<(2|1,b_{2})<(2|1,b_{3})
\]
\[
=2|1|3<2|3|1<3|2|1.
\]

Now consider an element $(a,b)\in\mathcal{PP}_{n,m}\times\mathring{X}
_{m+1}^{j}(\alpha,\beta)\subset r{\mathcal{PP}}_{n,m}$, and recall that
$a=a_{1}\cdots a_{n+m}$ is represented by a piecewise linear path in
$\mathbb{N}^{2}$ with $m+n$ directed components. The element $(a,b)$ is
represented by a piecewise linear path in $\mathbb{N}^{3}$ of the form $BCA$
with $m+n+1+\alpha-\beta$ directed components from $(m+1,1,0)\in\mathbb{N}
^{2}\times0$ to $(1,n+1,1)\in\mathbb{N}^{2}\times1$. The component $A$ is
represented by the path from $(m+1,1,0)$ to $(\beta,j,0)$ in $\mathbb{N}
^{2}\times0$ corresponding to $a_{n+\beta-j}\cdots a_{n+m};$ the component $C$
is represented by the arrow $(\beta,j,0)\rightarrow(\alpha,j,1);$ and the
component $B$ is represented by the path from $(\alpha,j,1)$ to $(1,n+1,1)$ in
$\mathbb{N}^{2}\times1$ corresponding to $a_{1}\cdots a_{n+\alpha-j}$ (see
Figure 1). Note that the arrow representing $C$ is perpendicular to both
integer lattices if and only if $\alpha=\beta$, in which case the path has
$m+n+1$ directed components. The case of $(a,b)\in\mathcal{Y}_{n,m}$ is
analogous with the arrow representing $C$ lying in a vertical plane. For
example, for $m=n=1,$ the singleton $(a,b)\in\mathcal{X}_{1,1}^{0}
(=\mathcal{Y}_{1,1}^{0})$ is expressed by $DC^{\prime}B$ in Figure 1 below.

\unitlength=1.00mm \linethickness{0.4pt} \begin{picture}(92.00,47.00)
\put(52.00,9.00){\line(1,0){40.00}} \put(52.00,9.00){\circle*{1.00}}
\put(72.00,9.00){\circle*{1.00}} \put(62.00,9.00){\circle*{1.00}}
\put(52.00,5.00){\makebox(0,0)[cc]{$1$}} \put(62.00,5.00){\makebox(0,0)[cc]{$2$}}
\put(72.00,5.00){\makebox(0,0)[cc]{$3$}} \ \ \put(62.00,9.00){\vector(0,1){10.00}}
\put(82.00,9.00){\vector(-1,0){10.00}} \put(72.00,9.00){\vector(-1,0){10.00}}
\put(62.00,9.00){\vector(-1,0){10.00}} \put(52.00,9.00){\vector(0,1){10.00}}
\put(34.33,22.00){\line(0,1){25.00}} \put(34.33,22.00){\circle*{1.00}}
\put(44.33,22.00){\circle*{1.00}} \put(34.33,32.00){\circle*{1.00}}
\put(34.33,42.00){\circle*{1.00}} \put(30.66,22.00){\makebox(0,0)[cc]{$1$}}
\put(30.66,32.00){\makebox(0,0)[cc]{$2$}} \put(44.33,22.00){\vector(0,1){10.00}}
\put(44.33,32.00){\vector(-1,0){10.00}} \put(44.33,22.00){\vector(-1,0){10.00}}
\put(34.33,22.00){\vector(0,1){10.00}} \put(34.33,32.00){\vector(0,1){10.00}}
\put(62.00,9.00){\line(-4,3){9.33}} \ \put(51.33,16.67){\vector(-4,3){7.00}}
\put(51.60,9.00){\vector(-4,3){17.27}}
\put(51.60,19.00){\vector(-4,3){17.27}}
\put(52.00,19.00){\circle*{1.33}} \put(62.00,19.00){\circle*{1.33}}
\put(44.33,32.00){\circle*{1.33}}
\put(52.00,19.00){\line(0,1){16.67}}
\put(62.00,19.00){\line(-3,2){9.33}} \put(62.00,19.00){\vector(-1,0){10.00}}
\put(72.00,19.00){\vector(-1,0){10.00}} \put(54.33,22.00){\vector(-1,0){10.00}}
\put(44.33,22.00){\line(1,0){11.67}}
\put(59.00,22.00){\line(1,0){16.33}}
\put(72.00,19.00){\circle*{1.33}}
\put(30.33,42.00){\makebox(0,0)[cc]{$3$}} \
\put(59.00,24.00){\makebox(0,0)[cc]{$_{C^{'}}$}}
\put(58.00,14.20){\makebox(0,0)[cc]{$_{C}$}}
\put(40.33,25.20){\makebox(0,0)[cc]{$_{C^{'''}}$}}
\put(40.67,14.67){\makebox(0,0)[cc]{$_{C^{''}}$}} \
\put(56.67,17.67){\makebox(0,0)[cc]{$_{A}$}}
\put(64.00,13.67){\makebox(0,0)[cc]{$_{B}$}}
\put(57.20,7.00){\makebox(0,0)[cc]{$_{A^{'}}$}}
\put(50.33,14.33){\makebox(0,0)[cc]{$_{B^{'}}$}}
\put(40.00,33.33){\makebox(0,0)[cc]{$_{D}$}} \
\put(46.50,26.67){\makebox(0,0)[cc]{$_{E}$}}
\put(32.50,27.00){\makebox(0,0)[cc]{$_{E^{'}}$}} \
\put(41.90,20.60){\makebox(0,0)[cc]{$_{D^{'}}$}} \

\put(54.33,22.00){\circle*{1.33}}
\put(72.00,9.00){\vector(0,1){9.67}}
\put(51.33,26.43){\vector(-4,3){7.00}}
\end{picture}

\begin{center}
Figure 1. $DEC<DC^{\prime}B<C^{\prime\prime\prime}AB<C^{\prime\prime\prime
}B^{\prime}A^{\prime}.$ \vspace*{0.1in}
\end{center}

Now suppose $w\in{\mathcal{JJ}}_{n+1,m+1}$ is the projection of $\tilde
{w}=(a,b)\in r\mathcal{PP}_{n,m}.$ Transform $w$ into a $0$-dimensional
element of the $A_{\infty}$-bialgebra morphism matrad $\mathcal{JJ}_{\infty}$
in the following way (see Example \ref{octagon} below): If $\tilde{w}
\in\mathcal{X}_{n,m}^{0}$ and $a=u\times v,$ replace $a$ by $u\cdot b\cdot v;$
if $\tilde{w}\in\mathcal{X}_{n,m}^{\prime},$ replace $a=\cdots x_{\alpha}
^{\#}\cdots$ by $\cdots\,_{\alpha}\mathring{x}_{\alpha}\cdot x_{\alpha}
^{\#}\cdots;$ if $\tilde{w}\in\mathcal{X}_{n,m}^{\prime\prime}$ and
$a=\cdots_{\alpha}\mathring{x}_{\beta-1}^{\prime\prime}\cdots,$ replace
$_{\alpha}x_{\beta-1}^{\prime\prime}$ by $_{\alpha}\mathring{x}_{\beta
}=x_{\alpha}\cdots x_{\beta},$ if $\tilde{w}\in\mathcal{Y}_{n,m}^{\prime},$
replace $a=\cdots y_{\delta}^{\#}\cdots$ by $\cdots y_{\delta}^{\#}
\cdot\,_{\delta}\mathring{y}_{\delta}\cdots;$ if $\tilde{w}\in\mathcal{Y}
_{n,m}^{\prime\prime}$ and $a=\cdots\,_{\epsilon-1}\mathring{y}_{\delta
}^{\prime\prime}\cdots,$ replace $_{\epsilon-1}y_{\delta}^{\prime\prime}$ by
$_{\epsilon}\mathring{y}_{\delta}=y_{\epsilon}\cdots y_{\delta}.$ Now replace
each row $(\mathring{2},\mathring{1},...,\mathring{1})$ by $(\mathfrak{f}
_{1}^{1},\mathfrak{f}_{1}^{1},...,\mathfrak{f}_{1}^{1})$, and each column
$(\mathring{1},...,\mathring{1},\mathring{2})^{T} $ by $(\mathfrak{f}_{1}
^{1},...,\mathfrak{f}_{1}^{1},\mathfrak{f}_{1}^{1})^{T};$ delete
$\overset{\centerdot}{1}$'s; replace each $1$ by $\mathbf{1},$ replace each
$2$ in $x_{i}^{\#}$ by $\theta_{2}^{1}$, and replace each $2$ in $y_{j}^{\#}$
by $\theta_{1}^{2}.$ This transformation induces the bijection in (\ref{JJ})
as in the absolute case.

\begin{example}
\label{octagon} The labels $v_{1}$ and $v_{2}$ are the midpoints of the edges
$1|23$ and $13|2$ of $P_{3},$ respectively (see Figure 4).
\[
\begin{array}
[c]{rllll}
1|2|3 & \leftrightarrow & DEC & = & \left[
\begin{array}
[c]{c}
\theta_{2}^{1}\medskip\\
\theta_{2}^{1}
\end{array}
\right]  \left[
\begin{array}
[c]{cc}
\theta_{1}^{2} & \theta_{1}^{2}
\end{array}
\right]  \left[
\begin{array}
[c]{cc}
\mathfrak{f}_{1}^{1} & \mathfrak{f}_{1}^{1}
\end{array}
\right]  \bigskip\\
1|3|2 & \leftrightarrow & DC^{\prime}B & = & \left[
\begin{array}
[c]{c}
\theta_{2}^{1}\medskip\\
\theta_{2}^{1}
\end{array}
\right]  \left[
\begin{array}
[c]{cc}
\mathfrak{f}_{1}^{1}\medskip & \mathfrak{f}_{1}^{1}\\
\mathfrak{f}_{1}^{1} & \mathfrak{f}_{1}^{1}
\end{array}
\right]  \left[
\begin{array}
[c]{cc}
\theta_{1}^{2} & \theta_{1}^{2}
\end{array}
\right]  \bigskip\\
3|1|2 & \leftrightarrow & C^{\prime\prime\prime}AB & = & \left[
\begin{array}
[c]{c}
\mathfrak{f}_{1}^{1}\medskip\\
\mathfrak{f}_{1}^{1}
\end{array}
\right]  \left[
\begin{array}
[c]{c}
\theta_{2}^{1}\medskip\\
\theta_{2}^{1}
\end{array}
\right]  \left[
\begin{array}
[c]{cc}
\theta_{1}^{2} & \theta_{1}^{2}
\end{array}
\right]  \bigskip\\
2|1|3 & \leftrightarrow & E^{\prime}D^{\prime}C & = & \left[  \theta_{1}
^{2}\right]  \left[  \theta_{2}^{1}\right]  \left[
\begin{array}
[c]{cc}
\mathfrak{f}_{1}^{1} & \mathfrak{f}_{1}^{1}
\end{array}
\right]  \bigskip\\
2|3|1 & \leftrightarrow & E^{\prime}C^{\prime\prime}A^{\prime} & = & \left[
\theta_{1}^{2}\right]  \left[  \mathfrak{f}_{1}^{1}\right]  \left[  \theta
_{2}^{1}\right]  \bigskip\\
3|2|1 & \leftrightarrow & C^{\prime\prime\prime}B^{\prime}A^{\prime} & = &
\left[
\begin{array}
[c]{c}
\mathfrak{f}_{1}^{1}\medskip\\
\mathfrak{f}_{1}^{1}
\end{array}
\right]  \left[  \theta_{1}^{2}\right]  \left[  \theta_{2}^{1}\right]
\end{array}
\]
\[
\begin{array}
[c]{rll}
v_{1} & \leftrightarrow & \left[
\begin{array}
[c]{c}
\theta_{2}^{1}\medskip\\
\theta_{2}^{1}
\end{array}
\right]  \left[
\begin{array}
[c]{cc}
\left[  \theta_{1}^{2}\right]  \left[  \mathfrak{f}_{1}^{1}\right]  & \left[
\begin{array}
[c]{c}
\mathfrak{f}_{1}^{1}\medskip\\
\mathfrak{f}_{1}^{1}
\end{array}
\right]  \left[  \theta_{1}^{2}\right]
\end{array}
\right]  \bigskip\\
v_{2} & \leftrightarrow & \left[
\begin{array}
[c]{c}
\left[  \theta_{2}^{1}\right]  \left[
\begin{array}
[c]{cc}
\mathfrak{f}_{1}^{1} & \mathfrak{f}_{1}^{1}
\end{array}
\right]  \medskip\\
\left[  \mathfrak{f}_{1}^{1}\right]  \left[  \theta_{2}^{1}\right]
\end{array}
\right]  \left[
\begin{array}
[c]{cc}
\theta_{1}^{2} & \theta_{1}^{2}
\end{array}
\right]  .
\end{array}
\hspace*{0.45in}
\]

\end{example}

The bijection in (\ref{JJ}) can be described on the codimension 1 level in the
following way: The components $\left[  \theta_{m+1}^{n+1}\right]  \left[
\mathfrak{f}_{1}^{1}\cdots\mathfrak{f}_{1}^{1}\right]  $ and $\left[
\mathfrak{f}_{1}^{1}\cdots\mathfrak{f}_{1}^{1}\right]  \left[  \theta
_{m+1}^{n+1}\right]  $ of $(\mathcal{JJ}_{\infty})_{n+1,m+1}$ are assigned to
the cells $KK_{n+1,m+1}\times0$ and $KK_{n+1,m+1}\times1,$ which are the
respective projections of $\underline{n+m}\mid n+m+1$ and $n+m+1\mid
\underline{n+m}$ in $P_{m+n+1},$ and labeled by the leaf sequences
\[
\begin{array}
[c]{c}
\overset{n+1}{\overbrace{^{\mathstrut}{\scriptsize 1\cdots1}^{\mathstrut}}}\\
\underset{m+1}{\underbrace{_{\mathstrut}{\scriptsize \mathring{1}
\cdots\mathring{1}}_{\mathstrut}}}
\end{array}
\text{ \ and \ }
\begin{array}
[c]{c}
\overset{n+1}{\overbrace{^{\mathstrut}{\scriptsize \mathring{1}\cdots
\mathring{1}}^{\mathstrut}}}\\
\underset{m+1}{\underbrace{_{\mathstrut}{\scriptsize 1\cdots1}_{\mathstrut}}}
\end{array}
\]
respectively. The other components ${A}_{p}^{\mathbf{y}}R_{\mathbf{x}}^{q}$
and ${Q}_{p}^{\mathbf{y}}B_{\mathbf{x}}^{q}$ of $(\mathcal{JJ}_{\infty
})_{n+1,m+1}$ are assigned to the cells $e_{1}$ and $e_{2}$ of $JJ_{n+1,m+1}$
obtained by subdividing $K_{n+1,m+1}\times I$ in the following ways: Let
$e_{(\mathbf{y},\mathbf{x})}=C|D$ be the codimension 1 cell of $P_{m+n}$
defined in line (\ref{cell}). Then
\[
C|\left(  D\cup\left\{  m+n+1\right\}  \right)  \cup\left(  C\cup\left\{
m+n+1\right\}  \right)  |D=C|D\times I\subset P_{m+n}\times I\approx
P_{m+n+1},
\]
\[
e_{1}=(\vartheta_{n,m}\times1)(C|D\cup\left\{  m+n+1\right\}  ),\text{ and
}e_{2}=(\vartheta_{n,m}\times1)\left(  C\cup\left\{  m+n+1\right\}  |D\right)
.
\]
The leaf sequences
\[
\begin{array}
[c]{c}
\mathbf{y}\\
\mathbf{\mathring{x}}
\end{array}
\text{ \ and \ }
\begin{array}
[c]{c}
\mathbf{\mathring{y}}\\
\mathbf{x}
\end{array}
\]
label $e_{1}$ and $e_{2},$ respectively (see Figures 2-5 below). \bigskip

\unitlength=1.00mm
\linethickness{0.4pt}
\ifx\plotpoint\undefined\newsavebox{\plotpoint}\fi\begin{picture}(52.5,20)(0,0)
\put(58,14.5){\makebox(0,0)[cc]{$\bullet$}}
\put(58.3,18){\makebox(0,0)[cc]{1}}
\put(58.5,11.25){\makebox(0,0)[cc]{$_{\mathring{1}}$}}
\put(58.5,8.75){\makebox(0,0)[cc]{$_{\mathring{1}}$}}
\end{picture}\vspace*{-0.2in}

\begin{center}
Figure 2. The point $JJ_{1,1}.$\vspace{0.3in}
\end{center}

\noindent\unitlength=1.00mm
\linethickness{0.4pt}\begin{picture}(-40.33,13.33)
\put(11.00,8.67){\line(1,0){43.00}}
\put(11.33,8.67){\circle*{1.33}}
\put(53.67,8.67){\circle*{1.33}}
\put(11.33,13.33){\makebox(0,0)[cc]{$1|2$}}
\put(11.33,03.33){\makebox(0,0)[cc]{$_{\mathring{1}}^{2}$}}
\put(53.67,13.00){\makebox(0,0)[cc]{$2|1$}}
\put(53.67,03.00){\makebox(0,0)[cc]{$_{1}^{\mathring{1}\mathring{1}}$}}
\put(71.00,8.67){\line(1,0){43.00}}
\put(71.33,8.67){\circle*{1.33}}
\put(113.67,8.67){\circle*{1.33}}
\put(71.33,13.33){\makebox(0,0)[cc]{$1|2$}}
\put(71.33,03.33){\makebox(0,0)[cc]{$_{\mathring{1}\mathring{1}}^{1}$}}
\put(113.67,13.00){\makebox(0,0)[cc]{$2|1$}}
\put(113.67,03.00){\makebox(0,0)[cc]{$_{2}^{\mathring{1}}$}}
\end{picture}\smallskip

\begin{center}
Figure 3. The intervals $JJ_{2,1}$ and $JJ_{1,2}$.\vspace{0.3in}
\end{center}

\vspace{0.3in} \noindent\unitlength=1.00mm \linethickness{0.4pt}
\begin{picture}(88.33,43.67)
\put(38.67,40.67){\line(1,0){48.33}} \put(87.00,40.67){\line(0,-1){31.67}}
\put(87.00,9.00){\line(-1,0){48.33}} \put(38.67,9.00){\line(0,1){31.67}}
\put(38.67,40.67){\circle*{1.33}} \put(38.67,25.00){\circle*{1.33}}
\put(38.67,9.00){\circle*{1.33}} \put(87.00,40.67){\circle*{1.33}}
\put(87.00,9.00){\circle*{1.33}} \put(87.00,25.00){\circle*{1.33}}
\put(32.33,33.33){\makebox(0,0)[cc]{$13|2$}}
\put(42.33,33.33){\makebox(0,0)[cc]{$^{\mathring{1}\mathring{1}}_{11}$}}
\put(32.00,17.00){\makebox(0,0)[cc]{$1|23$}}
\put(42.80,17.00){\makebox(0,0)[cc]{$_{\mathring{1}\mathring{1}}^{11}$}}
\put(61.00,5.00){\makebox(0,0)[cc]{$12|3$}}
\put(61.00,13.00){\makebox(0,0)[cc]{$_{\mathring{1}\mathring{1}}^{2}$}}
\put(62.00,43.67){\makebox(0,0)[cc]{$3|12$}}
\put(62.00,35.67){\makebox(0,0)[cc]{$^{\mathring{1}\mathring{1}}_{2}$}}
\put(93.33,17.00){\makebox(0,0)[cc]{$2|13$}}
\put(84.33,17.00){\makebox(0,0)[cc]{$_{\mathring{2}}^{2}$}}
\put(93.33,32.67){\makebox(0,0)[cc]{$23|1$}}
\put(84.33,32.67){\makebox(0,0)[cc]{$_{2}^{\mathring{2}}$}}
\put(38.67,32.33){\circle*{0.67}} \put(38.67,17.33){\circle*{0.67}}
\end{picture}

\begin{center}
Figure 4. The octagon $JJ_{2,2}$ as a subdivision of $P_{3}$.\vspace{0.2in}
\end{center}

\noindent\emph{Combinatorial data for} $JJ_{2,3}$:\bigskip

$
\begin{array}
[c]{lllllllllll}
123|4 & \leftrightarrow & _{\mathring{1}\mathring{1}\mathring{1}}^{2} =
\theta_{3}^{2}/\mathfrak{f}_{1}^{1}\mathfrak{f}_{1}^{1}\mathfrak{f}_{1}
^{1}\vspace{1mm} &  & 1|234 & \leftrightarrow & _{\mathring{2}\mathring{1}
}^{11}
 = \theta_{2}^{1}\theta_{2}^{1}/\mathfrak{f}_{2}^{2}\left(
\mathfrak{f}_{1}^{1}\mathfrak{f}_{1}^{1}/\theta_{1}^{2}\right) \\
4|123 & \leftrightarrow & _{3}^{\mathring{1}\mathring{1}}
\  = \mathfrak{f}
_{1}^{1}\mathfrak{f}_{1}^{1}/\theta_{3}^{2}\vspace{1mm} &
 & 14|23 &
\leftrightarrow & _{21}^{\mathring{1}\mathring{1}}
 = [(\mathfrak{f}
_{1}^{1}/\theta_{2}^{1})\mathfrak{f}_{2}^{1}+\mathfrak{f}_{2}^{1}\left(
\theta_{2}^{1}/\mathfrak{f}_{1}^{1}\mathfrak{f}_{1}^{1}\right)  ]/\theta
_{2}^{2}\theta_{1}^{2}\\
13|24 & \leftrightarrow & _{\mathring{2}\mathring{1}}^{2}
 \ = \theta_{2}
^{2}/\mathfrak{f}_{2}^{1}\mathfrak{f}_{1}^{1}\vspace{1mm} &  & 2|134 &
\leftrightarrow & _{\mathring{1}\mathring{2}}^{11}
 = \theta_{2}^{1}
\theta_{2}^{1}/(\theta_{1}^{2}/\mathfrak{f}_{1}^{1})\mathfrak{f}_{2}^{2}\\
134|2 & \leftrightarrow & _{21}^{\mathring{2}}
\  = \mathfrak{f}_{2}
^{2}/\theta_{2}^{1}\theta_{1}^{1}\vspace{1mm} &
 & 24|13 & \leftrightarrow &
_{12}^{\mathring{1}\mathring{1}}
 = [(\mathfrak{f}_{1}^{1}/\theta_{2}
^{1})\mathfrak{f}_{2}^{1}+\mathfrak{f}_{2}^{1}\left(  \theta_{2}
^{1}/\mathfrak{f}_{1}^{1}\mathfrak{f}_{1}^{1}\right)  ]/\theta_{1}^{2}
\theta_{2}^{2}\\
3|124 & \leftrightarrow & _{\mathring{3}}^{2}
\ \  = \theta_{1}^{2}
/\mathfrak{f}_{3}^{1}\vspace{1mm} &  & 23|14 & \leftrightarrow &
_{\mathring{1}\mathring{2}}^{2}
= \theta_{2}^{2}/\mathfrak{f}_{1}
^{1}\mathfrak{f}_{2}^{1}\\
34|12 & \leftrightarrow & _{3}^{\mathring{2}}
\ \ = \mathfrak{f}_{1}
^{2}/\theta_{3}^{1}\vspace{1mm} &  & 234|1 & \leftrightarrow & _{12}
^{\mathring{2}}
= \mathfrak{f}_{2}^{2}/\theta_{1}^{1}\theta_{2}^{1}
\end{array}
$

$
\begin{array}
[c]{lllll}
12|34 & \leftrightarrow & _{\mathring{1}\mathring{1}\mathring{1}}^{11}
  =
[\left(  \theta_{2}^{1}/\theta_{2}^{1}\theta_{1}^{1}\right)  \theta_{3}
^{1}+\theta_{3}^{1}\left(  \theta_{2}^{1}/\theta_{1}^{1}\theta_{2}^{1}\right)
]/\\
  &  & \hspace{0.38in}  [(\theta_{1}^{2}/\mathfrak{f}_{1}^{1})(\theta_{1}^{2}
/\mathfrak{f}_{1}^{1})\mathfrak{f}_{1}^{2}+\,(\theta_{1}^{2}/\mathfrak{f}
_{1}^{1})\mathfrak{f}_{1}^{2}(\mathfrak{f}_{1}^{1}\mathfrak{f}_{1}^{1}
/\theta_{1}^{2})+\mathfrak{f}_{1}^{2}(\mathfrak{f}_{1}^{1}\mathfrak{f}_{1}
^{1}/\theta_{1}^{2})(\mathfrak{f}_{1}^{1}\mathfrak{f}_{1}^{1}/\theta_{1}
^{2})]\vspace{1mm}\\
124|3 & \leftrightarrow & _{111}^{\mathring{1}\mathring{1}}
= [(\theta
_{2}^{1}/\theta_{2}^{1}\theta_{1}^{1}/\mathfrak{f}_{1}^{1}\mathfrak{f}_{1}
^{1}\mathfrak{f}_{1}^{1})\mathfrak{f}_{3}^{1}+\mathfrak{f}_{3}^{1}
(\mathfrak{f}_{1}^{1}/\theta_{2}^{1}/\theta_{1}^{1}\theta_{2}^{1})+(\theta
_{3}^{1}/\mathfrak{f}_{1}^{1}\mathfrak{f}_{1}^{1}\mathfrak{f}_{1}^{1}
)(\theta_{2}^{1}/\mathfrak{f}_{1}^{1}\mathfrak{f}_{2}^{1})\\
   &  & \hspace{0.38in}  +(\theta_{3}^{1}/\mathfrak{f}_{1}^{1}\mathfrak{f}_{1}
^{1}\mathfrak{f}_{1}^{1})(\mathfrak{f}_{2}^{1}/\theta_{1}^{1}\theta_{2}
^{1})+(\theta_{2}^{1}/\mathfrak{f}_{1}^{1}\mathfrak{f}_{2}^{1})(\mathfrak{f}
_{2}^{1}/\theta_{1}^{1}\theta_{2}^{1})\\
&  & \hspace{0.38in} -(\theta_{2}^{1}/\mathfrak{f}_{2}
^{1}\mathfrak{f}_{1}^{1})(\mathfrak{f}_{2}^{1}/\theta_{2}^{1}\theta_{1}^{1})  -(\theta_{2}^{1}/\mathfrak{f}_{2}^{1}\mathfrak{f}_{1}
^{1})(\mathfrak{f}_{1}^{1}/\theta_{3}^{1})\\
&  & \hspace{0.38in} -(\mathfrak{f}_{2}^{1}/\theta
_{2}^{1}\theta_{1}^{1})(\mathfrak{f}_{1}^{1}/\theta_{3}^{1})]/\,\theta_{1}
^{2}\theta_{1}^{2}\theta_{1}^{2}
\end{array}
\bigskip$

\begin{center}
\begin{center}
\includegraphics[
height=2.6749in,
width=3.4722in
]
{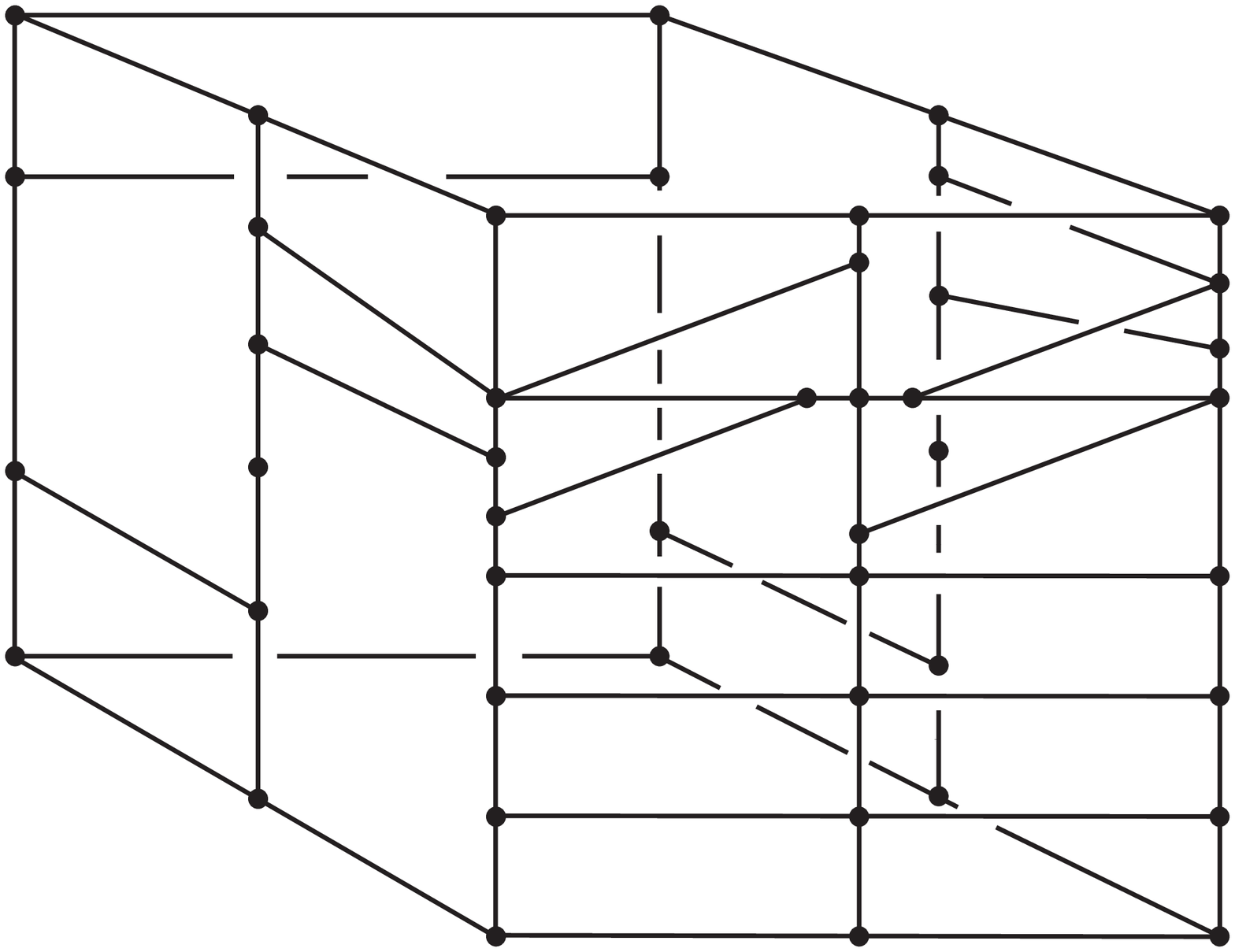}
\\
\ \newline Figure 5. The bimultiplihedron $JJ_{2,3}$ as a subdivision of
$P_{4}$.
\end{center}
\medskip
\end{center}

Unlike the 3-dimensional biassociahedra, certain 2-cells of the 3-dimensional
bimultiplihedra are defined in terms of $\Delta_{P}$. When constructing
$JJ_{2,3},$ for example, we use $\Delta_{P}(P_{3})$ to subdivide the cell
$124|3$, and label the faces in a manner similar to the labeling on $KK_{2,3}$
(see \cite{SU4}, Figure 19). Furthermore, it is convenient to think of
$JJ_{2,3}$ as a subdivision of the cylinder $KK_{2,3}\times I;$ then 2-faces
of $(\mathcal{JJ}_{\infty})_{2,3}$ are represented by a path
$(3,1,0)\rightarrow(s,t,\epsilon)\rightarrow(1,2,1)$ in $\mathbb{N}^{3}$ with
$\epsilon=0,1$. For example, the paths $(3,1,0)\rightarrow(2,2,1)\rightarrow
(1,2,1)$ and $(3,1,0)\rightarrow(2,2,0)\rightarrow(1,2,1)$ represent the
2-faces $e_{1}=1|234$ and $e_{2}=14|23$, respectively. Finally, we remark that
a general codimension 1 face of $(\mathcal{JJ}_{\infty})_{n,m}$ is detected by
the pair of leaf sequences $(\mathbf{x},\mathbf{y})$ and the corresponding path.

While $JJ_{1,n}$ and $JJ_{n,1}$ are combinatorially isomorphic to the
multiplihedron $J_{n}$, their distinct combinatorial representations are
related by the cellular isomorphism $JJ_{1,n}\rightarrow JJ_{n,1}$ induced by
the reversing map $\tau:a_{1}|\cdots|a_{n}\mapsto a_{n}|\cdots|a_{1}$ on
$P_{n}.$ Let $\pi_{1}:P_{n}\rightarrow J_{n}$ and $\pi_{2}:P_{n}\rightarrow
J_{n}$ be the respective cellular projections $P_{n}=|r\mathcal{PP}
_{0,n-1}|\rightarrow JJ_{1,n}$ and $P_{n}=|r\mathcal{PP}_{n-1,0}|\rightarrow
JJ_{n,1},$ $n\geq2;$ then $\pi_{1}=\pi_{2}\circ\tau$.

\begin{example}
\label{mult4}The bijection (\ref{JJ}) is represented on the vertices of
$J_{4}$ in $\pi_{1}(2|134\cup24|13)$ as follows:
\[
\begin{array}
[c]{ccc}
\left[  \theta_{2}^{1}\right]  \left[  \mathbf{1}\text{ }\theta_{2}
^{1}\right]  \left[  \theta_{2}^{1}\text{ }\mathbf{1}\text{ }\mathbf{1}
\right]  \left[  \mathfrak{f}_{1}^{1}\text{ }\mathfrak{f}_{1}^{1}\text{
}\mathfrak{f}_{1}^{1}\text{ }\mathfrak{f}_{1}^{1}\right]  & \sim & \left[
\theta_{2}^{1}\right]  \left[  \theta_{2}^{1}\text{ }\mathbf{1}\right]
\left[  \mathbf{1}\text{ }\mathbf{1}\text{ }\theta_{2}^{1}\right]  \left[
\mathfrak{f}_{1}^{1}\text{ }\mathfrak{f}_{1}^{1}\text{ }\mathfrak{f}_{1}
^{1}\text{ }\mathfrak{f}_{1}^{1}\right] \\
\updownarrow &  & \updownarrow\\
\left(  \left(  2\right)  ,\left(  12\right)  ,\left(  211\right)  ,\left(
\mathring{2}\mathring{1}\mathring{1}\mathring{1}\right)  \right)  & \sim &
\left(  \left(  2\right)  ,\left(  21\right)  ,\left(  112\right)  ,\left(
\mathring{2}\mathring{1}\mathring{1}\mathring{1}\right)  \right) \\
\updownarrow &  & \updownarrow\\
2|1|3|4 & \overset{\pi_{1}}{\sim} & 2|3|1|4\\
&  & \\
\left[  \theta_{2}^{1}\right]  \left[  \mathbf{1}\text{ }\theta_{2}
^{1}\right]  \left[  \mathfrak{f}_{1}^{1}\text{ }\mathfrak{f}_{1}^{1}\text{
}\mathfrak{f}_{1}^{1}\right]  \left[  \theta_{2}^{1}\text{ }\mathbf{1}\text{
}\mathbf{1}\right]  &  & \left[  \theta_{2}^{1}\right]  \left[  \theta_{2}
^{1}\text{ }\mathbf{1}\right]  \left[  \mathfrak{f}_{1}^{1}\text{
}\mathfrak{f}_{1}^{1}\text{ }\mathfrak{f}_{1}^{1}\right]  \left[
\mathbf{1}\text{ }\mathbf{1}\text{{}}\theta_{2}^{1}\right] \\
\updownarrow &  & \updownarrow\\
\left(  \left(  2\right)  ,\left(  12\right)  ,\left(  \mathring{2}
\mathring{1}\mathring{1}\right)  ,\left(  \dot{1}211\right)  \right)  &  &
\left(  \left(  2\right)  ,\left(  21\right)  ,\left(  \mathring{2}
\mathring{1}\mathring{1}\right)  ,\left(  \dot{1}112\right)  \right) \\
\updownarrow &  & \updownarrow\\
2|1|4|3 &  & 2|3|4|1\\
&  & \\
\left[  \theta_{2}^{1}\right]  \left[  \mathfrak{f}_{1}^{1}\text{
}\mathfrak{f}_{1}^{1}\right]  \left[  \mathbf{1}\text{ }\theta_{2}^{1}\right]
\left[  \theta_{2}^{1}\text{ }\mathbf{1}\text{ }\mathbf{1}\right]  & \sim &
\left[  \theta_{2}^{1}\right]  \left[  \mathfrak{f}_{1}^{1}\text{
}\mathfrak{f}_{1}^{1}\right]  \left[  \theta_{2}^{1}\text{ }\mathbf{1}\right]
\left[  \mathbf{1}\text{ }\mathbf{1}\text{ }\theta_{2}^{1}\right] \\
\updownarrow &  & \updownarrow\\
\left(  \left(  2\right)  ,\left(  \mathring{2}\mathring{1}\right)  ,\left(
\dot{1}12\right)  ,\left(  \dot{1}211\right)  \right)  & \sim & \left(
\left(  2\right)  ,\left(  \mathring{2}\mathring{1}\right)  ,\left(  \dot
{1}21\right)  ,\left(  \dot{1}112\right)  \right) \\
\updownarrow &  & \updownarrow\\
2|4|1|3 & \overset{\pi_{1}}{\sim} & 2|4|3|1\\
&  & \\
\left[  \mathfrak{f}_{1}^{1}\right]  \left[  \theta_{2}^{1}\right]  \left[
\mathbf{1}\text{ }\theta_{2}^{1}\right]  \left[  \theta_{2}^{1}\text{
}\mathbf{1}\text{ }\mathbf{1}\right]  & \sim & \left[  \mathfrak{f}_{1}
^{1}\right]  \left[  \theta_{2}^{1}\right]  \left[  \theta_{2}^{1}\text{
}\mathbf{1}\right]  \left[  \mathbf{1}\text{ }\mathbf{1}\text{ }\theta_{2}
^{1}\right] \\
\updownarrow &  & \updownarrow\\
\left(  \left(  \mathring{2}\right)  ,\left(  \dot{1}2\right)  ,\left(
\dot{1}12\right)  ,\left(  \dot{1}211\right)  \right)  & \sim & \left(
\left(  \mathring{2}\right)  ,\left(  \dot{1}2\right)  ,\left(  \dot
{1}21\right)  ,\left(  \dot{1}112\right)  \right) \\
\updownarrow &  & \updownarrow\\
4|2|1|3 & \overset{\pi_{1}}{\sim} & 4|2|3|1.
\end{array}
\]
Note that the 2-cell $24|13$ and the edge $2|13|4$ of $P_{4}$ degenerate under
$\pi_{1},$ whereas the 2-cell $13|24$ and the edge $4|13|2$ of $P_{4}$
degenerate under $\pi_{2}.$ For comparison, the cellular map $\pi
:P_{n}\rightarrow J_{n}$ defined in \cite{SU2} differs from both $\pi_{1}$ and
$\pi_{2}$: The 2-cell $24|13$ and the edge $1|24|3$ of $P_{4}$ degenerate
under $\pi$.
\end{example}

\section{Morphisms Defined}

In this section we define the morphisms of $A_{\infty}$-bialgebras. But before
we begin, we mention three settings in which $A_{\infty}$-bialgebras naturally
appear:\vspace{0.1in}

\noindent(1) Let $X$ be a space. The cobar construction $\Omega S_{\ast}(X)$
on the simplicial singular chain complex of $X$ is a DG bialgebra with
coassociative coproduct \cite{Baues1}, \cite{CM}, \cite{KS1}, but whether or
not $\Omega^{2}S_{\ast}(X)$ admits a coassociative coproduct is unknown.
However, there is an $A_{\infty}$-coalgebra structure on $\Omega^{2}S_{\ast
}(X),$ which is compatible with the product, and $\Omega^{2}S_{\ast}(X)$ is an
$A_{\infty}$-bialgebra.\vspace{0.1in}

\noindent(2) Let $H$ be a graded bialgebra with nontrivial product and
coproduct, and let $\rho:RH\longrightarrow H$ be a (bigraded) multiplicative
resolution. Since $RH$ cannot be free and cofree, it is difficult to introduce
a coassociative coproduct on $RH$ so that $\rho$ is a bialgebra map. However,
there is always an $A_{\infty}$-bialgebra structure on $RH$ such that $\rho$
is a morphism of $A_{\infty}$-bialgebras.\vspace{0.1in}

\noindent(3) If $A$ is an $A_{\infty}$-bialgebra over a field, and
$g:H(A)\rightarrow A$ is a cycle-selecting homomorphism, there is an
$A_{\infty}$-bialgebra structure on $H(A),$ which is unique up to isomorphism,
and a morphism $G:H(A)\Rightarrow A$ of $A_{\infty}$-bialgebras extending $g$
(see Theorem \ref{AAHopf}).\vspace*{0.1in}

Recall the following equivalent definitions of an $A_{\infty}$-bialgebra:

\begin{definition}
\label{hopf2} A graded $R$-module $A$ together with an element
\[
\omega=\{\omega_{m}^{n}\in Hom^{m+n-3}(A^{\otimes m},A^{\otimes n}
)\}_{m,n\geq1}\in U_{A}
\]
is an $A_{\infty}$\textbf{-bialgebra} if either

\begin{enumerate}
\item[\textit{(i)}] $\omega\circledcirc\omega=0$ or

\item[\textit{(ii)}] the map $\theta_{m}^{n}\mapsto\omega_{m}^{n}$ extends to
a map ${\mathcal{H}}_{\infty}\rightarrow U_{A}$ of matrads.
\end{enumerate}
\end{definition}

There is an operation $\ominus:U_{B}\times U_{A,B}\times U_{A}\rightarrow
U_{A,B}$ analogous to $\circledcirc,$ which allows us to define a morphism of
$A_{\infty}$-bialgebras in two equivalent ways. Given DG $R$-modules (DGMs)
$\left(  A,d_{A}\right)  $ and $\left(  B,d_{B}\right)  ,$ let $d_{A}
:TA\rightarrow TA$ and $d_{B}:TB\rightarrow TB$ be the free linear extensions
of $d_{A}$ and $d_{B},$ and let $\nabla$ be the induced Hom differential on
$U_{A,B}$, i.e., for $f\in U_{A,B}$ define $\nabla f=d_{B}\circ f-\left(
-1\right)  ^{\left\vert f\right\vert }f\circ d_{A}.$ Given $\left(
\zeta,f,\eta\right)  \in U_{B}\times U_{A,B}\times U_{A}$ and $m,n\geq1,$
obtain $X^{n,m}+Y^{n,m}\in(U_{A,B})_{n,m}$ by replacing $(\theta
,\mathfrak{f},\theta)$ in the right-hand side of formula (\ref{reldiff}) with
$(\zeta,f,\eta)$. Then define
\[
\ominus(\zeta,f,\eta)=\{X^{n,m}\}_{m,n\geq1}+\{Y^{n,m}\}_{m,n\geq1}-\nabla f.
\]

\begin{definition}
Let $(A,\omega_{A})$ and $(B,\omega_{B})$ be $A_{\infty}$-bialgebras. An
element
\[
G=\{g_{m}^{n}\in Hom^{m+n-2}(A^{\otimes m},B^{\otimes n})\}_{m,n\geq1}\in
U_{A,B}
\]
is an $A_{\infty}$-\textbf{bialgebra morphism from} $A$ \textbf{to} $B$ if either

\begin{enumerate}
\item[\textit{(i)}] $\ominus(\omega_{B},G,\omega_{A})=0$ or

\item[\textit{(ii)}] the map $\mathfrak{f}_{m}^{n}\mapsto g_{m}^{n}$ extends
to a map $\mathcal{J}\mathcal{J}_{\infty}\rightarrow U_{A,B}$ of relative matrads.
\end{enumerate}

\noindent The symbol $G:A\Rightarrow B$ denotes an $A_{\infty}$-bialgebra
morphism $G$ from $A$ to $B.$ An $A_{\infty}$-bialgebra morphism $\Phi
=\{\phi_{m}^{n}\}_{m,n\geq1}:A\Rightarrow B$ is an \textbf{isomorphism} if
$\phi_{1}^{1}:A\rightarrow B{\ }$is an isomorphism of underlying modules.
\end{definition}

\section{Transfer of $A_{\infty}$-Structure}

If $A$ is a free DGM, $B$ is an $A_{\infty}$-coalgebra, and $g:A\rightarrow B$
is a homology isomorphism (weak equivalence) with a right-homotopy inverse,
the Coalgebra Perturbation Lemma transfers the $A_{\infty}$-coalgebra
structure from $B$ to $A$ (see \cite{Hueb-Kade}, \cite{Markl1}).
When $B$ is an $A_{\infty}$-bialgebra, Theorem \ref{transfer} generalizes the
CPL in two directions:

\begin{enumerate}
\item The $A_{\infty}$-bialgebra structure transfers from $B$ to $A.$

\item Neither freeness nor a existence of a right-homotopy inverse is required.
\end{enumerate}

\noindent Note that (2) formulates the transfer of $A_{\infty}
$-algebra structure in maximal generality (see Remark \ref{twoinfty}).

\begin{proposition}
\label{homology-iso}Let $A$ and $B$ be DGMs. If $g:A\rightarrow B$ is a chain
map and $u\in Hom\left(  A^{\otimes m},A^{\otimes n}\right)  $, the induced
map $\tilde{g}:U_{A}\rightarrow U_{A,B}$ defined by $\tilde{g}\left(
u\right)  =g^{\otimes n}u$ is a cochain map. Furthermore, if $g$ is also a
homology isomorphism, $\tilde{g}$ is a cohomology isomorphism if either

\begin{enumerate}
\item[\textit{(i)}] $A$ is free as an $R$-module or

\item[\textit{(ii)}] for each $n\geq1,$ there is a DGM\ $X\left(  n\right)  $
and a splitting $B^{\otimes n}=A^{\otimes n}\oplus X(n)$ as a chain complex
such that $H^{\ast}Hom\left(  A^{\otimes k},X\left(  n\right)  \right)  =0$
for all $k\geq1.$
\end{enumerate}
\end{proposition}

The proof is left to the reader.

\begin{theorem}
[\textbf{The} \textbf{Transfer}]\label{transfer}Let $\left(  A,d_{A}\right)  $
be a DGM, let $(B,d_{B},\omega_{B})$ be an $A_{\infty}$-bialgebra, and let
$g:A\rightarrow B$ be a chain map/homology isomorphism. If $\tilde{g}$ is a
cohomology isomorphism, then

\begin{enumerate}
\item[\textit{(i)}] (Existence) $g$ induces an $A_{\infty}$-bialgebra
structure $\omega_{A}=\{\omega_{A}^{n,m}\}$ on $A$, and extends to a map
$G=\{g_{m}^{n}\mid g_{1}^{1}=g\}:A\Rightarrow B$\ of $A_{\infty}$-bialgebras.

\item[\textit{(ii)}] (Uniqueness) $\left(  \omega_{A},G\right)  $ is unique up
to isomorphism, i.e., if $\left(  \omega_{A},G\right)  $ and $\left(
\bar{\omega}_{A},\bar{G}\right)  $ are induced by chain homotopic maps $g$ and
$\bar{g}$, there is an isomorphism\ of $A_{\infty}$-bialgebras $\Phi:\left(
A,\bar{\omega}_{A}\right)  \Rightarrow\left(  A,\omega_{A}\right)  $ and a
chain homotopy $T:\bar{G}\simeq G\circ\Phi.$
\end{enumerate}
\end{theorem}

\begin{proof}
We obtain the desired structures by simultaneously constructing a map of
matrads $\alpha_{A}:C_{\ast}\left(  KK\right)  \rightarrow U_{A}$ and a map of
relative matrads $\beta:C_{\ast}\left(  JJ\right)  \rightarrow U_{A,B}$.
Thinking of $JJ_{n,m}$ as a subdivision of the cylinder $KK_{n,m}\times I,$
identify the top dimensional cells of $KK_{n,m}$ and $JJ_{n,m}$ with
$\theta_{m}^{n}$ and $\mathfrak{f}_{m}^{n}$, and the faces $KK_{n,m}\times0$
and $KK_{n,m}\times1$ of $JJ_{n,m}$ with $\theta_{m}^{n}\left(  \mathfrak{f}
_{1}^{1}\right)  ^{\otimes m}$ and $\left(  \mathfrak{f}_{1}^{1}\right)
^{\otimes n}\theta_{m}^{n},$ respectively. By hypothesis, there is a map of
matrads $\alpha_{B}:C_{\ast}(KK)\rightarrow U_{B}$ such that $\alpha
_{B}(\theta_{m}^{n})=\omega_{B}^{n,m}.$

To initialize the induction, define $\mathcal{\beta}:C_{\ast}\left(
JJ_{1,1}\right)  \rightarrow Hom^{0}\left(  A,B\right)  $ by $\beta\left(
\mathfrak{f}_{1}^{1}\right)  =g_{1}^{1}=g$, and extend $\mathcal{\beta}$ to
$C_{\ast}\left(  JJ_{1,2}\right)  \rightarrow Hom^{1}\left(  A^{\otimes
2},B\right)  $ and $C_{\ast}\left(  JJ_{2,1}\right)  \rightarrow
Hom^{1}\left(  A,B^{\otimes2}\right)  $ in the following way: On the vertices
$\theta_{2}^{1}\left(  \mathfrak{f}_{1}^{1}\otimes\mathfrak{f}_{1}^{1}\right)
\in JJ_{1,2}$ and $\theta_{1}^{2}\mathfrak{f}_{1}^{1}\in JJ_{2,1}$, define
$\beta\left(  \theta_{2}^{1}\left(  \mathfrak{f}_{1}^{1}\otimes\mathfrak{f}
_{1}^{1}\right)  \right)  =\omega_{B}^{1,2}\left(  g\otimes g\right)  $ and
$\beta\left(  \theta_{1}^{2}\mathfrak{f}_{1}^{1}\right)  =\omega_{B}^{2,1}g.$
Since $\omega_{B}^{1,2}\left(  g\otimes g\right)  $ and $\omega_{B}^{2,1}g$
are $\nabla$-cocycles, and $\tilde{g}_{\ast}$ is an isomorphism, there exist
cocycles $\omega_{A}^{1,2}$ and $\omega_{A}^{2,1}$ in $U_{A}$ such that
\[
\tilde{g}_{\ast}[\omega_{A}^{1,2}]=[\omega_{B}^{1,2}\left(  g\otimes g\right)
]\text{ \ and \ }\tilde{g}_{\ast}[\omega_{A}^{2,1}]=[\omega_{B}^{2,1}g].
\]
Thus $\left[  \omega_{B}^{1,2}\left(  g\otimes g\right)  -g\omega_{A}
^{1,2}\right]  =\left[  \omega_{B}^{2,1}g-\left(  g\otimes g\right)
\omega_{A}^{2,1}\right]  =0,$ and there exist cochains $g_{2}^{1}$ and
$g_{1}^{2}$ in $U_{A,B}$ such that
\[
\nabla g_{2}^{1}=\omega_{B}^{1,2}\left(  g\otimes g\right)  -g\omega_{A}
^{1,2}\text{ \ and \ }\nabla g_{1}^{2}=\omega_{B}^{2,1}g-\left(  g\otimes
g\right)  \omega_{A}^{2,1}.
\]
For $m=1,2$ and $n=3-m,$ define $\alpha_{A}:C_{\ast}\left(  KK_{n,m}\right)
\rightarrow Hom\left(  A^{\otimes m},A^{\otimes n}\right)  $ by $\alpha
_{A}(\theta_{m}^{n})=\omega_{A}^{n,m}$, and define $\beta:C_{\ast}\left(
JJ_{n,m}\right)  \rightarrow Hom\left(  A^{\otimes m},B^{\otimes n}\right)  $
by
\[
\begin{array}
[c]{rllll}
\beta( \mathfrak{f}_{m}^{n}) & = & g_{m}^{n} \vspace{1mm} &  & \\
\beta( \mathfrak{f}_{1}^{1}\, \theta_{2}^{1}) & = & g\, \omega_{A}^{1,2} & (
m=2 ) \vspace{1mm} & \\
\beta( ( \mathfrak{f}_{1}^{1}\otimes\mathfrak{f}_{1}^{1})\, \theta_{1}^{2}) &
= & ( g\otimes g )\, \omega_{A}^{2,1} & (m=1). &
\end{array}
\
\]

Inductively, given $\left(  m,n\right)  ,$ $m+n\geq4,$ assume that for
$i+j<m+n$ there exists a map of matrads $\alpha_{A}:C_{\ast}\left(
KK_{j,i}\right)  \rightarrow Hom\left(  A^{\otimes i},A^{\otimes j}\right)  $
and a map of relative matrads $\beta:C_{\ast}\left(  JJ_{j,i}\right)
\rightarrow Hom\left(  A^{\otimes i},B^{\otimes j}\right)  $ such that
$\alpha_{A}(\theta_{i}^{j})=\omega_{A}^{j,i}$ and $\beta(\mathfrak{f}_{i}
^{j})=g_{i}^{j}$. Thus we are given chain maps $\alpha_{A}:C_{\ast}\left(
\partial KK_{n,m}\right)  \rightarrow Hom\left(  A^{\otimes m},A^{\otimes
n}\right)  $ and $\beta:C_{\ast}\left(  \partial JJ_{n,m}\smallsetminus
\operatorname*{int}KK_{n,m}\times1\right)  \rightarrow Hom\left(  A^{\otimes
m},B^{\otimes n}\right)  ;$ we wish to extend $\alpha_{A}$ to the top cell
$\theta_{m}^{n}$ of $KK_{n,m}$, and $\beta$ to the codimension 1 cell $\left(
\mathfrak{f}_{1}^{1}\right)  ^{\otimes n}\theta_{m}^{n}$ and the top cell
$\mathfrak{f}_{m}^{n}$ of $JJ_{n,m}$. Since $\alpha_{A}$ is a map of matrads,
the components of the cocycle
\[
z=\alpha_{A}\left(  C_{\ast}(\partial KK_{n,m})\right)  \in Hom^{m+n-4}\left(
A^{\otimes m},A^{\otimes n}\right)
\]
are expressed in terms of $\omega_{A}^{j,i}$ with $i+j<m+n;$ similarly, since
$\beta$ is a map of relative matrads, the components of the cochain
\[
\varphi=\beta\left(  C_{\ast}(\partial JJ_{n,m}\smallsetminus
\operatorname*{int}KK_{n,m}\times1)\right)  \in Hom^{m+n-3}\left(  A^{\otimes
m},B^{\otimes n}\right)
\]
are expressed in terms of $\omega_{B},$ $\omega_{A}^{j,i}$ and $g_{i}^{j}$
with $i+j<m+n.$ Clearly $\tilde{g}\left(  z\right)  =\nabla\varphi;$ and
$\left[  z\right]  =\left[  0\right]  $ since $\tilde{g}$ is a homology
isomorphism. Now choose a cochain $b\in Hom^{m+n-3}\left(  A^{\otimes
m},A^{\otimes n}\right)  $ such that $\nabla b=z.$ Then
\[
\nabla\left(  \tilde{g}\left(  b\right)  -\varphi\right)  =\nabla\tilde
{g}\left(  b\right)  -\tilde{g}\left(  z\right)  =0.
\]
Choose a class representative $u\in\tilde{g}_{\ast}^{-1}\left[  \tilde
{g}\left(  b\right)  -\varphi\right]  ,\ $set $\omega_{A}^{n,m}=b-u,$ and
define $\alpha_{A}\left(  \theta_{m}^{n}\right)  =\omega_{A}^{n,m}.$ Then
$\left[  \tilde{g}\left(  \omega_{A}^{n,m}\right)  -\varphi\right]  =\left[
\tilde{g}\left(  b-u\right)  -\varphi\right]  =\left[  \tilde{g}\left(
b\right)  -\varphi\right]  -\left[  \tilde{g}\left(  u\right)  \right]
=\left[  0\right]  .$ Choose a cochain $g_{m}^{n}\in Hom^{m+n-2}\left(
A^{\otimes m},B^{\otimes n}\right)  $ such that
\[
\nabla g_{m}^{n}=g^{\otimes n}\omega_{A}^{n,m}-\varphi,
\]
and define 
$\beta\left(  \mathfrak{f}_{m}^{n}\right)  =g_{m}^{n}.$ To extend
$\beta$ as a map of relative matrads, define
\\
 $\beta  \! \left(\!  \left(
\mathfrak{f}_{1}^{1}\right)  ^{\otimes n}\theta_{m}^{n}\right)\!  =g^{\otimes
n}\omega_{A}^{n,m}.$ Passing to the limit we obtain the desired maps
$\alpha_{A}$ and $\beta.$

Furthermore, if chain maps $\bar{\alpha}_{A}$ and $\bar{\beta}$ are defined in
terms of different choices, beginning with a chain map $\bar{g}$ chain
homotopic to $g,$ let $\bar{\omega}_{A}=\operatorname{Im}\bar{\alpha}_{A}$ and
$\bar{G}=\operatorname{Im}\bar{\beta}.$ There is an inductively defined
isomorphism $\Phi=\sum\phi_{m}^{n}:\left(  A,\bar{\omega}_{A}\right)
\Rightarrow\left(  A,\omega_{A}\right)  $ with $\phi_{1}^{1}=\mathbf{1},$ and
a chain homotopy $T:\tilde{G}\simeq G\circ\Phi.$ To initialize the induction,
set $\phi_{1}^{1}=\mathbf{1}${, and note that}
\[
\nabla{g}_{2}^{1}=g{\omega}_{A}^{1,2}-{\omega}_{B}^{1,2}(g\otimes g)\text{ and
}\nabla\bar{g}_{2}^{1}={\bar{g}\bar{\omega}}_{A}^{1,2}-{\omega}_{B}^{1,2}
(\bar{g}\otimes\bar{g}).
\]
Let $s:\bar{g}\simeq g;$ then $c_{2}^{1}={\omega}_{B}^{1,2}\left(  s\otimes
g+\bar{g}\otimes s\right)  $ satisfies
\[
\nabla c_{2}^{1}={\omega}_{B}^{1,2}(g\otimes g)-{\omega}_{B}^{1,2}(\bar
{g}\otimes\bar{g}).
\]
Hence
\[
\nabla(g_{2}^{1}-\bar{g}_{2}^{1}+c_{2}^{1})=g{\omega}_{A}^{1,2}-\bar{g}
{\bar{\omega}}_{A}^{1,2}
\]
and
\[
\bar{g}({\omega}_{A}^{1,2}-{\bar{\omega}}_{A}^{1,2})=\nabla(g_{2}^{1}-\bar
{g}_{2}^{1}+c_{2}^{1}-s\,\omega_{A}^{1,2}).
\]
Consequently, there is $\phi_{2}^{1}:A^{\otimes2}\rightarrow A$ such that
$\nabla\phi_{2}^{1}={\omega}_{A}^{1,2}-${$\bar{\omega}$}$_{A}^{1,2};$ and, as
above, $\phi_{2}^{1}$ may be chosen so that $\bar{g}\phi_{2}^{1}-(g_{2}
^{1}-\bar{g}_{2}^{1}+c_{2}^{1}-s\,\omega_{A}^{1,2})$ is cohomologous to zero.
Thus there is a component $t_{2}^{1}$ of $T$ such that
\[
\nabla(t_{2}^{1})=\bar{g}\phi_{2}^{1}-(g_{2}^{1}-g_{2}^{1}+c_{2}^{1}
+s\,\omega_{A}^{1,2}).
\]

\end{proof}

We shall refer to the algorithm in the proof of the Transfer Theorem as the
\emph{Transfer Algorithm}. Since $\tilde{g}$ is a homology isomorphism
whenever $A$ is free (cf. Proposition \ref{homology-iso}) we have:

\begin{corollary}
\label{m-bialgebra}Let $\left(  A,d_{A}\right)  $ be a free DGM, let
$(B,d_{B},\omega_{B})$ be an $A_{\infty}$-bialgebra, and let $g:A\rightarrow
B$ be a chain map/homology isomorphism. Then

\begin{enumerate}
\item[\textit{(i)}] (Existence) $g$ induces an $A_{\infty}$-bialgebra
structure $\omega_{A}$ on $A$, and extends to a map $G:A\Rightarrow B$\ of
$A_{\infty}$-bialgebras.

\item[\textit{(ii)}] (Uniqueness) $\left(  \omega_{A},G\right)  $ is unique up
to isomorphism.
\end{enumerate}
\end{corollary}

Given a chain complex $B$ of (not necessarily free) $R$-modules, there is
always a chain complex of free $R$-modules $\left(  A,d_{A}\right)  $, and a
homology isomorphism $g:A\rightarrow B.$ To see this, let $\left(
RH:\cdots\rightarrow R_{1}H\rightarrow R_{0}H\overset{\rho}{\rightarrow
}H,d\right)  $ be a free $R$-module resolution of $H=H_{\ast}\left(  B\right)
.$ Since $R_{0}H$ is projective, there is a cycle-selecting homomorphism
$g_{0}^{\prime}:R_{0}H\rightarrow Z\left(  B\right)  $ lifting $\rho$ through
the projection $Z\left(  B\right)  \rightarrow H$, and extending to a chain map
$g_{0}:\left(  RH,0\right)  \rightarrow\left(  B,d_{B}\right)  .$ If
$RH:0\rightarrow R_{1}H\rightarrow R_{0}H\rightarrow H$ is a short $R$-module
resolution of $H,$ then $g_{0}$ extends to a homology isomorphism $g:\left(
RH,d+h\right)  \rightarrow\left(  B,d_{B}\right)  $ with $\left(
A,d_{A}\right)  =\left(  RH,d\right)  .$ Otherwise, there is a perturbation
$h$ of $d$ such that $g:\left(  RH,d+h\right)  \rightarrow\left(
B,d_{B}\right)  $ is a homology isomorphism with $\left(  A,d_{A}\right)
=\left(  RH,d+h\right)  $ (see \cite{Berikashvili}, \cite{Saneblidze1}). Thus
an $A_{\infty}$-structure on $B$ always transfers to an $A_{\infty}$-structure
on $\left(  RH,d+h\right)  $ via Corollary \ref{m-bialgebra}, and we obtain
our main result concerning the transfer of $A_{\infty}$-structure to homology:

\begin{theorem}
\label{AAHopf} \textit{Let }$B$\textit{\ be an }$A_{\infty}$\textit{-bialgebra
with homology }$H=H_{\ast}\left(  B\right)  $\textit{, let }$\left(
RH,d\right)  $\textit{\ be a free }$R$\textit{-module resolution of }$H,$
\textit{and let }$h$\textit{\ be} \textit{a perturbation of }$d$\textit{\ such
that }$g:\left(  RH,d+h\right)  \rightarrow\left(  B,d_{B}\right)
$\textit{\ is a homology isomorphism. Then}

\begin{enumerate}
\item[\textit{(i)}] (Existence) $g$ induces an $A_{\infty}$-bialgebra
structure $\omega_{RH}$ on $RH$, and extends to a map $G:RH\Rightarrow B$ of
$A_{\infty}$-bialgebras.

\item[\textit{(ii)}] (Uniqueness) $\left(  \omega_{RH},G\right)  $ is unique
up to isomorphism.
\end{enumerate}
\end{theorem}

\begin{remark}
Note that $A_{\infty}$-bialgebra structures induced by the Transfer Algorithm
are isomorphic for all choices of the map $g:\left(  RH,d+h\right)
\rightarrow\left(  B,d_{B}\right)  ,$ and we obtain an isomorphism class of
$A_{\infty}$-bialgebra structures on $RH.$
\end{remark}

\begin{remark}
\label{twoinfty} When $H=H_{\ast}(B)$ is a free module, we recover the
classical results of Kadeishvili \cite{Kadeishvili1}, Markl \cite{Markl1}, and
others, which transfer a DG (co)algebra structure to an $A_{\infty}
$-(co)algebra structure on homology, by setting $RH=H$. Furthermore, any pair
of $A_{\infty}$-(co)algebra structures $\left\{  {\omega}_{H}^{n,1}\right\}
_{n\geq1}$ and $\left\{  {\omega}_{H}^{1,m}\right\}  _{m\geq1}$ on $H$ induced
by the same cycle-selecting map $g:H\rightarrow B$ extend to an $A_{\infty}
$-bialgebra structure $\left(  H,{\omega}_{H}^{n,m}\right)  ,$ by the proof of
Theorem \ref{AAHopf}. For an example of a DGA $B$ whose cohomology $H(B)$ is not free,
and whose DGA structure transfers to an $A_{\infty}$-algebra structure on
$H\left(  B\right)$ via Theorem \ref{AAHopf} along a map $g:H(B) \to B$ having no
right-homotopy inverse, see \cite{Umble2}.
\end{remark}

\section{Applications and Examples}

The applications and examples in this section apply the Transfer Algorithm
given by the proof of Theorem \ref{transfer}. Three kinds of specialized
$A_{\infty}$-bialgebras $\left(  A,\left\{  \omega_{m}^{n}\right\}  \right)  $
are relevant here:

\begin{enumerate}
\item $\omega_{m}^{1}=0$ for $m\geq3$ (the $A_{\infty}$-algebra substructure
is trivial).

\item $\omega_{m}^{n}=0$ for $m,n\geq2$ (all higher order structure is
concentrated in the $A_{\infty}$-algebra and $A_{\infty}$-coalgebra substructures).

\item Conditions (1) and (2) hold simultaneously.
\end{enumerate}

\noindent Of these, $A_{\infty}$-bialgebras of the first and third kind appear
in the applications.

Structure relations defining $A_{\infty}$-bialgebras of the second and third
kind are expressed in terms of the S-U diagonal on associahedra $\Delta_{K}$
\cite{SU2} and have especially nice form. Structure relations of the second
kind were derived in \cite{Umble}. Structure relations in an $A_{\infty}
$-bialgebra $\left(  A,\omega\right)  $ of the third kind with $\omega_{1}
^{1}=0 $, $\mu=\omega_{2}^{1}$ and $\psi^{n}=\omega_{1}^{n}$ are a special
case of those derived in \cite{Umble}, and are given by the formula
\begin{equation}
\left\{  \psi^{n}\mu=\mu^{\otimes n}\Psi^{n}\right\}  _{n\geq2},
\label{multformula}
\end{equation}
where the $n$-ary $A_{\infty}$-coalgebra operation
\[
\Psi^{n}=\left(  \sigma_{n,2}\right)  _{\ast}\iota\left(  \xi\otimes
\xi\right)  \Delta_{K}\left(  e^{n-2}\right)  :A\otimes A\rightarrow\left(
A\otimes A\right)  ^{\otimes n}
\]
is defined in terms of

\begin{itemize}
\item a map $\xi:C_{\ast}(K)\rightarrow Hom(A,TA)$ of operads sending the top
dimensional cell $e^{n-2}\subseteq K_{n}$ to $\psi^{n},$

\item the canonical isomorphism
\[
\iota:Hom\left(  A,A^{\otimes n}\right)  ^{\otimes2}\rightarrow Hom\left(
A^{\otimes2},\left(  A^{\otimes n}\right)  ^{\otimes2}\right)  ,
\]

\item and the induced isomorphism
\[
\left(  \sigma_{n,2}\right)  _{\ast}:Hom\left(  A^{\otimes2},\left(
A^{\otimes n}\right)  ^{\otimes2}\right)  \rightarrow Hom\left(  A^{\otimes
2},\left(  A^{\otimes2}\right)  ^{\otimes n}\right)  .
\]

\end{itemize}

Structure relations defining a morphism $G=\left\{  g^{n}\right\}
:(A,\omega_{A})\Rightarrow(B,\omega_{B})$ between $A_{\infty}$-bialgebras of
the third kind are expressed in terms of the S-U diagonal on multiplihedra
$\Delta_{J}$ \cite{SU2} by the formula
\begin{equation}
\left\{  g^{n}\mu_{A}=\mu_{B}^{\otimes n}\mathbf{g}^{n}\right\}  _{n\geq1},
\label{fmultformula}
\end{equation}
where
\[
\mathbf{g}^{n}=(\sigma_{n,2})_{\ast}\iota(\upsilon\otimes\upsilon)\Delta
_{J}(e^{n-1}):A\otimes A\rightarrow\left(  B\otimes B\right)  ^{\otimes n},
\]
and $\upsilon:C_{\ast}(J)\rightarrow Hom(A,TB)$ is a map of relative
prematrads sending the top dimensional cell $e^{n-1}\subseteq J_{n}$ to
$g^{n}$ (the maps $\left\{  \mathbf{g}^{n}\right\}  $ define the tensor
product morphism $G\otimes G:\left(  A\otimes A,\Psi_{A\otimes A}\right)
\Rightarrow\left(  B\otimes B,\Psi_{B\otimes B}\right)  $).

Given a simply connected topological space $X,$ consider the Moore loop space
$\Omega X$ and the simplicial singular cochain complex $S^{\ast}(\Omega X;R)$.
Under the hypotheses of the Transfer Theorem, the DG bialgebra structure of
$S^{\ast}(\Omega X;R)$ transfers to an $A_{\infty}$-bialgebra structure on
$H^{\ast}(\Omega X;R)$. Our next two theorems apply this principle, and
identify some important $A_{\infty}$-bialgebras of the third kind on loop
space (co)homology.

\begin{theorem}
\label{rational3}If $X$ is simply connected, $H^{\ast}(\Omega X;\mathbb{Q})$
admits an induced $A_{\infty}$-bialgebra structure of the third kind.
\end{theorem}

\begin{proof}
Let $\mathcal{A}_{X}$ be a free DG commutative algebra cochain model for $X$
over $\mathbb{Q}$ (e.g., Sullivan's minimal or Halperin-Stasheff's filtered
model); then $H^{\ast}\left(  \mathcal{A}_{X}\right)  \approx H^{\ast
}(X;\mathbb{Q}).$ The bar construction $\left(  B=B\mathcal{A}_{X}
,d_{B},\Delta_{B}\right)  $ with shuffle product is a cofree DG commutative
Hopf algebra cochain model for $\Omega X,$ and $H=H^{\ast}(B,d_{B})$ is a Hopf
algebra with induced coproduct $\psi^{2}=\omega_{1}^{2}$ and free graded
commutative product $\mu=\omega_{2}^{1}$ (by a theorem of Hopf). Since $H$ is
a free commutative algebra, there is a multiplicative cocycle-selecting map
$g_{1}^{1}:H\rightarrow B.$ Consequently, we may set $\omega_{n}^{1}=0$ for
all $n\geq3$ and $g_{n}^{1}=0$ for all $n\geq2$, and obtain a trivial
$A_{\infty}$-algebra structure $\left(  H,\mu\right)  $ induced by $g_{1}
^{1}.$ There is an induced $A_{\infty}$-coalgebra structure $\left(
H,\psi^{n}\right)  _{n\geq2}$, and an $A_{\infty}$-coalgebra map $G=\left\{
g^{n}\mid g^{1}=g_{1}^{1}\right\}  _{n\geq1}:H\Rightarrow B$ constructed as
follows: For $n\geq2,$ assume $\psi^{n}$ and $g^{n-1}$ have been constructed,
and apply the Transfer Algorithm to obtain candidates $\omega_{1}^{n+1}$ and
$g_{1}^{n}.$ Restrict $\omega_{1}^{n+1}$ to generators, and let $\psi^{n+1}$ be
the free extension of $\omega_{1}^{n+1}$ to all of $H $ using Formula
\ref{multformula}. Similarly, restrict $g_{1}^{n}$ to generators, and let
$g^{n}$ be the free extension of $g_{1}^{n}$ to all of $H$ using Formula
\ref{fmultformula}.

To complete the proof, we show that all other $A_{\infty}$-bialgebra
operations may be trivially chosen. Refer to the Transfer Algorithm, and note
that the Hopf relation $\psi^{2}\mu=\left(  \mu\otimes\mu\right)  \sigma
_{2,2}\left(  \psi^{2}\otimes\psi^{2}\right)  $ implies $\beta(\partial
\mathfrak{f}_{2}^{2})|_{C_{1}(JJ_{2,2}\setminus\operatorname*{int}
(KK_{2,2}\times1))}=0.$ Thus we may choose $\omega_{2}^{2}=g_{2}^{2}=0\ $so
that $\beta(\partial\mathfrak{f}_{2}^{2})=\beta(\mathfrak{f}_{2}^{2})=0.$
Inductively, assume that $\omega_{2}^{n-1}=g_{2}^{n-1}=0$ for $n\geq3.$ Then
$\beta(\partial\mathfrak{f}_{2}^{n})|_{C_{n-1}(JJ_{n,2}\setminus
\operatorname*{int}(KK_{n,2}\times1))}=0$, and we may choose $\omega_{2}^{n}=0$
and $g_{2}^{n}=0 $ so that $\beta(\partial\mathfrak{f}_{2}^{n})=\beta
(\mathfrak{f}_{2}^{n})=0.$ Finally, for $m\geq3$ set $\omega_{m}^{n}=0$ and
$g_{m}^{n}=0. $ Then $\left(  H,\mu,\psi^{n}\right)  _{n\geq2}$ as an
$A_{\infty}$-bialgebra of the third kind with structure relations given by
Formula \ref{multformula}, and $G$ is a map of $A_{\infty}$-bialgebras
satisfying Formula \ref{fmultformula}.
\end{proof}

Note that the components of the $A_{\infty}$-bialgebra map $G$ in the proof of
Theorem \ref{rational3} are exactly the components of a map of underlying
$A_{\infty}$-coalgebras given by the Transfer Algorithm.

Let $R$ be a PID, and let $X$ be a connected space such that $H_{\ast}(X;R)$ is
torsion free. Then the Bott-Samelson Theorem \cite{Bott} asserts that
$H_{\ast}\left(  \Omega\Sigma X;R\right)  $ is isomorphic as an algebra to the
tensor algebra $T^{a}\tilde{H}_{\ast}(X;R)$ generated by the reduced homology
of $X$, and the adjoint $i:X\rightarrow\Omega\Sigma X$ of the identity
$\mathbf{1}:\Sigma X\rightarrow\Sigma X$ induces the canonical inclusion
$i_{\ast}:\tilde{H}_{\ast}\left(  X;R\right)  \hookrightarrow T^{a}\tilde
{H}_{\ast}(X;R)\approx H_{\ast}\left(  \Omega\Sigma X;R\right)  $. Thus if
$\{\psi^{n}\}_{n\geq2}$ is an $A_{\infty}$-coalgebra structure on ${H}_{\ast
}(X;R),$ the tensor algebra $T^{a}\tilde{H}_{\ast}(X;R)$ admits a canonical
$A_{\infty}$-bialgebra structure of the third kind with respect to the free
extension of each $\psi^{n}$ via Formula \ref{multformula}.

Furthermore, the canonical inclusion $t:X\hookrightarrow\Omega\Sigma X$
induces a DG coalgebra map of simplicial singular chains $t_{\#}:S_{\ast
}(X;R)\rightarrow S_{\ast}(\Omega\Sigma X;R)$, which extends to a homology
isomorphism $t_{\#}:T^{a}\tilde{S}_{\ast}(X;R)\approx S_{\ast}(\Omega\Sigma
X;R)$ of DG Hopf algebras. Thus the induced \emph{Bott-Samelson Isomorphism}
$t_{\ast}:T^{a}\tilde{H}_{\ast}\left(  X;R\right)  \approx H_{\ast}
(\Omega\Sigma X;R)$ is an isomorphism of Hopf algebras (\cite{husemoller},
\cite{KS1}), and $T^{a}\tilde{S}_{\ast}(X;R)$ is a free DG Hopf algebra chain
model for $\Omega\Sigma X.$

\begin{theorem}
\label{B-S}Let $R$ be a PID, and let $X$ be a connected space such that
$H_{\ast}(X;R)$ is torsion free.

\begin{enumerate}
\item[\textit{(i)}] Then $T^{a}\tilde{H}_{\ast}(X;R)$ admits an $A_{\infty}$
-bialgebra structure of the third kind, which is trivial if and only if the
$A_{\infty}$-coalgebra structure of $H_{\ast}(X;R)$ is trivial.

\item[\textit{(ii)}] The Bott-Samelson Isomorphism $t_{\ast}:T^{a}\tilde
{H}_{\ast}(X;R)\approx H_{\ast}(\Omega\Sigma X;R)$ extends to an isomorphism
of $A_{\infty}$-bialgebras.
\end{enumerate}
\end{theorem}

\begin{proof}
Since $H_{\ast}(X;R)$ is free as an $R$-module, we may choose a
cycle-selecting map $\bar{g}={\bar{g}}_{1}^{1}:H_{\ast}(X,R)\rightarrow
S_{\ast}(X;R)$ and apply the Transfer Algorithm to obtain an induced
$A_{\infty}$-coalgebra structure $\bar{\omega}=\{{\bar{\omega}}_{1}
^{n}\}_{n\geq2}$ on $H_{\ast}(X,R)$ and a corresponding map of $A_{\infty}
$-coalgebras $\bar{G}=\{{\bar{g}}_{1}^{n}\}_{n\geq1}:H_{\ast}(X;R)\Rightarrow
S_{\ast}(X;R).$ Let $H=T^{a}\tilde{H}_{\ast}(X;R),$ let $B=T^{a}\tilde
{S}_{\ast}(X;R),$ and consider the free (multiplicative) extension $g=T\left(
\bar{g}\right)  :H\rightarrow B.$ As in the proof of Theorem \ref{rational3},
use formulas \ref{multformula} and \ref{fmultformula} to freely extend
$\bar{\omega}$ and $\bar{G}$ to families $\omega=\left\{  \omega_{1}
^{n}\right\}  $ and $G=\{{g}_{1}^{n}\mid{g}_{1}^{1}=g\}_{n\geq1}$ defined on
$H$, and choose all other $A_{\infty}$-bialgebra operations to be zero. Then
$\bar{\omega}$ lifts to an $A_{\infty}$-bialgebra structure $\left(
H,\omega,\mu\right)  $ of the third kind with free product $\mu$, and $\bar
{G}$ lifts to a map $G:H\Rightarrow B$ of $A_{\infty}$-bialgebras.
Furthermore, restricting $\omega$ to the multiplicative generators $H_{\ast
}(X;R)$ recovers the $A_{\infty}$-coalgebra operations on $H_{\ast}(X;R)$.
Thus $A_{\infty}$ -bialgebra structure of $H$ is trivial if and only if the
$A_{\infty}$-coalgebra structure of $H_{\ast}(X;R)$ is trivial. Finally, since
$B$ is a free DG Hopf algebra chain model for $\Omega\Sigma X,$ the
Bott-Samelson Isomorphism $t_{\ast}$ extends to an isomorphism of $A_{\infty}
$-bialgebras, and identifies the $A_{\infty}$-bialgebra structure of $H_{\ast
}(\Omega\Sigma X;R)$ with $\left(  H,\omega,\mu\right)  $.
\end{proof}

It is important to note that prior to this work, all known \emph{rational}
homology invariants of $\Omega\Sigma X$ are trivial for any space $X$.
However, we now have the following:

\begin{corollary}
\label{invariant}A nontrivial $A_{\infty}$-coalgebra structure on $H_{\ast
}(X;\mathbb{Q})$ induces a nontrivial $A_{\infty}$-bialgebra structure on
$H_{\ast}(\Omega\Sigma X;\mathbb{Q}).$ Thus the $A_{\infty}$-bialgebra
structure of $H_{\ast}(\Omega\Sigma X;\mathbb{Q})$ is a nontrivial rational
homology invariant.
\end{corollary}

\begin{proof}
First, $H=H_{\ast}(\Omega\Sigma X;\mathbb{Q})$ admits an induced $A_{\infty}
$-bialgebra structure of the third kind by Theorem \ref{B-S}, which is trivial
if and only if the $A_{\infty}$-coalgebra structure of $H_{\ast}
(X;\mathbb{Q})$ is trivial. Second, the dual version of Theorem
\ref{rational3} imposes an induced $A_{\infty}$-bialgebra structure on $H$
whose $A_{\infty}$-coalgebra substructure is trivial, and whose $A_{\infty}
$-algebra substructure is trivial if and only if the $A_{\infty}$-coalgebra
structure of $H_{\ast}(X;\mathbb{Q})$ is trivial.
\end{proof}

The two $A_{\infty}$-bialgebras identified in the proof of Corollary
\ref{invariant} -- one with trivial $A_{\infty}$-coalgebra substructure and
the other with trivial $A_{\infty}$-algebra substructure -- are in fact
isomorphic, and represent the same isomorphism class of $A_{\infty}$-bialgebra
structures on $H_{\ast}(\Omega\Sigma X;\mathbb{Q})$ (cf. Remark 1). Indeed,
choose a pair of isomorphisms for the two $A_{\infty}$-(co)algebra
substructures (their component in bidegree $(1,1)$ is $\mathbf{1}:H\rightarrow
H$). Since $\omega_{i}^{j}=0$ for $i,j\geq2,$ these isomorphisms clearly
determine an isomorphism of $A_{\infty}$-bialgebras.

Our next example exhibits an $A_{\infty}$-bialgebra of the first but not the
second kind. Given a $1\,$-connected DGA $\left(  A,d_{A}\right)  ,$ the bar
construction of $A$, denoted by $BA,$ is the cofree DGC $T^{c}\left(
\downarrow\overline{A}\right)  $ whose differential $d\ $and coproduct
$\Delta$ are defined as follows:\ Let $\left\lfloor x_{1}|\cdots
|x_{n}\right\rfloor $ denote the element $\left.  \downarrow x_{1}\mid
\cdots\mid\downarrow x_{n}\right.  \in BA$, and let $e$ denote the unit
$\left\lfloor \ \right\rfloor .$ Then
\[
d\left\lfloor x_{1}|\cdots|x_{n}\right\rfloor =\sum_{i=1}^{n}\pm\left\lfloor
x_{1}|\cdots|d_{A}x_{i}|\cdots|x_{n}\right\rfloor +\sum_{i=1}^{n-1}
\pm\left\lfloor x_{1}|\cdots|x_{i}x_{i+1}|\cdots|x_{n}\right\rfloor \text{;}
\]
\[
\Delta\left\lfloor x_{1}|\cdots|x_{n}\right\rfloor =e\otimes\left\lfloor
x_{1}|\cdots|x_{n}\right\rfloor +\left\lfloor x_{1}|\cdots|x_{n}\right\rfloor
\otimes e+\sum_{i=1}^{n-1}\left\lfloor x_{1}|\cdots|x_{i}\right\rfloor
\otimes\left\lfloor x_{i}|\cdots|x_{n}\right\rfloor .
\]
Given an $A_{\infty}$-coalgebra $\left(  C,\Delta^{n}\right)  _{n\geq1},$ the
tilde-cobar construction of $C$, denoted by $\tilde{\Omega}C,$ is the free DGA
$T^{a}\left(  \uparrow\overline{C}\right)  $ with differential $d_{\tilde
{\Omega}A}$ given by extending $\sum_{i\geq1}\Delta^{n}$ as a derivation. Let
$\left\lceil x_{1}|\cdots|x_{n}\right\rceil $ denote $\left.  \uparrow
x_{1}\mid\cdots\mid\uparrow x_{n}\right.  \in\tilde{\Omega}H.$

\begin{example}
\label{hga}Consider the DGA $A=\mathbb{Z}_{2}\left[  a,b\right]  /\left(
a^{4},ab\right)  $ with $\left\vert a\right\vert =3,$ $\left\vert b\right\vert
\!=\!5$ and trivial differential. Define a homotopy Gerstenhaber algebra $($HGA$)$
structure $\{E_{p,q}:A^{\otimes p}\otimes A^{\otimes q}\rightarrow
A\}_{p,q\geq0;\text{ }p+q>0}$ with $E_{p,q}$ acting trivially except
$E_{1,0}=E_{0,1}=\mathbf{1}$ and $E_{1,1}(b;b)=a^{3}$ $($cf. \!\cite{Voronov},\!
\cite{KS1}$)$. Form the tensor coalgebra $BA\otimes BA$ with coproduct
$\psi=\sigma_{2,2}\left(  \Delta\otimes\Delta\right)  ,$ and consider the
induced map
\[
\phi=E_{1,0}^{\prime}+E_{0,1}^{\prime}+E_{1,1}^{\prime}:BA\otimes
BA\rightarrow A
\]
of degree $+1,$ which acts trivially except for $E_{1,0}^{\prime}(\lfloor
x\rfloor\otimes e)=E_{0,1}^{\prime}(e\otimes\lfloor x\rfloor)=x$ for all $x\in
A,$ and $E_{1,1}^{\prime}(\lfloor b\rfloor\otimes\lfloor b\rfloor)=a^{3}.$
Since $E_{p,q}$ is an HGA structure, $\phi$ is a twisting cochain, which lifts
to a chain map of DG coalgebras $\mu:BA\otimes BA\rightarrow BA$ defined by
\[
\mu=\sum_{k=0}^{\infty}\downarrow^{\otimes k+1}\phi^{\otimes k+1}\bar{\psi
}^{\left(  k\right)  },
\]
where $\bar{\psi}^{(0)}=\mathbf{1,}$ $\bar{\psi}^{\left(  k\right)  }=\left(
\bar{\psi}\otimes\mathbf{1}^{\otimes k-1}\right)  \cdots\left(  \bar{\psi
}\otimes\mathbf{1}\right)  \bar{\psi}$ for $k>0,$ and $\bar{\psi}$ is the
reduced coproduct on $BA\otimes BA.$ Then, for example, $\mu\left(
\left\lfloor b\right\rfloor\!\otimes\! \left\lfloor b\right\rfloor \right)
=\left\lfloor a^{3}\right\rfloor $ and $\mu\left(  \left\lfloor b\right\rfloor
\! \otimes\! \left\lfloor a|b\right\rfloor \right)  =\left\lfloor a|a^{3}
\right\rfloor +\left\lfloor b|a|b\right\rfloor .$ It follows that
$(BA,d,\Delta,\mu)$ is a DG\ Hopf algebra. Let $\mu_{H}$ and $\Delta_{H}$ be
the product and coproduct on $H=H^{\ast}\left(  BA\right)  $ induced by $\mu$
and $\Delta;$ then $\left(  H,\Delta_{H},\mu_{H}\right)  $ is a graded
bialgebra. Let $\alpha={\operatorname*{cls}}\left\lfloor a\right\rfloor $ and
$z={\operatorname*{cls}}\left\lfloor a|a^{3}\right\rfloor $ in $H,$ and note
that $\left\lfloor a^{3}\right\rfloor =d\left\lfloor a|a^{2}\right\rfloor .$
Let $g:H\rightarrow BA$ be a cycle-selecting map such that $g\left(
\operatorname*{cls}\left\lfloor x_{1}|\cdots|x_{n}\right\rfloor \right)
=\left\lfloor x_{1}|\cdots|x_{n}\right\rfloor .$ Then
\[
\bar{\Delta}_{H}\left(  z\right)  ={\operatorname*{cls}\bar{\Delta}
}\left\lfloor a|a^{3}\right\rfloor ={\operatorname*{cls}}\left\{  \left\lfloor
a\right\rfloor \otimes\left\lfloor a^{3}\right\rfloor \right\}  =0
\]
so that
\[
\left\{  \Delta g+\left(  g\otimes g\right)  \Delta_{H}\right\}  \left(
z\right)  =\left\lfloor a\right\rfloor \otimes\left\lfloor a^{3}\right\rfloor
.
\]
By the Transfer Theorem, we may choose a map $g^{2}:H\rightarrow BA\otimes BA
$ such that $g^{2}\left(  z\right)  =\left\lfloor a\right\rfloor
\otimes\left\lfloor a|a^{2}\right\rfloor ;$ then
\[
\nabla g^{2}\left(  z\right)  =\left\{  \Delta g+\left(  g\otimes g\right)
\Delta_{H}\right\}  \left(  z\right)  .
\]
Furthermore, note that
\[
\left\{  \left(  g^{2}\otimes g+g\otimes g^{2}\right)  \Delta_{H}+\left(
\Delta\otimes\mathbf{1}+\mathbf{1}\otimes\Delta\right)  g^{2}\right\}  \left(
z\right)  =\left\lfloor a\right\rfloor \otimes\left\lfloor a\right\rfloor
\otimes\left\lfloor a^{2}\right\rfloor .
\]
Since $\left\lfloor a^{2}\right\rfloor =d\left\lfloor a|a\right\rfloor ,$
there is an $A_{\infty}$-coalgebra operation $\Delta_{H}^{3}:H\rightarrow
H^{\otimes3}$, and a map $g^{3}:H\rightarrow\left(  BA\right)  ^{\otimes3}$
satisfying the general relation on $J_{3}$ such that $\Delta_{H}^{3}(z)=0$ and
$g^{3}\left(  z\right)  =\left\lfloor a\right\rfloor \otimes\left\lfloor
a\right\rfloor \otimes\left\lfloor a|a\right\rfloor .$ In fact, we may choose
$\Delta_{H}^{3}$ to be identically zero on $H$ so that
\[
\nabla g^{3}=\left(  \Delta\otimes\mathbf{1}+\mathbf{1}\otimes\Delta\right)
g^{2}+\left(  g^{2}\otimes g+g\otimes g^{2}\right)  \Delta_{H}.
\]

Now the potentially non-vanishing terms in the image of $J_{4}$ are
\[
\left(  g^{3}\otimes g+g^{2}\otimes g^{2}+g\otimes g^{3}\right)  \Delta
_{H}+\left(  \Delta\otimes\mathbf{1}\otimes\mathbf{1}+\mathbf{1}\otimes
\Delta\otimes\mathbf{1}+\mathbf{1}\otimes\mathbf{1}\otimes\Delta\right)
g^{3},
\]
and evaluating at $z$ gives $\left\lfloor a\right\rfloor \otimes\left\lfloor
a\right\rfloor \otimes\left\lfloor a\right\rfloor \otimes\left\lfloor
a\right\rfloor .$ Thus there is an $A_{\infty}$-coalgebra operation
$\Delta_{H}^{4}$ and a map $g^{4}:H\rightarrow\left(  BA\right)  ^{\otimes4}$
satisfying the general relation on $J_{4}$ such that $\Delta_{H}^{4}\left(
z\right)  =\alpha\otimes\alpha\otimes\alpha\otimes\alpha$ and $g^{4}\left(
z\right)  =0.$ Now recall that the induced $A_{\infty}$-coalgebra structure on
$H\otimes H$ is given by
\begin{align*}
\Delta_{H\otimes H} &  =\sigma_{2,2}\left(  \Delta_{H}\otimes\Delta
_{H}\right)  \\
\Delta_{H\otimes H}^{4} &  =\sigma_{4,2}[\Delta_{H}^{4}\otimes(\mathbf{1}
^{\otimes2}\otimes\Delta_{H})(\mathbf{1}\otimes\Delta_{H})\Delta_{H}
+(\Delta_{H}\otimes\mathbf{1}^{\otimes2})(\Delta_{H}\otimes\mathbf{1}
)\Delta_{H}\otimes\Delta_{H}^{4}]\\
&  \vdots
\end{align*}
Let $\beta={\operatorname*{cls}}\left\lfloor b\right\rfloor ,$
$u={\operatorname*{cls}}\left\lfloor a|b\right\rfloor ,$
$v={\operatorname*{cls}}\left\lfloor b|a\right\rfloor ,$ and
$w={\operatorname*{cls}}\left\lfloor b|a|b\right\rfloor $ in $H,$ and consider
the induced map of tilde-cobar constructions
\[
\widetilde{\mu_{H}}=\sum_{n\geq1}\left(  \uparrow\mu_{H}\downarrow\right)
^{\otimes n}:\tilde{\Omega}(H\otimes H)\rightarrow\tilde{\Omega}H.
\]
Then
\[
\widetilde{\mu_{H}}\left\lceil \beta\otimes u\right\rceil =\left\lceil \mu
_{H}\left(  \beta\otimes u\right)  \right\rceil =\left\lceil
{\operatorname*{cls}}\text{ }\mu\left(  \left\lfloor b\right\rfloor
\otimes\left\lfloor a|b\right\rfloor \right)  \right\rceil =\left\lceil
z+w\right\rceil
\]
so that
\[
d_{\tilde{\Omega}H}\widetilde{\mu_{H}}\left\lceil \beta\otimes u\right\rceil
=d_{\tilde{\Omega}H}\left\lceil z+w\right\rceil =\left\lceil \alpha
|\alpha|\alpha|\alpha\right\rceil +\left\lceil \beta|u\right\rceil
+\left\lceil v|\beta\right\rceil .
\]
But on the other hand,
\begin{align*}
d_{\tilde{\Omega}(H\otimes H)}\left\lceil \beta\otimes u\right\rceil  &
=\left\lceil \Delta_{H\otimes H}\left(  \beta\otimes u\right)  \right\rceil \\
&  =\left\lceil e\otimes u\mid\beta\otimes e\right\rceil +\left\lceil
\beta\otimes e\mid e\otimes u\right\rceil \\
&  \hspace*{0.5in}+\left\lceil \beta\otimes\alpha\mid e\otimes\beta
\right\rceil +\left\lceil e\otimes\alpha\mid\beta\otimes\beta\right\rceil
\end{align*}
so that
\[
\widetilde{\mu_{H}}d_{\tilde{\Omega}(H\otimes H)}\left\lceil \beta\otimes
u\right\rceil =\left\lceil \beta|u\right\rceil +\left\lceil v|\beta
\right\rceil .
\]
Although $\widetilde{\mu_{H}}$ fails to be a chain map, the Transfer Theorem
implies there is a chain map $\widetilde{\mu_{H}}^{2}:\tilde{\Omega}(H\otimes
H)\rightarrow\tilde{\Omega}H$ such that $\widetilde{\mu_{H}}^{2}\left\lceil
e\otimes\alpha\mid\beta\otimes\beta\right\rceil =\left\lceil \alpha
|\alpha|\alpha|\alpha\right\rceil ,$ which can be realized by defining
\[
\widetilde{\mu_{H}}^{2}=\sum_{n\geq1}\left(  \uparrow\mu_{H}\downarrow
+\uparrow^{\otimes3}\omega_{2}^{3}\downarrow\right)  ^{\otimes n},
\]
where $\omega_{2}^{3}\left(  \beta\otimes\beta\right)  =\alpha\otimes
\alpha\otimes\alpha.$ Indeed, to see that the required equality holds, note
that $\mu_{H}\left(  \beta\otimes\beta\right)  =0$ since $\left\lfloor
a^{3}\right\rfloor =d\left\lfloor a|a^{2}\right\rfloor .$ Thus there is a map
$g_{2}:H\otimes H\rightarrow BA$ such that $g_{2}\left(  \beta\otimes
\beta\right)  =\left\lfloor a|a^{2}\right\rfloor $, and $\nabla g_{2}=g\mu
_{H}+\mu\left(  g\otimes g\right)  .$ Furthermore, there is the following
general relation on $JJ_{2,2}:$
\begin{align}
\nabla g_{2}^{2} &  =\omega_{BA}^{2,2}\left(  g\otimes g\right)  +\left(
\mu\otimes\mu\right)  \sigma_{2,2}\left(  \Delta g\otimes g^{2}+g^{2}
\otimes\left(  g\otimes g\right)  \Delta_{H}\right)  +g^{2}\mu_{H}
\label{three}\\
&  +\left(  \mu\left(  g\otimes g\right)  \otimes g_{2}+g_{2}\otimes g\mu
_{H}\right)  \sigma_{2,2}\left(  \Delta_{H}\otimes\Delta_{H}\right)  +\Delta
g_{2}+\left(  g\otimes g\right)  \omega_{2}^{2}.\nonumber
\end{align}
The first expression on the right-hand side vanishes since $BA$ has trivial
higher order structure, and the next two expressions vanish since $\mu
_{H}\left(  \beta\otimes\beta\right)  =0$ and $g^{2}(\beta)=0$ ($\beta$ is
primitive). However,
\[
\left\{  \left(  \mu\left(  g\otimes g\right)  \otimes g_{2}+g_{2}\otimes
g\mu_{H}\right)  \sigma_{2,2}\left(  \Delta_{H}\otimes\Delta_{H}\right)
+\Delta g_{2}\right\}  \left(  \beta\otimes\beta\right)
\]
\[
=\bar{\Delta}\,g_{2}(\beta\otimes\beta)=\left\lfloor a\right\rfloor
\otimes\left\lfloor a^{2}\right\rfloor .
\]
Since $d\left\lfloor a|a\right\rfloor =\left\lfloor a^{2}\right\rfloor ,$
there an operation $\omega_{2}^{2}:H^{\otimes2}\rightarrow H^{\otimes2}$, and a
map $g_{2}^{2}:H^{\otimes2}\rightarrow\left(  BA\right)  ^{\otimes2}$
satisfying relation (\ref{three}) such that $\omega_{2}^{2}\left(
\beta\otimes\beta\right)  =0$ and $g_{2}^{2}\left(  \beta\otimes\beta\right)
=\left\lfloor a\right\rfloor \otimes\left\lfloor a|a\right\rfloor $.
Similarly, there is an operation $\omega_{2}^{3}:H^{\otimes2}\rightarrow
H^{\otimes3}$, and a map $g_{2}^{3}:H^{\otimes2}\rightarrow\left(  BA\right)
^{\otimes3}$ satisfying the general relation on $JJ_{3,2}$ such that
$\omega_{2}^{3}\left(  \beta\otimes\beta\right)  =\alpha\!\otimes
\!\alpha\!\otimes\!\alpha$ and $g_{2}^{3}\left(  \beta\otimes\beta\right)
=0.$ Thus $\left(  H,\mu_{H},\Delta_{H},\,\Delta_{H}^{4},\omega_{2}
^{3},...\right)  $ is an $A_{\infty}$-bialgebra of the first kind.
\end{example}

One can think of the algebra $A$ in Example \ref{hga} as the singular
$\mathbb{Z}_{2}$-cohomology algebra of a space $X$ with the Steenrod algebra
$\mathcal{A}_{2}$ acting nontrivially via $Sq_{1}b=a^{3}$ (recall that
$Sq_{1}:H^{n}\left(  X;\mathbb{Z}_{2}\right)  \rightarrow H^{2n-1}\left(
X;\mathbb{Z}_{2}\right)  $ is a homomorphism defined by $Sq_{1}\left[
x\right]  =\left[  x\smile_{1}x\right]  $). Recall that a space $X$ is
$\mathbb{Z}_{2}$-\emph{formal} if there exists a DGA $B$ and cohomology
isomorphisms $C^{\ast}(X;\mathbb{Z}_{2})\leftarrow B\rightarrow H^{\ast
}\left(  X;\mathbb{Z}_{2}\right)  .$ Thus when $X$ is $\mathbb{Z}_{2}$-formal,
$H^{\ast}\left(  BA\right)  \approx H^{\ast}(\Omega X;\mathbb{Z}_{2})$ as
graded coalgebras. Now consider a $\mathbb{Z}_{2}$-formal space $X$ whose
cohomology $H^{\ast}\left(  X;\mathbb{Z}_{2}\right)  $ is generated
multiplicatively by $\left\{  a_{1},\ldots,a_{n+1},b\right\}  ,$ $n\geq2.$
Then Example \ref{hga} suggests the following conditions on $X,$ which if
satisfied, give rise to a nontrivial operation $\omega_{2}^{n}$ with $n\geq2,$
on the loop cohomology $H^{\ast}(\Omega X;\mathbb{Z}_{2})$:

\begin{enumerate}
\item $a_{1}b=0;$

\item $a_{1}\cdots a_{n+1}=0;$

\item $a_{i_{1}}\cdots a_{i_{k}}\neq0$ whenever $k\leq n$ and $i_{p}\neq
i_{q}$ for all $p\neq q;$

\item $Sq_{1}(b)=a_{2}\cdots a_{n+1}$.
\end{enumerate}

\noindent To see this, consider the non-zero classes $\alpha_{i}
=\operatorname*{cls}\left\lfloor a_{i}\right\rfloor ,$ $\beta
=\operatorname*{cls}\left\lfloor b\right\rfloor ,$ $u=\operatorname*{cls}
\left\lfloor a_{1}|b\right\rfloor ,$ $w=\operatorname*{cls}\left\lfloor
b|a_{1}|b\right\rfloor ,$ and $z=\operatorname*{cls}\left\lfloor a_{1}
|a_{2}\cdots a_{n+1}\right\rfloor $ in $H=H^{\ast}\left(  BA\right)  .$
Conditions (2) and (3) give rise to an induced $A_{\infty}$-coalgebra
structure $\left\{  \Delta_{H}^{k}:H\rightarrow H^{\otimes k}\right\}  $ such
that $\Delta_{H} ^{k}(z)=0$ for $3\leq k\leq n$, and $\Delta_{H}^{n+1}
(z)=\alpha_{1} \otimes\cdots\otimes\alpha_{n+1}$ with $g^{k}(z)=\left\lfloor
a_{1}\right\rfloor \otimes\cdots\otimes\left\lfloor a_{k-1}\right\rfloor
\otimes\left\lfloor a_{k}|a_{k+1}\cdots a_{n+1}\right\rfloor $ for $2\leq
k\leq n$ and $g^{n+1}(z)=0.$ Next, condition (4) implies $\beta\smile u=w+z,$
and we can define $\omega_{2}^{k}(\beta\otimes\beta)=0$ for $2\leq k<n$ and
$\omega_{2}^{n}(\beta\otimes\beta)=\alpha_{2}\otimes\cdots\otimes\alpha_{n+1}
$ with $g_{2}^{1}(\beta\otimes\beta)=\left\lfloor a_{2}|a_{3}\cdots
a_{n+1}\right\rfloor ,$ $g_{2}^{k}(\beta\otimes\beta)=\left\lfloor
a_{2}\right\rfloor \otimes\cdots\otimes\left\lfloor a_{k}\right\rfloor
\otimes\left\lfloor a_{k+1}\right.  \left.  | a_{k+2}\cdots a_{n+1}
\right\rfloor $ for $2\leq k\leq n-1$ and $g_{2}^{n}(\beta\otimes\beta)=0.$
Indeed, the Transfer Theorem implies the existence of an $A_{\infty}
$-bialgebra structure in which $\omega_{2}^{n}$ satisfies the required
structure relation on $JJ_{n,2}$. Note that the $\mathbb{Z}_{2} $-formality
assumption is in fact superfluous here, as it is sufficient for $\alpha_{i},$
$\beta,$ and $u$ to be non-zero.

Spaces $X$ with $\mathbb{Z}_{2}$-cohomology satisfying conditions (1)-(4) abound.

\begin{example}
\label{ex7.5}Given an integer $n\geq2,$ choose positive integers $r_{1}
,\ldots,r_{n+1}$ and $m\geq2$ such that $r_{2}+\cdots+r_{n+1}=4m-3.$ Consider
the \textquotedblleft thick bouquet\textquotedblright\ $S^{r_{1}}\veebar
\cdots\veebar S^{r_{n+1}},$ i.e., $S^{r_{1}}\times\cdots\times S^{r_{n+1}}$
with top dimensional cell removed, and generators $\bar{a}_{i}\in H^{r_{i}
}(S^{r_{i}};\mathbb{Z}_{2}).$ Also consider the suspension of complex
projective space $\Sigma\mathbb{C}P^{2m-2}$ with generators $\bar{b}\in
H^{2m-1}(\Sigma\mathbb{C}P^{2m-2};\mathbb{Z}_{2})$ and $Sq_{1}\bar{b}\in
H^{4m-3}(\Sigma\mathbb{C}P^{2m-2};\mathbb{Z}_{2})$. Let $Y_{n}=S^{r_{1}
}\veebar\cdots\veebar S^{r_{n+1}}\vee\Sigma\mathbb{C}P^{2m-2},$ and choose a
map $f:Y_{n}\rightarrow K(\mathbb{Z}_{2},4m-3)$ such that $f^{\ast}
(\iota_{4m-3})=\bar{a}_{2}\cdots\bar{a}_{n+1}+Sq_{1}\bar{b}. $ Finally,
consider the pullback $p:X_{n}\rightarrow Y_{n}$ of the following path
fibration:
\[
\begin{array}
[c]{ccccc}
K\left(  \mathbb{Z}_{2},4m-4\right)  & \longrightarrow & X_{n}\medskip &
\longrightarrow & \mathcal{L}K\left(  \mathbb{Z}_{2},4m-3\right) \\
&  & p\downarrow\medskip\text{ } &  & \downarrow\\
&  & Y_{n} & \overset{f}{\longrightarrow} & K\left(  \mathbb{Z}_{2}
,4m-3\right) \\
&  & \bar{a}_{2}\cdots\bar{a}_{n+1}+Sq_{1}\bar{b} & \underset{f^{\ast
}}{\longleftarrow} & \iota_{4m-3}
\end{array}
\]
Let $a_{i}=p^{\ast}\left(  \bar{a}_{i}\right)  $ and $b=p^{\ast}\left(
\bar{b}\right)  ;$ then $a_{1},\ldots,a_{n+1},b$ are multiplicative generators
of $H^{\ast}(X_{n};\mathbb{Z}_{2})$ satisfying conditions (1) - (4) above. We
remark that one can also obtain a space $X_{2}^{\prime}$ with a non-trivial
$\omega_{2}^{2}$ on its loop cohomology by setting $Y_{2}^{\prime}=\left(
S^{2}\times S^{3}\right)  \vee\Sigma\mathbb{C}P^{2}$ in the construction above
(see \cite{Umble2} for details).
\end{example}

Finally, we note that the cohomology of Eilenberg-MacLane spaces and Lie
groups fail to satisfy all of (1)-(4), and it would not be surprising to find
that the operations $\omega_{2}^{n}$ vanish in their loop cohomologies for all
$n\geq2.$ In the case of Eilenberg-MacLane spaces, recent work of Berciano and
the second author seems to support this conjecture. Indeed, each tensor factor
$A=E\left(  v,2n+1\right)  \otimes\Gamma(w,2np+2)\subset H_{\ast}\left(
K\left(  \mathbb{Z},n\right)  ;\mathbb{Z}_{p}\right)  ,$ $n\geq3, $ and $p$ an
odd prime, is an $A_{\infty}$-bialgebra of the third kind of the form $\left(
A,\Delta^{2},\Delta^{p},\mu\right)  $ (see \cite{BU}).

HGAs with nontrivial actions of the Steenrod algebra $\mathcal{A}_{2}$ were
first considered by the first author in \cite{Saneblidze2} and
\cite{Saneblidze3}. In general, the Steenrod $\smile_{1}$-cochain operation
(together with the other higher cochain operations) induces a nontrivial HGA
structure on $S^{\ast}\left(  X\right)  $. However, the failure of the
differential to be a $\smile_{1}$-derivation prevents an immediate lifting of
the HGA structure to cohomology. Nevertheless, such liftings are possible in
certain situations, as we have seen in Example \ref{hga}. Here is another such example.

\begin{example}
\label{hopfinv}Let $g:S^{2n-2}\rightarrow S^{n}$ be a map of spheres, and let
$Y_{m,n}=S^{m}\times\left(  e^{2n-1}\cup_{g}S^{n}\right)  .$ Let $\ast$ be the
wedge point of $S^{m}\vee S^{n}\subset Y_{m,n},$ let $f:S^{2m-1}\rightarrow
S^{m}\times\ast,$ and let $X_{m,n}=e^{2m}\cup_{f}Y_{m,n}.$ Then $X_{m,n}$ is
$\mathbb{Z}_{2}$-formal for each $m$ and $n$ (by a dimensional argument), and
we may consider $A=H^{\ast}(X_{m,n};\mathbb{Z}_{2})$ and $H=H^{\ast
}(BA)\approx H^{\ast}(\Omega X_{m,n};\mathbb{Z}_{2}).$ Below we prove that:

\begin{enumerate}
\item[\textit{(i)}] The $A_{\infty}$-coalgebra structure of $H$ is nontrivial
if and only if the Hopf invariant $h(f)=1,$ in which case $m=2,4,8 $.

\item[\textit{(ii)}] If $h\left(  f\right)  =1,$ the $A_{\infty}$-coalgebra
structure on $H$ extends to a nontrivial $A_{\infty}$-bialgebra structure on
$H.$ Furthermore, let $a\in A^{n}$ and $c\in A^{2n-1}$ be multiplicative
generators; then there is a perturbed multiplication $\varphi$ on $H$ if and
only if $Sq_{1}a=c,$ in which case $n=3,5,9$; otherwise $\varphi$ is induced
by the shuffle product on $BA$.
\end{enumerate}
\end{example}

\begin{proof}
Suppose $h\left(  f\right)  =1.$ Then $A$ is generated multiplicatively by
$a\in A^{n},$ $b\in A^{m},$ and $c\in A^{2n-1}$ subject to the relations
\[
a^{2}=c^{2}=ac=ab^{2}=b^{2}c=b^{3}=0.
\]
Let $\alpha=\operatorname*{cls}\left\lfloor a\right\rfloor ,$ $\beta
=\operatorname*{cls}\left\lfloor b\right\rfloor ,$ $\gamma=\operatorname*{cls}
\left\lfloor c\right\rfloor ,$ and $z=\operatorname*{cls}\left\lfloor
b^{2}\right\rfloor \in H=H^{\ast}\left(  BA\right)  .$ Given $x_{i}
=\operatorname*{cls}\left\lfloor u_{i}\right\rfloor \in H$ with $u_{i}
u_{i+1}=0,$ let $x_{1}|\cdots|x_{n}=\operatorname*{cls}\left\lfloor
u_{1}|\cdots|u_{n}\right\rfloor .$ Note that $x=\alpha|z=z|\alpha$ and
$y=\gamma|z=z|\gamma.$ Let $\Delta_{H}\ $denote the coproduct in $H$ induced
by the cofree coproduct $\Delta$ in $BA.$ Then $x$ and $y$ are primitive, and
$\Delta_{H}\left(  \alpha|z|\alpha\right)  =e\otimes\alpha|z|\alpha
+x\otimes\alpha+\alpha\otimes x+\alpha|z|\alpha\otimes e.$ Define $g\left(
x\right)  =\left\lfloor a|b^{2}\right\rfloor $ and $g^{2}\left(  x\right)
=\left\lfloor a\right\rfloor \otimes\left\lfloor b|b\right\rfloor ;$ define
$g\left(  \alpha|z|\alpha\right)  =\left\lfloor a|b^{2}|a\right\rfloor $ and
$g^{2}\left(  \alpha|z|\alpha\right)  =\left\lfloor a\right\rfloor
\otimes\left(  \left\lfloor a|b|b\right\rfloor +\left\lfloor
b|a|b\right\rfloor +\left\lfloor b|b|a\right\rfloor \right)  .$ There is an
induced $A_{\infty}$-coalgebra operation $\Delta_{H}^{3}:H\rightarrow
H^{\otimes3},$ which vanishes except on elements of the form $\cdots
|z|\cdots,$ and may be defined on the elements $x,$ $y,$ and $\alpha|z|\alpha$
by
\[
\begin{tabular}
[c]{l}
$\Delta_{H}^{3}(x)=\alpha\otimes\beta\otimes\beta,$\\
$\Delta_{H}^{3}(y)=\gamma\otimes\beta\otimes\beta,$ and\\
$\Delta_{H}^{3}(\alpha|z|\alpha)=\alpha\otimes(\alpha|\beta+\beta
|\alpha)\otimes\beta+\alpha\otimes\beta\otimes(\alpha|\beta+\beta|\alpha).$
\end{tabular}
\ \ \ \
\]
Then $\left\{  \Delta_{H},\Delta_{H}^{3}\right\}  $ defines an $A_{\infty}
$-coalgebra structure on $H$. Furthermore, if $Sq_{1}a=c,$ which can only
occur when $n=3,5,9,$ the induced HGA structure on $A$ is determined by
$Sq_{1}$, and induces a perturbation of the shuffle product $\mu:BA\otimes
BA\rightarrow BA$ with $\mu\left(  \left\lfloor a\right\rfloor \otimes
\left\lfloor a\right\rfloor \right)  =\left\lfloor c\right\rfloor .$ The
product $\mu$ lifts to a perturbed product $\varphi$ on $H$ such that
\[
\varphi(\alpha\otimes\alpha|z)=\alpha|z|\alpha+\gamma|z,
\]
and the $A_{\infty}$-coalgebra structure $\left(  H,\Delta_{H},\Delta_{H}
^{3}\right)  $ extends to an $A_{\infty}$-bialgebra structure $\left(
H,\Delta_{H},\Delta_{H}^{3},\varphi\right)  $ as in Example \ref{hga}. On the
other hand, if $Sq_{1}a=0,$ then is induced by the shuffle product on $BA$ and
$\varphi(\alpha\otimes\alpha|z)=\alpha|z|\alpha.$ Conversely, if $h(f)=0,$
then $b^{2}=0$ so that $\Delta_{H}^{k}=0,\ $for all $k\geq3.$
\end{proof}

We conclude with an investigation of the $A_{\infty}$-bialgebra structure on
the double cobar construction. To this end, we first prove a more general
fact, which follows our next definition:

\begin{definition}
Let $\left(  A,d,\psi,\varphi\right)  $ be a free DG\ bialgebra, i.e., free as
a DGA. An \textbf{acyclic cover of }$A$ is a collection of acyclic DG
submodules
\[
\mathcal{C}\left(  A\right)  =\left\{  C^{a}\subseteq A\mid a\text{ is a
monomial of }A\right\}
\]
such that $\psi\left(  C^{a}\right)  \subseteq C^{a}\otimes C^{a}$ \textit{and
}$\varphi\left(  C^{a}\otimes C^{b}\right)  \subseteq C^{ab}.$
\end{definition}

\begin{proposition}
\label{prop}Let $\left(  A,d,\psi,\varphi\right)  $ be a free DG\ bialgebra
with acyclic cover $\mathcal{C}\left(  A\right)  .$

\begin{enumerate}
\item[\textit{(i)}] Then $\varphi\ $and $\psi$ extend to an $A_{\infty}
$-bialgebra structure of the third kind.

\item[\textit{(ii)}] Let $\left(  A^{\prime},d^{\prime},\psi^{\prime}
,\varphi^{\prime}\right)  $ be a free DG\ bialgebra with acyclic cover
$\mathcal{C}\left(  A^{\prime}\right)  ,$ and let $f:A\rightarrow A^{\prime}$
be a DGA\ map such that $f(C^{a})\subseteq C^{f\left(  a\right)  }$ for all
$C^{a}\in\mathcal{C}\left(  A\right)  .$ Then $f $ extends to a morphism of
$A_{\infty}$-bialgebras.
\end{enumerate}
\end{proposition}

\begin{proof}
Define an $A_{\infty}$-coalgebra structure as follows: Let $\psi^{2}=\psi;$
arbitrarily define $\psi^{3}$ on multiplicative generators, and extend
$\psi^{3}$ to decomposables via
\[
\psi^{3}\mu_{A}=\mu_{A}^{\otimes3}\sigma_{3,2}\left(  \psi^{3}\otimes\psi
^{3}\right)  .
\]
Inductively, if $\left\{  \psi^{i}\right\}  _{i<n}$ have been constructed,
arbitrarily define $\psi^{n}$ on multiplicative generators, and use
(\ref{multformula}) to extend $\psi^{n}$ to decomposables. Since each
$\psi^{n}$ preserves $\mathcal{C}\left(  A\right)  $ by hypothesis, $\left\{
\varphi,\psi^{2},\psi^{3},\ldots\right\}  $ is an $A_{\infty}$-bialgebra
structure of the third kind, as desired. The proof of part (ii) is similar.
\end{proof}

Given a space $X,$ choose a base point $y\in X.$ Let $\operatorname{Sing}
^{2}X$ denote the Eilenberg 2-subcomplex of $\operatorname{Sing}X$, and let
$C_{\ast}(X)=C_{\ast}(\operatorname{Sing}^{2}X)/C_{>0}(\operatorname{Sing}
\,y).$ In \cite{SU2} we constructed an explicit (non-coassociative) coproduct
on the double cobar construction $\Omega^{2}C_{\ast}(X),$ which turns it into
a DG bialgebra. Let $\mathbf{\Omega}{^{2}}$ denote the functor from the
category of (2-reduced) simplicial sets to the category of permutahedral sets
(\cite{SU2}, \cite{KS2}) such that $\Omega^{2}C_{\ast}(X)=C_{\ast
}^{\diamondsuit}(\mathbf{\Omega}{^{2}}\operatorname{Sing}^{2}X),$ where
$C_{\ast}^{\diamondsuit}\left(  Y\right)  \!=\!C_{\ast}\!\left(
\operatorname{Sing}^{M}Y\right)  \!/\!\left\langle \text{degeneracies}
\right\rangle \!,$ where $\operatorname{Sing}^{M}Y$ is the multipermutahedral
singular complex of $Y$ (see Definition 15 in \cite{SU2}; cf. \cite{Baues1},
\cite{CM}). Now consider the monoidal permutahedral set ${\mathbf{\Omega}^{2}
}\operatorname{Sing}^{2}X,$ and let $V_{\ast}$ be its monoidal
(non-degenerate) generators. For each $a\in V_{n},$ let
\[
C^{a}=R\left\langle r\text{-faces of}\text{ }a\mid0\leq r\leq n\right\rangle .
\]
Then $\left\{  C^{a}\right\}  $ is an acyclic cover, and by Proposition
\ref{prop} we immediately have:

\begin{theorem}
\label{double}The DG bialgebra structure on the double cobar construction
$\Omega^{2}C_{\ast}(X)$ extends to an $A_{\infty}$-bialgebra of the third kind.
\end{theorem}

\begin{conjecture}
Given a $2$-connected space $X,$ the chain complex $C_{\ast}^{\diamondsuit
}(\Omega^{2}X)$ admits an $A_{\infty}$-bialgebra structure extending the DG
bialgebra structure constructed in \cite{SU2}. Moreover, there exists a
morphism of $A_{\infty}$-bialgebras
\[
G=\left\{  g_{m}^{n}\right\}  :\Omega^{2}C_{\ast}(X)\Rightarrow C_{\ast
}^{\diamondsuit}(\Omega^{2}X)
\]
such that $g_{1}^{1}$ is a homology isomorphism.\vspace{0.2in}
\end{conjecture}

\end{document}